\documentclass[11pt, oneside,a4paper]{article}   	
\usepackage[margin=1in]{geometry}                		
\usepackage{graphicx}			
\usepackage{amsmath}				
\usepackage{amsmath,amssymb,mathtools,hyperref,bbold}
\usepackage{color}	

\usepackage{subcaption}
\usepackage{tikz,pgfplots}

\newcommand{\red}[1]{\textcolor{red}{#1}}

\newcommand{\modif}[1]{\textcolor{black}{#1}}

\newcommand{\added}[1]{\textcolor{black}{#1}}

\usepackage{amsthm}
\usepackage{cleveref}

\newtheorem{theorem}{Theorem}[section]
\newtheorem{remark}[theorem]{Remark}
\newtheorem{lemma}[theorem]{Lemma}
\newtheorem{proposition}[theorem]{Proposition}

\newtheorem{definition}[theorem]{Definition} 
\newtheorem{assumption}[theorem]{Assumption}
\newtheorem{example}[theorem]{Example}

\newcommand{\Rbb}{\mathbb{R}}
\newcommand{\R}{\mathbb{R}}
\newcommand{\Ebb}{\mathbb{E}}
\newcommand{\Pbb}{\mathbb{P}}

\newcommand{\Nbb}{\mathbb{N}}
\newcommand{\Ac}{\mathcal{A}}
\newcommand{\Ec}{\mathcal{E}}
\newcommand{\Fc}{\mathcal{F}}
\newcommand{\Hc}{V}
\newcommand{\Ic}{\mathcal{I}}
\newcommand{\Lc}{\mathcal{L}}
\newcommand{\Mc}{\mathcal{M}}
\newcommand{\Nc}{\mathcal{N}}
\newcommand{\Pc}{\mathcal{P}}
\newcommand{\Rc}{\mathcal{R}}
\newcommand{\Sc}{\mathcal{S}}
\newcommand{\Tc}{\mathcal{T}}
\newcommand{\Xc}{\mathcal{X}}
\newcommand{\Zc}{\mathcal{Z}}
\newcommand{\rank}{\mathrm{rank}}
\newcommand{\pen}{\mathrm{pen}}
\newcommand{\crit}{\mathrm{crit}}
\newcommand{\Rad}{\mathrm{Rad}}

\usepackage{authblk}

\DeclareMathOperator*{\esssup}{ess\,sup}

\title{Learning with tree tensor networks: complexity estimates and model selection}
\author{Bertrand Michel and 
 Anthony Nouy\thanks{Centrale Nantes, Laboratoire de Math\'ematiques Jean Leray, CNRS UMR 6629, France}}

\begin{document}

\maketitle

\begin{abstract}
In this paper, we propose and analyze a model selection method for tree tensor networks in an empirical risk minimization framework and analyze its performance over a wide range of smoothness classes.  Tree tensor networks, or tree-based tensor formats, are prominent model classes for the approximation of high-dimensional functions in numerical analysis and data science. They correspond to sum-product neural networks with a sparse connectivity associated with a dimension partition tree $T$, widths given by a tuple $r$ of tensor ranks, and multilinear activation functions (or units). The approximation power of these model classes has been proved to be optimal (or near to optimal) for classical smoothness classes. 
However, in an empirical risk minimization framework with a limited number of observations, the dimension tree $T$ and ranks $r$ should be selected carefully to balance estimation and approximation errors. 
In this paper, we propose a complexity-based model selection strategy \`a la  Barron, Birg\'e, Massart. Given a family of model classes associated with different trees, ranks, tensor product feature spaces and sparsity patterns for sparse tensor networks, a model is selected by minimizing a penalized empirical risk, with a penalty depending on the complexity of the model class. After deriving bounds of the metric entropy of tree tensor networks with bounded parameters, we deduce a form  
of the penalty from bounds on suprema of empirical processes. This choice of penalty yields a risk bound for the predictor associated with the selected model. In a least-squares setting, after deriving fast rates of convergence of the risk, we show that the proposed strategy is (near to) minimax adaptive to a wide range of smoothness classes including Sobolev or Besov spaces (with isotropic, anisotropic or mixed dominating smoothness) and analytic functions. We discuss the role of sparsity of the tensor network for obtaining optimal performance in several regimes.
In practice, the amplitude of the penalty is calibrated with a slope heuristics method. Numerical experiments in a least-squares regression setting illustrate the performance of the strategy for the approximation of multivariate functions and univariate functions identified with tensors by tensorization (quantization). 
\end{abstract}

 \section{Introduction}

Typical tasks in statistical learning include the estimation of a regression function or of  posterior probabilities for classification  (supervised learning), or the estimation of the probability distribution of a random variable   from  samples of the distribution (unsupervised learning).  These approximation tasks can be formulated as a minimization problem of a risk functional  $\Rc(f)$ whose minimizer $f^\star$ is the target (or oracle) function, and such that $\Rc(f) - \Rc(f^\star)$ measures some discrepancy between the function $f$ and $f^\star.$ The risk is usually defined as 
$$
\Rc(f) = \Ebb(\gamma(f,Z)),
$$
with $Z = (X,Y)$ for supervised learning or $Z=X$ for unsupervised learning, and where $\gamma$ is a {contrast}  function. 
For supervised learning, the contrast $\gamma$ is usually chosen as $\gamma(f,(x,y)) = \ell(y,f(x)) $ where $\ell(y,f(x))$ measures some discrepancy between $y$ and the prediction $f(x)$ for a given realization $(x,y)$ of $(X,Y)$.
In practice, given i.i.d. realizations $(Z_1,\hdots,Z_n)$ of $Z$, an approximation $\hat f^M_n$ is obtained by the minimization of an  {empirical risk} 
$$
\widehat \Rc_n(f) = \frac{1}{n} \sum_{i=1}^n \gamma(f,Z_i)
$$  
over a subset of functions $M$, also called a  {model class} or  {hypothesis set}. Assuming that the risk admits a minimizer $f^M$ over $M$, 
the error $\Rc(\widehat f^M_n) - \Rc(f^\star)$ can be  decomposed into two contributions: an approximation error $\Rc(f^M) - \Rc(f^\star)$ which quantifies the best we can expect from the model class $M$, and an estimation error $\Rc(\widehat f^M_n) - \Rc(f^M)$ which is due to the use of a limited number of observations. For a given model class, a first problem  
 is to understand how these errors behave under some assumptions on the target function. When considering an increasing  sequence of model classes, the approximation error decreases but the estimation error usually increases. Then strategies are required for the selection of a particular model class. \\\par
 
 In many applications, the target function  $f^\star(x)$ is a function of many variables $x = (x_1,\hdots,x_d)$. For applications in image or signal classification, $x$ may be an image (with $d$ the number of pixels or patches) or a discrete time signal (with $d$ the number of time instants) and $f^\star(x)$ provides a label to a particular input $x$.    For applications in computational science, the target function may be the solution of a high-dimensional partial differential equation, a parameter-dependent equation or a stochastic equation.  In all these applications, when $d$ is large and when the number of observations  is limited, 
 one has to rely on suitable model classes $M$ of moderate complexity that exploit specific  
structures of the target function $f^\star$ and yield an approximation $\widehat f^M_n$ with low approximation and estimation errors. 
Typical examples of model classes include additive functions $f_1(x_1) + \cdots + f_d(x_d)$, sums of multiplicative functions $\sum_{k=1}^m f_1^k(
x_1)\cdots f_d^k(x_d)$, projection pursuit $f_1(w_1^T x) + \cdots + f_m(w_m^T x)$, or feed-forward neural networks $\sigma_L\circ f_L \circ \hdots \circ \sigma_1 \circ f_1(x)$ where the $f_k$ are affine maps and the $\sigma_k$ are given nonlinear functions.
\\\par
In this paper, we consider the class of functions in tree-based tensor format, or tree tensor networks. These model classes are well-known approximation tools in numerical analysis and computational physics and have also been  more recently considered in statistical  learning. They are particular cases of feed-forward neural networks with an architecture given by a dimension partition tree and multilinear activation functions (see \cite{khrulkov2018expressive,cohen2016expressive}). 
 For an overview of these tools, the reader is referred to the monograph \cite{hackbusch2012book} and the surveys  \cite{nouy:2017_morbook,bachmayr2016tensor,Khoromskij:2012fk,cichocki2016tensor1,cichocki2017tensor2}. Some results on the approximation power of tree tensor networks can be found in \cite{Schneider201456,Griebel2019,bachmayr2020compositional} for multivariate functions, or in \cite{Kazeev2015,Kazeev2017,Ali2020partI,Ali2020partII,Ali2021partIII} for tensorized (or quantized) functions. 
\\
 A tree-based tensor format is a set of functions 
$$
M_r^T(\Hc)  = \{ f \in \Hc : \rank_{\alpha}(f) \le r_\alpha , \alpha \in T\},
$$
where $T$ is a dimension partition tree over $\{1,\hdots,d\}$, $r = (r_\alpha) \in \Nbb^{|T|}$ is a tuple of integers and $\Hc = \Hc_1\otimes \hdots \otimes \Hc_d$ is a finite dimensional tensor space of multivariate  functions (e.g., polynomials, splines), that is a tensor product feature space.
 A function $f$ in $M_r^T(\Hc)$ is such that for each $\alpha \in T$, the $\alpha$-rank $\rank_{\alpha}(f)$ of $f$ is bounded by $r_\alpha$. That means that for each $\alpha \in T$, $f$ admits a representation 
$$
f(x) = \sum_{k=1}^{r_\alpha} g_k^\alpha(x_\alpha) h_k^{\alpha^c}(x_{\alpha^c})
$$
for some functions $g_k^\alpha$ and $h_k^{\alpha^c}$ of complementary groups of variables.
\modif{Such a representation can be written using tensor diagram notations as 
$$
f(x) = \begin{array}{c}
  \begin{tikzpicture}[scale=1]
  \tikzstyle{mynode}=[draw,thick,minimum size=.6cm]  
  \tikzstyle{connexion}=[draw,thick]  
  \draw (0,1) node[mynode,label=above:{}] (1) {$g^\alpha$}
   (2,1) node[mynode,label=above:{}] (2) {$h^{\alpha^c}$}
   (0,0) node[label=above:{}] (11) {$x_\alpha$}
   (2,0) node[label=above:{}] (12) {$x_{\alpha^c}$}
  ;
  \path (1) edge node[label = above:$k$] {} (2);
  \path (11) edge  (1);  
  \path (12) edge (2);
  \end{tikzpicture}
\end{array},
$$
where $g^\alpha$ and $h^{\alpha^c}$ are order-two tensors with indices $(k,x_\alpha$) and $(k,x_{\alpha^c})$ respectively, and the edge between the two tensors has to be interpreted as a contraction of the two connected tensors. A function $f$ in $M_r^T(\Hc)$   
admits a parametrization in terms of a collection of low-order tensors $\mathbf{v} = (v^{\alpha})_{\alpha\in T}$ forming a \emph{tree tensor network}.
For instance, for the dimension tree of \Cref{fig:DimensionTree}, the function $f$ admits the representation of \Cref{fig:TN} using tensor diagram notations.
 If the tensors $v^{\alpha}$ are sparse, the tensor network $\mathbf{v} $ is called a \emph{sparse tensor network}. 
By identifying the tensors $v^\alpha$ with multilinear functions with values in $\Rbb^{r_\alpha}$,
the function $f$ also  admits a representation as a composition of  multilinear functions, that   
corresponds to a sum-product feed-forward neural network illustrated on \Cref{fig:neural_network}. }

\modif{Model classes $M_r^T(\Hc)$ associated with different trees (or architecture of the tensor network) are known to capture very different structures of multivariate functions. The choice of a good tree  is then crucial in many applications.
This requires robust strategies that select not only the ranks for a given tree but the tree and the associated ranks. }

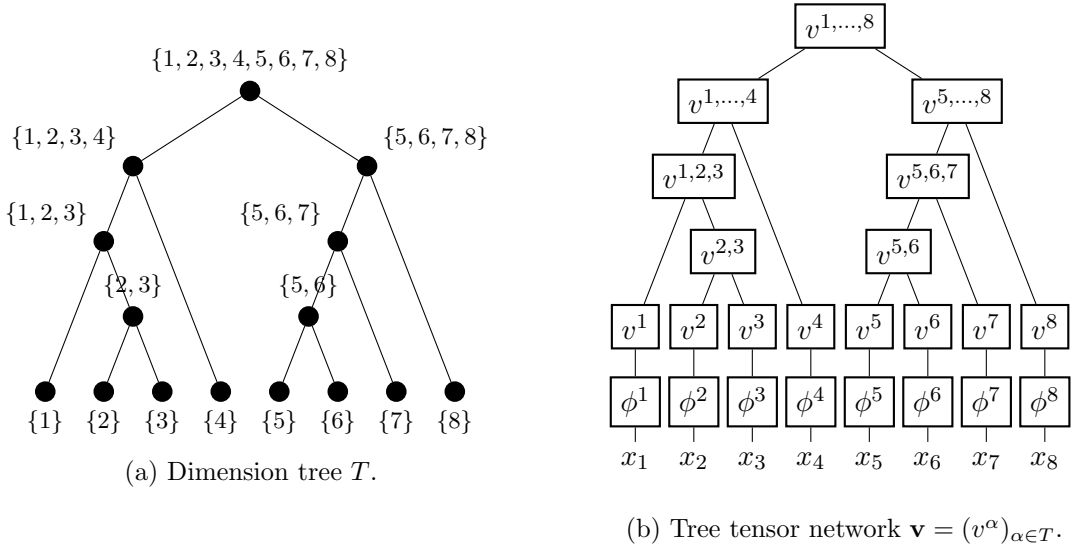
\begin{figure}[h]
\begin{subfigure}[]{.47\textwidth} \footnotesize
\centering 
    \begin{tikzpicture}[xscale=7,yscale=6]  
         \tikzstyle{mynode}=[circle,draw,thick,fill=black,scale=0.8]  
         \tikzstyle{connexion}=[draw,thick,fill=black,scale=0.8] 
         \draw (0.5,0.83333) node[mynode,label=above:{$\{1,2,3,4,5,6,7,8\}$}] (1) {}
         (0.27778,0.66667) node[mynode,label=above left:{$\{1,2,3,4\}$}] (2) {}(0.72222,0.66667) node[mynode,label=above right:{$\{5,6,7,8\}$}] (3) {}(0.22222,0.5) node[mynode,label=above left:{$\{1,2,3\}$}] (4) {}(0.44444,0.16667) node[mynode,label=below:{$\{4\}$}] (5) {}(0.66667,0.5) node[mynode,label=above left:{$\{5,6,7\}$}] (6) {}(0.88889,0.16667) node[mynode,label=below:{$\{8\}$}] (7) {}(0.11111,0.16667) node[mynode,label=below:{$\{1\}$}] (8) {}(0.27778,0.33333) node[mynode,label=above:{$\{2,3\}$}] (9) {}(0.61111,0.33333) node[mynode,label=above:{$\{5,6\}$}] (10) {}(0.77778,0.16667) node[mynode,label=below:{$\{7\}$}] (11) {}(0.22222,0.16667) node[mynode,label=below:{$\{2\}$}] (12) {}(0.33333,0.16667) node[mynode,label=below:{$\{3\}$}] (13) {}(0.55556,0.16667) node[mynode,label=below:{$\{5\}$}] (14) {}(0.66667,0.16667) node[mynode,label=below:{$\{6\}$}] (15) {};\path (1) edge (2);\path (1) edge (3);\path (2) edge (4);\path (2) edge (5);\path (3) edge (6);\path (3) edge (7);\path (4) edge (8);\path (4) edge (9);\path (6) edge (10);\path (6) edge (11);\path (9) edge (12);\path (9) edge (13);\path (10) edge (14);\path (10) edge (15);\end{tikzpicture}
      \caption{Dimension tree $T$.}\label{fig:DimensionTree}
\end{subfigure}
\begin{subfigure}[]{.47\textwidth} \footnotesize
\centering 
  \begin{tikzpicture}[xscale=7,yscale=6]  \tikzstyle{mynode}=[draw,thick]  \normalsize
  \tikzstyle{connexion}=[draw,thick,fill=black]  \draw (0.5,0.83333) node[mynode,label=above:{}] (1) {$v^{1,...,8}$}(0.27778,0.66667) node[mynode,label=above:{}] (2) {$v^{1,...,4}$}(0.72222,0.66667) node[mynode,label=above:{}] (3) {$v^{5,...,8}$}(0.22222,0.5) node[mynode,label=above:{}] (4) {$v^{1,2,3}$}(0.44444,0.16667) node[mynode,label=below:{}] (5) {$v^{4}$}(0.66667,0.5) node[mynode,label=above:{}] (6) {$v^{5,6,7}$}(0.88889,0.16667) node[mynode,label=below:{}] (7) {$v^{8}$}(0.11111,0.16667) node[mynode,label=below:{}] (8) {$v^{1}$}(0.27778,0.33333) node[mynode,label=above:{}] (9) {$v^{2,3}$}(0.61111,0.33333) node[mynode,label=above:{}] (10) {$v^{5,6}$}(0.77778,0.16667) node[mynode,label=below:{}] (11) {$v^{7}$}(0.22222,0.16667) node[mynode,label=below:{}] (12) {$v^{2}$}(0.33333,0.16667) node[mynode,label=below:{}] (13) {$v^{3}$}(0.55556,0.16667) node[mynode,label=below:{}] (14) {$v^{5}$}(0.66667,0.16667) node[mynode,label=below:{}] (15) {$v^{6}$} 
(0.11111,0)  node[mynode,label=below:{}] (101) {$\phi^1$}  
(0.22222,0)  node[mynode,label=below:{}] (102) {$\phi^2$}
(0.33333,0)  node[mynode,label=below:{}] (103) {$\phi^3$}
(0.44444,0)  node[mynode,label=below:{}] (104) {$\phi^4$}
(0.55555,0)  node[mynode,label=below:{}] (105) {$\phi^5$}
(0.66666,0)  node[mynode,label=below:{}] (106) {$\phi^6$}
(0.77777,0)  node[mynode,label=below:{}] (107) {$\phi^7$}
(0.88888,0)  node[mynode,label=below:{}] (108) {$\phi^8$}
(0.11111,-0.13)  node[label=below:{}] (201) {$x_1$}  
(0.22222,-0.13)  node[label=below:{}] (202) {$x_2$}
(0.33333,-0.13)  node[label=below:{}] (203) {$x_3$}
(0.44444,-0.13)  node[label=below:{}] (204) {$x_4$}
(0.55555,-0.13)  node[label=below:{}] (205) {$x_5$}
(0.66666,-0.13)  node[label=below:{}] (206) {$x_6$}
(0.77777,-0.13)  node[label=below:{}] (207) {$x_7$}
(0.88888,-0.13)  node[label=below:{}] (208) {$x_8$}
  ;\path (1) edge (2);\path (1) edge (3);\path (2) edge (4);\path (2) edge (5);\path (3) edge (6);\path (3) edge (7);\path (4) edge (8);\path (4) edge (9);\path (6) edge (10);\path (6) edge (11);\path (9) edge (12);\path (9) edge (13);\path (10) edge (14);\path (10) edge (15);
  \path (8) edge (101);
  \path (12) edge (102);
  \path (13) edge (103);
  \path (5) edge (104);
  \path (14) edge (105);
  \path (15) edge (106);
  \path (11) edge (107);
  \path (7) edge (108);
  \path (101) edge (201);
  \path (102) edge (202);
  \path (103) edge (203);
  \path (104) edge (204);
  \path (105) edge (205);
  \path (106) edge (206);
  \path (107) edge (207);
  \path (108) edge (208);
  \end{tikzpicture}
      \caption{Tree tensor network $\mathbf{v} = (v^\alpha)_{\alpha\in T}$.}\label{fig:TN}
\end{subfigure}
\caption{Dimension tree $T$ over $\{1,\hdots,8\}$ (a) \modif{and corresponding tree tensor network $\mathbf{v} = (v^\alpha)_{\alpha \in T}$ (b). The vector $\phi^\nu(x_\nu) = (\phi^\nu_k(x_\nu))_{1 \le k \le N_\nu} \in \Rbb^{N_\nu}$ represents $N_\nu$ features in the variable $x_\nu$.}}
\end{figure}
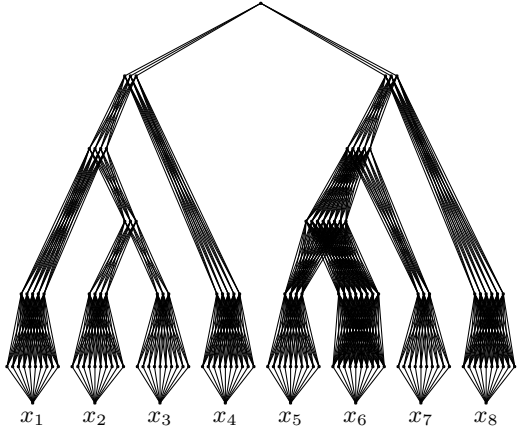
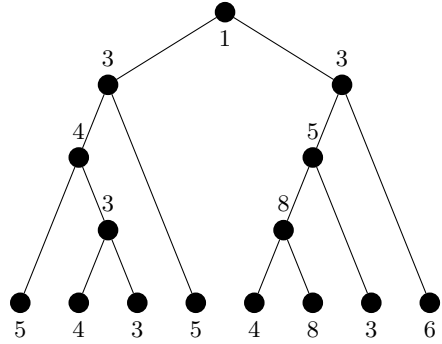
\begin{figure}[h]
\begin{subfigure}[]{.47\textwidth}\footnotesize
\centering 
\include{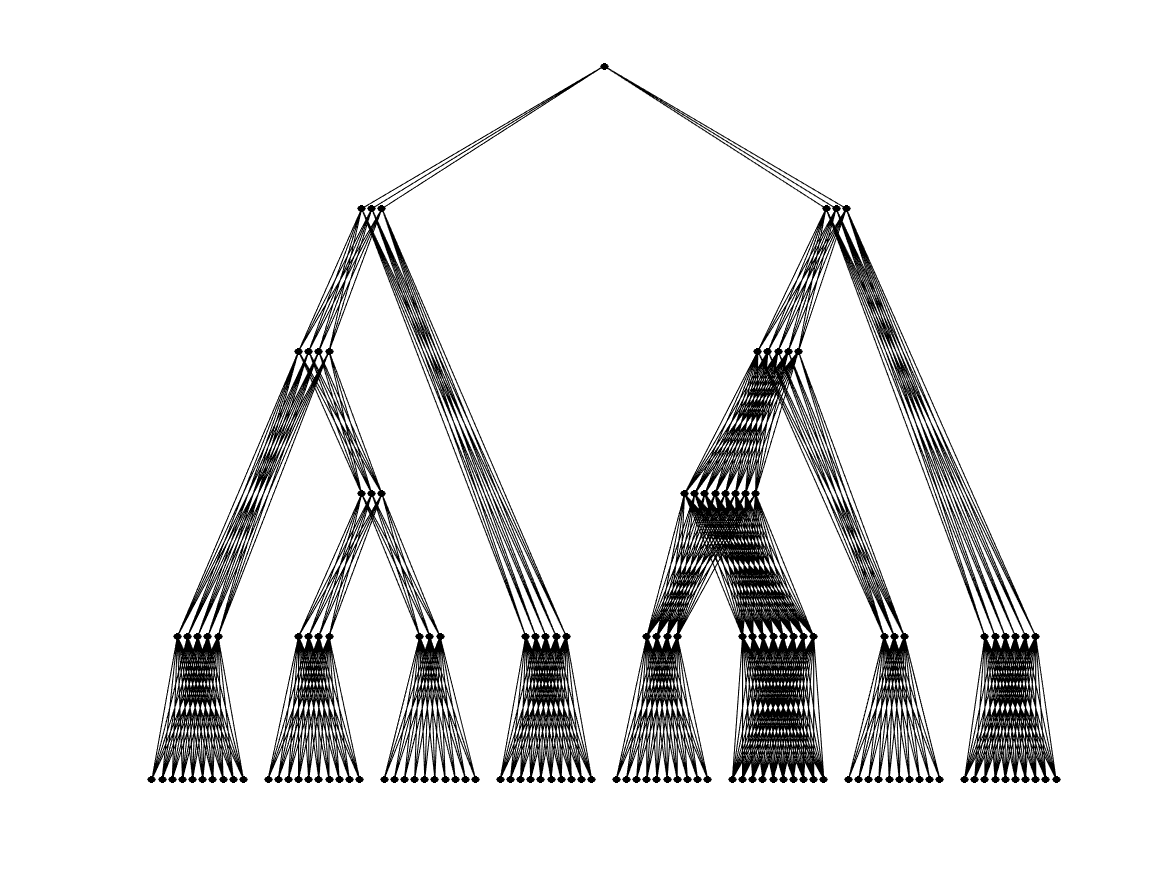}
\caption{Feed-forward neural network.}
\end{subfigure}
\hfill 
\begin{subfigure}[]{.47\textwidth}\footnotesize
\centering 
   \begin{tikzpicture}[xscale=7,yscale=5.8]  \tikzstyle{mynode}=[circle,draw,thick,fill=black,scale=0.8]  
  \tikzstyle{mynodewhite}=[circle,fill=white,scale=0.8]  
  \tikzstyle{connexion}=[draw,thick,fill=black,scale=0.8] 
   \draw (0.5,0.83333) node[mynode,label=below:{1}] (1) {}(0.27778,0.66667) node[mynode,label=above:{3}] (2) {}(0.72222,0.66667) node[mynode,label=above:{3}] (3) {}(0.22222,0.5) node[mynode,label=above:{4}] (4) {}(0.44444,0.16667) node[mynode,label=below:{5}] (5) {}(0.66667,0.5) node[mynode,label=above:{5}] (6) {}(0.88889,0.16667) node[mynode,label=below:{6}] (7) {}(0.11111,0.16667) node[mynode,label=below:{5}] (8) {}(0.27778,0.33333) node[mynode,label=above:{3}] (9) {}(0.61111,0.33333) node[mynode,label=above:{8}] (10) {}(0.77778,0.16667) node[mynode,label=below:{3}] (11) {}(0.22222,0.16667) node[mynode,label=below:{4}] (12) {}(0.33333,0.16667) node[mynode,label=below:{3}] (13) {}(0.55556,0.16667) node[mynode,label=below:{4}] (14) {}(0.66667,0.16667) 
   node[mynode,label=below:{8}] (15) {}
 (0.66667,-0.15)   node[mynodewhite,label=below:{$\,$}] (1000) {}
   ;\path (1) edge (2);\path (1) edge (3);\path (2) edge (4);\path (2) edge (5);\path (3) edge (6);\path (3) edge (7);\path (4) edge (8);\path (4) edge (9);\path (6) edge (10);\path (6) edge (11);\path (9) edge (12);\path (9) edge (13);\path (10) edge (14);\path (10) edge (15);\end{tikzpicture}
\caption{Ranks $r_\alpha$, $\alpha \in T$.}
\end{subfigure}

\caption{A feed-forward sum-product neural network (a) corresponding to the format $M_r^T(\Hc)$ with $N_\nu=10$ features per variable $x_\nu$, the dimension tree $T$  of \Cref{fig:DimensionTree},  and a tuple of ranks $r$ given in figure (b).}  \label{fig:neural_network}
\end{figure}

\par 
The main contribution of the paper is a complexity-based strategy for the selection of a model class  in an empirical risk minimization framework. Given a family $(M_m)_{m\in \Mc}$ of tensor networks (full or sparse) associated with different trees $T_m$, ranks $r_m$, \modif{feature tensor spaces} $\Hc_m$ \modif{(and different sparsity patterns for sparse tensor networks)}, and given the corresponding predictors $\hat f_m$ that minimize the empirical risk, we propose a strategy to select a particular model $\hat m$ with a guaranteed performance. For that purpose, we make use of the model selection approach of Barron, Birg\'e and Massart (see  \cite{Massart:07} for a general introduction to the topic)  where $\hat m $ is obtained by 
minimizing a penalized empirical risk 
$$
\widehat \Rc_n(\hat f_m) + \pen(m)
$$ 
 with a penalty function $\pen(m)$ derived from complexity estimates of the model classes $M_m$, of the form 
 $\pen(m) \sim O(\sqrt{C_m/n})$ (up to logarithmic terms) in a general setting, or of the form  $\pen(m) \sim O(C_m/n)$ (again up to logarithmic terms) in a bounded least-squares setting where faster convergence rates can be obtained. \modif{Here, the complexity $C_m$ is related to the number of parameters in the tensor network (total number of entries of the tensors $v^\alpha$), or the number of non-zero parameters in the tensor network when exploiting sparsity of the tensors $v^\alpha$}.
 
\modif{In a bounded least-squares  setting (for regression or density estimation), using a particular feature space based on tensorization of functions, we find that our strategy is minimax (or near to minimax) adaptive to a wide range of smoothness spaces including Sobolev or Besov spaces with isotropic, anisotropic or mixed dominating smoothness, and analytic function spaces.} 

 In practice, the penalty is taken of the form 
  $\pen(m) = \lambda \sqrt{C_m/n}$ (or $\pen(m) = \lambda C_m/n$ in a bounded \modif{least-squares} setting), where $\lambda$ is calibrated with the slope heuristics method proposed in \cite{birge2007minimal}. 
 The family of models can be generated by adaptive 
  learning algorithms such as the ones proposed in \cite{Grelier:2018, Grelier:2019}. \\\par
Note that our method is a $\ell_0$ type approach. Convex regularization methods would be an interesting alternative route to follow.  A straightforward convexification of tensor formats consists in using the sum of nuclear norms of unfoldings (see e.g.  \cite{signoretto2010nuclear} for Tucker format) but this is known to be far from optimal from a statistical point of view (see \cite{recht2010guaranteed}). 
  A convex regularization method based on the tensor nuclear norm has been proposed for the Tucker format, or shallow tensor network, which comes with theoretical guarantees (see  \cite{yuan2016tensor}). However, there is no straightforward extension of this approach to general tree tensor networks.
\\\par 
The outline of the paper is as follows. In \Cref{sec:tensor-networks}, we describe the model class of tree tensor networks (or tree-based tensor formats) \cite{hackbusch2012book,Falco2018SEMA}. 
In \Cref{sec:entropy}, we provide estimates of the metric and bracketing entropies in $L^p$ spaces for tree tensor networks $M_m$ with bounded parameters. 
\modif{In \Cref{sec:general_risk_bounds}, we  derive bounds for the estimation error in a classical empirical risk minimization framework. These bounds are deduced from concentration inequalities for empirical processes. Then  we present the complexity-based model selection approach and we derive risk bounds for particular choices of penalty in a general setting. We then introduce different collections of tensor networks (full or sparse) corresponding to different adaptive settings, where the feature space and the tree are considered either fixed or free, and we analyze the richness of these collections of models.
 Then in \Cref{sec:leastsquares}, we consider a bounded least-squares setting, for which we derive improved risk bounds with fast rates. That allows us in \Cref{sub:rates}
 to prove that our strategy is (near to) minimax adaptive to a large range of smoothness classes.}
Finally in \Cref{sec:practical} we present the practical aspects of the model selection approach, which includes 
the  slope heuristics method for penalty calibration and  the 
exploration strategies for the generation of a sequence of model classes and associated predictors.  In \Cref{sec:results},  we present some numerical experiments that validate the proposed model selection strategy.

\section{Tree tensor networks}\label{sec:tensor-networks}

We consider functions $f(x) = f(x_1,\hdots,x_d)$ defined on a product set $\Xc = \Xc_1 \times \hdots \times \Xc_d $ and with values in $\Rbb$. Typically, $\Xc_\nu$ is a subset of $\Nbb$ or $\Rbb$ but it could be a set of more general objects (vectors in $\Rbb^{d_\nu}$, sequences, functions, graphs...). 

\subsection{Tensor product feature space}
For each $\nu \in \{1,\hdots,d\}$, we introduce a finite-dimensional space $\Hc_\nu$ of functions defined on $\Xc_\nu$. We let $\{\phi_{i_\nu}^\nu : i_\nu \in I^\nu\}$ be a basis of $\Hc_\nu$, with $I^\nu = \{1,\hdots,N_\nu\}$. The functions $\phi_{i_\nu}^\nu(x_\nu)$ may be polynomials, splines, wavelets, kernel functions, or more general functions that extract $N_\nu$ features from a given input $x_\nu \in \Xc_\nu$.  We let 
 $\phi^\nu : \Xc_\nu \to \Rbb^{N_\nu} $ be the associated \emph{feature map} defined by $\phi^\nu(x_\nu) = (\phi_{1}^\nu(x_\nu),\hdots,\phi_{N_\nu}^\nu(x_\nu))^T \in \Rbb^{N_\nu}$. The functions $ \phi_i(x)= \phi^1_{i_1}(x_1)\hdots \phi^d_{i_d}(x_d)$, $i\in I = I^1\times \hdots \times I^d$,  form a basis of the tensor product space $\Hc = \Hc_1\otimes \hdots \otimes \Hc_d$. A function $f\in \Hc$ admits a representation 
 \begin{equation}\label{representation-a}
 f(x) = \sum_{i\in I} a_i \phi_i(x)  = \sum_{i_1=1}^{N_1} \hdots \sum_{i_d=1}^{N_d} a_{i_1,\hdots,i_d} \phi_{i_1}^1(x_1) \hdots \phi_{i_d}^d(x_d),
 \end{equation}
 where $a \in \Rbb^I =  \Rbb^{N_1\times \hdots \times N_d}$ is an algebraic tensor (or multi-dimensional array) of size $N_1\times \hdots \times N_d$. The map $\phi $ from $\Xc$ to $\Rbb^I$ which associates to $x$ the elementary tensor $\phi(x) = \phi^1(x_1)\otimes \hdots \otimes \phi^d(x_d) \in \Rbb^I$ defines a \emph{tensor product feature map}. 
 
\modif{ \begin{remark}
 In \Cref{sec:tensorization}, we present a particular feature space based on \emph{tensorization}, that yields spaces $V_\nu$ with a tensor product structure and an identification of $f$ with a tensor of order higher than $d$. 
 \end{remark}}

\subsection{Tree-based ranks}
For any $\alpha \subset \{1,\hdots,d\} := D$, and $x\in \Xc$, 
we denote by $x_\alpha = (x_\nu)_{\nu\in \alpha}\in \Xc_\alpha $ the group of variables $\alpha$ that take values in $ \Xc_\alpha =  \times_{\nu \in \alpha} \Xc_\nu$. We let $\alpha^c = D\setminus \alpha$.
\begin{definition}[Ranks of multivariate functions \modif{and minimal subspaces}]\label{def:alpha-rank}
For a non-empty and strict subset $\alpha$ in $D$, the $\alpha$-rank of a function $f: \Xc \to \Rbb$, denoted $\rank_\alpha(f)$, is the minimal integer $r_\alpha$ such that  
\begin{align}
f(x) = \sum_{k=1}^{r_\alpha} g_k^\alpha(x_\alpha) h_k^{\alpha^c}(x_{\alpha^c}) \label{alpha-rank}
\end{align} 
for some functions $g_k^\alpha : \Xc_\alpha \to \Rbb$ and $h_k^{\alpha^c} : \Xc_{\alpha^c} \to \Rbb$.   \modif{The $r_\alpha$-dimensional subspace spanned by the functions $\{g_k^\alpha\}_{k = 1}^{r_\alpha}$ is the $\alpha$-minimal subspace $U_\alpha^{\mathrm{min}}(f)$ of $f$.}
For $\alpha = \emptyset$ or $\alpha= D$, we use the convention $ \rank_\emptyset(f) = 1$ and $\rank_D(f)=1$.
\end{definition}
We let $T$ be a \emph{dimension partition tree} over $D$, with root $D$ and leaves 
$\{\nu\}$, $1\le \nu \le d$. For a node $\alpha\in T$, we denote by $S(\alpha)$ the set of children of $\alpha$. For any node $\alpha$, we have either $S(\alpha) = \emptyset$ (for leaf nodes) or $S(\alpha) \ge 2$ (for interior nodes).  
We denote by $\Lc(T)$ the set of leaves of $T$, and by $\Ic(T) = T\setminus \Lc(T) $ its interior nodes. For an interior node $\alpha \in \Ic(T) $,  $S(\alpha)$ forms a partition of $\alpha$. 
The $T$-rank (or tree-based rank) of a function $f$ is the tuple $\rank_T(f) = (\rank_\alpha(f))_{\alpha \in T}$. 
The number of nodes of a dimension partition tree over $D$ is bounded as $|T|\le 2d-1$ (\modif{with equality for a binary tree}). 

\modif{\begin{remark}[Vector-valued functions]
The above definition and the subsequent notions can be easily extended to the case of vector-valued functions $f$ defined on $\Xc$ with values in $\Rbb^s$ ($s\in \Nbb$), by identifying $f$ with a function  $\tilde f(x_1,\hdots,x_d,i) = f_i(x_1,\hdots,x_d)$ of $d+1$ variables. Most of the results of this paper then easily extends to this setting. 
\end{remark}}

\subsection{Tree tensor networks}
Given a tuple $r = (r_\alpha)_{\alpha\in T} \in \mathbb{N}^{|T|}$ we introduce the model class $M_r^T(\Hc)$ of functions in $\Hc$ with ranks bounded by $r$, 
$$
M_r^T(\Hc) = \{f \in \Hc : \rank_\alpha(f) \le r_\alpha , \alpha \in T\}. 
$$
The set $M_r^T(\Hc)$ is called a tree-based (or hierarchical) tensor format. \modif{A function $f\in M_r^T(\Hc)$ admits a representation \eqref{alpha-rank} for any $\alpha \in T$, with $\{g_k^\alpha\}_{k=1}^{r_\alpha}$ a basis of the minimal subspace $U^{\mathrm{min}}_{\alpha}(f)$.   From the definition of minimal subspaces,} $f$ belongs to the tensor product space  $\bigotimes_{\alpha\in S(D)} U^{\mathrm{min}}_{\alpha}(f)$, and therefore 
 admits the representation (using tensor diagram notations\footnote{We use tensor diagram notations where each node represents a tensor and an edge connecting two nodes represents a contraction of two tensors over one of their modes.}) \begin{equation}
f(x) = \sum_{\substack{1\le k_\alpha \le r_\alpha \\ \text{for }\alpha \in S(D)} }  v^D_{1,(k_\alpha)_{\alpha \in S(D)}}  \prod_{\alpha \in S(D)} g^\alpha_{k_\alpha}(x_\alpha) =
 \begin{array}{c}
 \begin{tikzpicture}[xscale=1,yscale=1]  
 \tikzstyle{mynode}=[draw,thick]  \normalsize
 \tikzstyle{connexion}=[draw,thick,fill=black]  \draw 
  (1,1) node[mynode,label=above:{}] (1) {$v^D$} 
  (0,0) node[mynode,label=above:{}] (2) {$g^{\alpha_1}$}
  (1,0) node[label=above:{}] (3) {$\hdots$}
  (2,0) node[mynode,label=above:{}] (4) {$g^{\alpha_{\vert S(D) \vert}}$}
 (0,-1)  node[label=below:{}] (202) {$x_{\alpha_1}$}  
 (1,-1)  node[label=below:{}] (203) {}
 (2,-1)  node[label=below:{}] (204) {$x_{\alpha_{\vert S(D) \vert}}$}
  ;\path (1) edge (2);
  \path (1) edge (4);\path (2) edge (202);\path (4) edge (204);
   \end{tikzpicture}
   \end{array}
,\label{recursion-root}
\end{equation}
where $v^D$ is a tensor in $\Rbb^{\times_{\alpha \in S(D)} r_\alpha }$ and where $g^\alpha(x_\alpha) = (g^\alpha_k(x_\alpha))_{1\le k \le r_\alpha}$, with functions $g^\alpha_{k_\alpha} \in U^{\mathrm{min}}_{\alpha}(f) \subset \Hc_\alpha = \bigotimes_{\nu\in \alpha} \Hc_\nu$. \modif{From the nestedness property of minimal subspaces   \cite[Proposition 2]{Falco2018SEMA}}, for any interior node 
 $\alpha \in \Ic(T) \setminus \{D\}$, the functions \modif{$g_{k_\alpha}^\alpha \in  \bigotimes_{\beta\in S(\alpha)} U^{\mathrm{min}}_{\beta}(f)$ and therefore}, they admit the representation 
\begin{align}
g_{k_\alpha}^\alpha(x_\alpha) = \sum_{\substack{1\le k_\beta \le r_\beta \\ \text{for } \beta \in S(\alpha)} } v_{k_\alpha,(k_\beta)_{\beta \in S(\alpha)}}^\alpha \prod_{\beta\in S(\alpha)}g_{k_\beta}^\beta(x_\beta) = 
\begin{array}{c}
 \begin{tikzpicture}[xscale=1,yscale=1]  
 \tikzstyle{mynode}=[draw,thick]  \normalsize
 \tikzstyle{connexion}=[draw,thick,fill=black]  \draw 
  (1,2) node[label=above:{}] (0) {$k_\alpha$} 
  (1,1) node[mynode,label=above:{}] (1) {$v^\alpha$} 
  (0,0) node[mynode,label=above:{}] (2) {$g^{\beta_1}$}
  (1,0) node[label=above:{}] (3) {$\hdots$}
  (2,0) node[mynode,label=above:{}] (4) {$g^{\beta_{\vert S(\alpha) \vert}}$}
 (0,-1)  node[label=below:{}] (202) {$x_{\beta_1}$}  
 (1,-1)  node[label=below:{}] (203) {}
 (2,-1)  node[label=below:{}] (204) {$x_{\beta_{\vert S(\alpha) \vert}}$}
  ;\path (1) edge (2);
  \path (1) edge (0);
  \path (1) edge (4);\path (2) edge (202);\path (4) edge (204);
   \end{tikzpicture}
   \end{array}
   ,
\label{recursion-interior}
\end{align}
where $v^\alpha \in \Rbb^{r_\alpha \times (\times_{\beta\in S(\alpha)} r_\beta)}$. For a leaf node $\alpha \in \Lc(T)$, the functions $g_{k_\alpha}^\alpha \in \Hc_\alpha$ admit the representation
\begin{align}
g_{k_\alpha}^\alpha(x_\alpha) =
 \sum_{i_\alpha \in I^\alpha} v^\alpha_{k_\alpha,i_\alpha} \phi^\alpha_{i_\alpha}(x_\alpha)
 =
 \begin{array}{c}
 \begin{tikzpicture}[xscale=1,yscale=1]  
 \tikzstyle{mynode}=[draw,thick]  \normalsize
 \tikzstyle{connexion}=[draw,thick,fill=black]  \draw 
  (1,2) node[label=above:{}] (0) {$k_\alpha$} 
  (1,1) node[mynode,label=above:{}] (1) {$v^\alpha$} 
  (1,0) node[mynode,label=above:{}] (2) {$\phi^{\alpha}$}
   (1,-1)  node[label=below:{}] (202) {$x_{\alpha}$}  
  ;\path (1) edge (2);
  \path (1) edge (0);
  \path (2) edge (202);
   \end{tikzpicture}
   \end{array}
 .
\label{recursion-leaves}
\end{align}
A function $f$ in $M^T_r(\Hc)$ therefore admits an explicit representation 
\begin{equation} \label{representation-C}
f(x) = \sum_{\substack{i_\alpha\in I^\alpha \\ \text{for } \alpha\in \Lc(T) }} 
 \sum_{\substack{1\le k_\alpha \le r_\alpha \\ \text{for } \alpha \in T}}   \prod_{\gamma \in T\setminus \Lc(T)} {v^{\gamma}_{k_\gamma,(k_\beta)_{\beta\in S(\gamma)}}} \prod_{\gamma \in \Lc(T)} {v^{\gamma}_{k_\gamma,i_\gamma}}   \prod_{\gamma \in \Lc(T)} \phi^{\gamma}_{i_\gamma}(x_\gamma)
\end{equation}
where the set of parameters $ \mathbf{v} = (v^\alpha)_{\alpha\in T}$ form a \emph{tree tensor network} (see \Cref{fig:TN} for a representation using tensor diagram notations). The tensor  
$$v^\alpha \in \Rbb^{\{1,\hdots,r_\alpha\} \times I^\alpha} :=  \Rbb^{K^\alpha},$$ 
with  $I^\alpha = \times_{\beta \in S(\alpha)} \{1,\hdots,r_\beta\}$ for $\alpha \in \Ic(T)$ or $I^\alpha = \{1,\hdots,N_\alpha\}$ for $\alpha \in \Lc(T)$. 

\modif{\begin{remark}[Tree tensor networks as compositional functions]
A function associated with a tree tensor network $\mathbf{v} = (v^\alpha)_{\alpha \in T}$  admits a representation as a composition of multilinear functions, by identifying a tensor 
$v^\alpha \in \Rbb^{r_\alpha \times n_1 \times \hdots \times n_a}$ with a multilinear map from $ \Rbb^{n_1} \times \hdots \times \Rbb^{n_a}$ to $\Rbb^{r_\alpha}$.
For example, for the dimension tree of \Cref{fig:DimensionTree}, $f$ admits the representation
\begin{align*}
f(x) = v^{1,...,8}\Big(&v^{1,2,3,4}\big(v^{1,2,3}(v^1(\phi^1(x_1)) , v^{2,3}(v^{2}(\phi^2(x_2)) , v^3(\phi^3(x_3))) ) , v^4(\phi^4(x_4))\big) , \\
& v^{5,6,7,8}\big(v^{5,6,7}(v^{5,6}(v^5(\phi^5(x_5)),v^6(\phi^6(x_6))),v^7(\phi^7(x_7))),v^{8}(\phi^8(x_8))\big)\Big).
\end{align*}
For details, see \Cref{sec:multilinear-composition}.
\end{remark}}

\added{A tensor network $\mathbf{v} = (v^\alpha)_{\alpha \in T}$ is said to be a \emph{sparse tensor network} if the $v^\alpha$ are sparse tensors. For $\Lambda^\alpha \subset K^\alpha$, a tensor $v^\alpha$ is said to be $\Lambda^\alpha$-sparse if $v^\alpha_j=0$ for $j  \in K^\alpha \setminus  \Lambda^\alpha$.  For a given $\Lambda = \times_{\alpha\in T} \Lambda^\alpha$, with $K^\alpha \subset \Lambda^\alpha$, a tensor network $\mathbf{v}$ is said to be $\Lambda$-sparse if the $v^\alpha$ are $\Lambda^\alpha$-sparse for all $\alpha \in T$. 
}

\modif{
\subsection{Parameter space and representation map}
We introduce the product space of parameters  $$\Pc_{\Hc,T,r} := \bigtimes_{\alpha \in T} \Pc^\alpha , \quad \Pc^\alpha := \Rbb^{K^\alpha},$$  and  let $\Rc_{\Hc,T,r}$ be the map which associates to the tensor network $\mathbf{v} \in \Pc_{\Hc,T,r}$ the function $f = \Rc_{\Hc,T,r}(\mathbf{v})$ defined by \eqref{representation-C}, so that 
$$
M^T_r(\Hc) = \{f = \Rc_{\Hc,T,r}(\mathbf{v} ) : \mathbf{v} \in \Pc_{\Hc,T,r}\}.
$$ }
From the representation \eqref{representation-C}, we obtain the following
\begin{lemma}\label{lem:multilinear}
The representation map $\Rc_{r,T,\Hc}$ is a multilinear map from the product space $\Pc_{\Hc,T,r} = \bigtimes_{\alpha \in T} \Pc^\alpha$ to $\Hc$.
\end{lemma}

\added{For $\Lambda^\alpha \subset K^\alpha$, we denote by $\Pc^\alpha_{\Lambda^\alpha}$  the linear subspace of $\Lambda^\alpha$-sparse tensors in $\Pc^\alpha$. Then for 
$\Lambda = \times_{\alpha\in T} \Lambda^\alpha$, we denote by $\Pc_{\Hc,T,r,\Lambda} \subset \Pc_{\Hc,T,r}$ the set of $\Lambda$-sparse tensor networks and 
we introduce the corresponding model class 
$$
M^T_{r,\Lambda}(\Hc) =  \{f = \Rc_{\Hc,T,r}(\mathbf{v}) :\mathbf{v} \in  \Pc_{\Hc,T,r,\Lambda} \}.
$$
}

%

\subsection{Complexity of a tensor network} \label{sec:ReprComp}
When interpreting a tensor network $\mathbf{v} \in \Pc_{\Hc,T,r} = \bigtimes_{\alpha \in T} \Pc^\alpha$ as a neural network, a classical measure of complexity is the number of neurons, which is the sum of ranks $r_\alpha$, $\alpha \in T$. 
From an approximation or statistical perspective, a more natural measure of complexity is the number of parameters (or representation complexity), that is the dimension $\sum_{\alpha\in T} \dim(\Pc^\alpha)$ of the corresponding parameter space $\Pc_{\Hc,T,r}$, or the number of weights of the corresponding  neural network. 
Then the representation complexity of $\mathbf{v} $ is 
\begin{align}
C(T,r,\Hc) := \sum_{\alpha\in T} \vert K^\alpha \vert  =
\sum_{\alpha\in \Ic(T)}  r_\alpha \prod_{\beta\in S(\alpha)} r_\beta  + \sum_{\alpha\in \Lc(T)} r_\alpha N_\alpha .\label{representation-complexity}
\end{align}
\modif{For a sparse tensor network $\mathbf{v} \in \Pc_{\Hc,T,r,\Lambda} = \bigtimes_{\alpha \in T} \Pc^\alpha_{\Lambda^\alpha}$, a natural measure of complexity is given by
\begin{align}
C(T,r,\Hc,\Lambda)= \sum_{\alpha \in T} \vert \Lambda^\alpha \vert,\label{sparse-representation-complexity}
\end{align}
which only counts the number of non-zero parameters (or non-zero weights in the corresponding neural network). 
We note that   $C(T,r,\Hc,\Lambda)  \le  C(T,r,\Hc) .$ }
The different measures of complexity defined above lead to the definition of different approximation tools and corresponding approximation classes, see \cite{Ali2020partI,Ali2020partII,Ali2021partIII} for tensor networks, and \cite{2019arXiv190501208G} for similar results on ReLU or RePU neural networks.

 

\section{Metric entropy of tree tensor networks}\label{sec:entropy}
 \modif{In this section, we provide an estimate of the metric entropy of the set of tree tensor networks (full or sparse) with normalized parameters. This is obtained by showing that tree tensor networks admit a Lipschitz parametrization.}

We assume that the sets $\Xc_\nu$ are equipped with finite measures $\mu_\nu$, for all $\nu\in D = \{1,\hdots,d\}$, and the set $\Xc$ is equipped with the product measure $\mu = \mu_1\otimes \hdots \otimes \mu_d$. For $1\le p\le \infty,$ we consider the space $L^p_\mu(\Xc)$ of real-valued measurable functions defined on $\Xc$, with bounded norm $\Vert \cdot \Vert_{p,\mu}$ defined by 
 $$
\Vert f \Vert_{p,\mu}^p = \int_{\Xc} \vert f(x) \vert_p^p d\mu(x) \quad  \text{for $1\le p<\infty$,} \quad \text{or} \quad   \Vert f \Vert_{\infty,\mu} = \mu\mbox{-}\esssup_{\Xc} | f  |. 
$$
If $\Hc_\nu \subset L^p_{\mu_\nu}(\Xc_\nu)$ for all $\nu\in D$, then $\Hc \subset L^p_\mu(\Xc)$.
 
\subsection{Normalized parametrization}
A function $f\in M^T_r(\Hc)$ admits infinitely many equivalent parametrizations. From the multilinearity of the representation map $\Rc_{\Hc,T,r}$ (see \Cref{lem:multilinear}), it is clear that the model class $M^T_r(\Hc)$ is a cone, i.e. $a M^T_r(\Hc) \subset M^T_r(\Hc)$ for any $a\in \Rbb$. Given some norms $\Vert \cdot \Vert_{\Pc^\alpha}$ on the spaces $\Pc^\alpha = \Rbb^{K^\alpha}$, $\alpha \in T$, and the corresponding product norm on $\Pc_{\Hc , T, r}$ defined by
$$
\Vert (v^\alpha)_{\alpha \in T} \Vert_{\Pc_{\Hc , T, r}} = \max_{\alpha \in T} \Vert v^\alpha \Vert_{\Pc^\alpha},
$$
we have 
$$
M^T_r(\Hc) =\{ a f : a\in \Rbb, f \in M^T_r(\Hc)_1 \},
$$
where  $M^T_r(\Hc)_1$ are elements of $M^T_r(\Hc)$ with bounded parameters, defined by
\begin{equation}
M^T_r(\Hc)_1 = \{  f = \Rc_{\Hc,T,r}( \mathbf{v}) : \mathbf{v} \in \Pc_{\Hc,T,r},  \modif{\Vert \mathbf{v} \Vert_{\Pc_{\Hc,T,r}} \le 1 }  \}. \label{M1}
\end{equation}
\added{The same normalization is used for defining the model class of sparse tensor networks 
$M^T_{r,\Lambda}(\Hc)_1 = M^T_{r,\Lambda}(\Hc) \cap M^T_{r}(\Hc)_1$.
}

\subsection{Continuity of the parametrization}
We here study the continuity of the representation map $\Rc_{\Hc,T,r}$ as a map from $\Pc_{\Hc,T,r} = \times_{\alpha \in T} \Pc^\alpha$ to $\Hc \subset L^p_\mu(\Xc)$. 
%
From the multilinearity of $\Rc_{\Hc,T,r}$ (\Cref{lem:multilinear}), we easily deduce the following property. 
\begin{lemma}\label{lem:map-continuity}
 Assuming $\Hc \subset L^p_\mu(\Xc)$, the multilinear map $\Rc_{\Hc,T,r}$ from $\Pc_{\Hc,T,r}$ to $ \Hc \subset L^p_\mu(\Xc)$ is continuous and such that for all $f = \Rc_{\Hc,T,r}((v^\alpha)_{\alpha \in T})$ in $M^T_r(\Hc)$, 
$$
 \Vert f \Vert_{p,\mu}  \le  L_{p,\mu} \prod_{\alpha \in T} \Vert v^\alpha \Vert_{\Pc^\alpha}
$$
for some constant $L_{p,\mu}<\infty$ independent of $f$ defined by 
\begin{equation}\label{Lp}
L_{p,\mu}  = \sup_{f = \Rc_{\Hc,T,r}((v^\alpha)_{\alpha \in T})} \frac{\Vert f \Vert_{p,\mu}}{\prod_{\alpha \in T} \Vert v^\alpha \Vert_{\Pc^\alpha}} .
\end{equation}
\end{lemma} 
We denote by $B(\Pc^\alpha)$ the unit ball of $\Pc^\alpha$ and by $B(\Pc_{\Hc,T,r})$ the unit ball of $\Pc_{\Hc,T,r}$. The set $M^T_r(\Hc)_1$ defined by \eqref{M1} is such that 
\begin{equation}
M^T_r(\Hc)_1 = \Rc_{\Hc,T,r}( B(\Pc_{\Hc,T,r})). 
\end{equation}
We then deduce that the map $\Rc_{\Hc,T,r}$ is Lipschitz continuous on the set $M^T_r(\Hc)_1$.
\begin{lemma}\label{lem:map-lipschitz}
 Assuming $\Hc \subset L^p_\mu(\Xc)$, for all $f = \Rc_{\Hc,T,r}(\mathbf{v})$ and $ \tilde f = \Rc_{\Hc,T,r}(\tilde{\mathbf{v}}) $ in $M^T_r(\Hc)_1$, 
$$
 \Vert f - \tilde{{ f}} \Vert_{p,\mu} \le L_{p,\mu}  \sum_{\alpha \in T} \Vert v^\alpha - \tilde v^\alpha \Vert_{\Pc^\alpha} \le  L_{p,\mu} |T| \Vert \mathbf{v} - \tilde{\mathbf{v}}\Vert_{\Pc_{\Hc,T,r}}.
$$
\end{lemma}
\begin{proof}
Denoting by $\alpha_1,\hdots,\alpha_{K}$ the elements  of $T$, we have  
 $$f-\tilde{{ f}} = \sum_{k=1}^{K} \Rc_{\Hc,T,r}(\tilde v^{\alpha_1},\cdots, v^{\alpha_k}-\tilde v^{\alpha_k},\cdots,v^{\alpha_K}).$$ Then from \Cref{lem:map-continuity}, we obtain 
\begin{align}
\Vert f - \tilde f \Vert_{p,\mu} \le  L_{p,\mu} \sum_{k=1}^{K}  \Vert v^{\alpha_k} - \tilde v^{\alpha_k} \Vert_{\Pc^{\alpha_k}} \prod_{i<k} \Vert \tilde v^{\alpha_i} \Vert_{\Pc^{\alpha_i}}  \prod_{i>k} \Vert v^{\alpha_i} \Vert_{\Pc^{\alpha_i}},    
\end{align}
and we conclude by noting that $\Vert v^\alpha \Vert_{\Pc^\alpha} \le 1$ and $\Vert \tilde v^\alpha \Vert_{\Pc^\alpha} \le 1$ for all $\alpha \in T.$
\end{proof}

\subsection{Metric entropy}
The metric entropy $H(\epsilon,K,\Vert \cdot \Vert_X)$ of a compact subset $K$ of a normed vector space $(X,\Vert \cdot \Vert_X)$ is defined as $$H(\epsilon,K,\Vert \cdot \Vert_X) = \log N(\epsilon,K,\Vert \cdot \Vert_X),$$ with $N(\epsilon,K,\Vert \cdot \Vert_X)$ the covering number of $K$, which is the minimal number of balls of  radius $\epsilon$ (for $\Vert \cdot \Vert_X$) necessary to cover $K$. We have the following result on the metric entropy of tensor networks with bounded parameters. 
\begin{proposition}\label{prop:metric-entropy}
 Assume that $\Hc \subset L^{p}_\mu(\Xc)$, $1\le p\le \infty$. 
The metric entropy of the model class 
\begin{equation}
M^T_{r}(\Hc)_R = \{a f : a\in \Rbb , \vert a \vert\le R , f\in M^T_r(\Hc)_1\}\label{MR}
\end{equation}
 in $L^p_\mu(\Xc)$  
is  such that
$$
H(\epsilon , M^T_{r}(\Hc)_R , \Vert \cdot\Vert_{p,\mu}) \le C(T,r,\Hc) \log(3\epsilon^{-1} R L_{p,\mu} |T|).
$$
\added{The metric entropy  in $L^p_\mu(\Xc)$  
of the model class of $\Lambda$-sparse tensors  
\begin{equation}M^T_{r,\Lambda}(\Hc)_R =M^T_{r}(\Hc)_R \cap M^T_{r,\Lambda}(\Hc)\label{MRsparse}
\end{equation}
is  such that
$$
H(\epsilon , M^T_{r,\Lambda}(\Hc)_R , \Vert \cdot\Vert_{p,\mu}) \le C(T,r,\Hc,\Lambda) \log(3\epsilon^{-1} R L_{p,\mu} |T|).
$$
}
 
\end{proposition}	
\begin{proof}
The covering number of the unit ball $B(\Pc^\alpha) $ of the $\vert K^\alpha\vert$-dimensional space $\Pc^\alpha$ 
is such that $N(\epsilon,B(\Pc^\alpha),\Vert \cdot \Vert_{\Pc^\alpha}) \le (3\epsilon^{-1})^{\vert K^\alpha\vert}$. Then the unit ball $B(\Pc_{\Hc,T,r})$ of the product space $\Pc_{\Hc,T,r}$ equipped with the product topology has a covering number 
$$N(\epsilon,B(\Pc_{\Hc,T,r}),\Vert \cdot \Vert_{\Pc_{\Hc,T,r}}) \le \prod_{\alpha\in T} N(\epsilon,B(\Pc^\alpha),\Vert \cdot \Vert_{\Pc^\alpha}) \le (3\epsilon^{-1})^{C(T,r,\Hc)}$$ with $C(T,r,\Hc)= \sum_{\alpha \in T} \vert K^\alpha\vert$. From the Lipschitz continuity of $\Rc_{\Hc,T,r}$ on $M^T_{r}(\Hc)_1$ (\Cref{lem:map-lipschitz}), we deduce that 
$N(\epsilon,M^T_{r}(\Hc)_1,\Vert \cdot \Vert_{p,\mu}) \le (3\epsilon^{-1} L_{p,\mu} |T|)^{C(T,r,\Hc)}$, from which we deduce that $N(\epsilon,M^T_{r}(\Hc)_R,\Vert \cdot \Vert_{p,\mu}) \le  (3\epsilon^{-1} R L_{p,\mu} |T|)^{C(T,r,\Hc)}$, which ends the proof of the first statement. \added{For sparse tensors, we first note that 
the unit ball $B(\Pc^\alpha_{\Lambda^\alpha}) $ of the $\vert \Lambda^\alpha\vert$-dimensional space $\Pc^\alpha_{\Lambda^\alpha}$ 
is such that $N(\epsilon,B(\Pc^\alpha_{\Lambda^\alpha}),\Vert \cdot \Vert_{\Pc^\alpha}) \le (3\epsilon^{-1})^{\vert \Lambda^\alpha\vert }$. Then a similar proof yields the desired upper bound with 
$C(T,r,\Hc,\Lambda) = \sum_{\alpha\in T} \vert \Lambda^\alpha\vert$.
}
\end{proof}

\subsection{A particular choice of norms}\label{sec:choice-norms}
 Assume that $\Hc \subset L^p_\mu(\Xc)$. 
The continuity constant $L_{p,\mu}$ of the map $\Rc_{\Hc,T,r}$  defined by \eqref{Lp} depends on $p$, $\mu$, the norms on parameter spaces $\Pc^\alpha$ and the chosen basis for $\Hc$.
We here introduce a particular choice of norms and basis functions which allows to bound the continuity constant $L_{p,\mu}$.
 For any interior node $\alpha \in \Ic(T) $, we introduce a norm $\Vert \cdot \Vert_{\Pc^\alpha}$ over the space $\Pc^\alpha$ 
defined by 
$$
\Vert v^\alpha \Vert_{\Pc^\alpha} = \max_{ (z_\beta)_{\beta\in S(\alpha)} \in \bigtimes_{\beta\in S(\alpha) } \Rbb^{r_\beta}} \frac{\Vert  v^\alpha((z_\beta)_{\beta\in S(\alpha)})\Vert_p}{  \prod_{\beta\in S(\alpha)} \Vert z_\beta \Vert_{p}},
$$
\modif{where the tensor $v^\alpha \in \Rbb^{r_\alpha \times (\times_{\beta \in S(\alpha)} r_\beta)}$ is identified with a multilinear map from $ \bigtimes_{\beta \in S(\alpha)} \Rbb^{r_\beta} $ to $ \Rbb^{r_\alpha}$, and where $\Vert \cdot \Vert_p$ refers to the vector $\ell^p$-norm (for more details, see \Cref{sec:multilinear-composition}).}
For a leaf node $\alpha \in \Lc(T)$, we introduce  a norm $\Vert \cdot \Vert_{\Pc^\alpha}$ over 
the space $\Pc^\alpha$ 
defined by
\begin{align}
\Vert v^\alpha \Vert_{\Pc^\alpha} = \max_{ z_\alpha \in \Rbb^{N_\alpha }}   \frac{\Vert  v^\alpha(z_\alpha) \Vert_p}{  \Vert z_\alpha \Vert_{p}}, \label{choice-norm-Falpha}
\end{align}  
\modif{where the order-two tensor $v^\alpha \in \Rbb^{N_\alpha \times r_\alpha}$ is identified with a linear map from $\Rbb^{N_\alpha} $ to $ \Rbb^{r_\alpha}$. This corresponds to the matrix $p$-norm of $v^\alpha$.}
 We assume that for any $\nu \in D$, the feature map $\phi^\nu : \Xc_\nu \to \Rbb^{N_\nu}$ is such  that 
 $\Vert  \phi^\nu \Vert_{p,\mu} = 1$. 
 For $p=\infty$, that means that 
 basis functions $\phi^\nu_{i}(x_\nu)$ have a unit norm in $L^{\infty}_{\mu_\nu}(\Xc_\nu)$. For $p<\infty$, that means that 
 $\sum_{i=1}^{N_\nu} \Vert \phi^\nu_i \Vert_{ p,\mu }^p =1$, which can be obtained by rescaling basis functions so that  $ \Vert \phi^\nu_i \Vert_{p,\mu} = N_\nu^{-1/p}.$
\begin{proposition}\label{prop:Lp-bound}
Assume $\Hc \subset L^p_\mu(\Xc)$, $1\le p \le \infty$. With 
the above choice of norms and normalization of basis functions, the continuity constant  $L_{p,\mu}$ defined by \eqref{Lp} is such that $L_{p,\mu} \le 1$,  and for all $1\le q \le p$, $L_{q,\mu} \le  \mu(\mathcal{X})^{1/q - 1/p} L_{p,\mu} \le  \mu(\mathcal{X})^{1/q - 1/p} $. 
\end{proposition}
\begin{proof} See \Cref{proofEntropy}.\end{proof}


\section{Risk bounds  and model selection for tree tensor networks}
\label{sec:general_risk_bounds}

Let $\Xc$ equipped with a finite measure $\mu$. In this section we analyze  empirical risk minimization for contrasts computed over general families of functions associated to tree tensor networks built on approximation spaces  in $L^\infty_\mu (\Xc)$. 

\subsection{Risk bounds for tree tensor networks}
\label{subs:riskbounds}
\modif{We consider a model class $M$ of tensor networks with bounded parameters  (with the norms defined in \Cref{sec:choice-norms}), with $M := M^T_r(\Hc)_R$ for full tensor networks 
or $M^T_{r,\Lambda}(\Hc)_R$ for $\Lambda$-sparse tensor networks. 
We here consider as fixed the approximation space $\Hc$, the dimension tree $T$  and the ranks $r \in \Nbb^{|T|}$, and also the sparsity pattern $\Lambda$ for sparse tensor networks.
  We  assume   that $\Hc \subset L^\infty_\mu (\Xc)$. We denote by $C_M = C(T,r,\Hc)$ the representation complexity of $M$ defined by \eqref{representation-complexity} for full tensor networks, or   $C_M = C(T,r,\Hc,\Lambda)$ the sparse representation complexity of $M$ defined by \eqref{sparse-representation-complexity}. }
 We consider a risk 
  $$\Rc(f) =\Ebb(\gamma(f,Z)), $$
 where $Z$ is a random variable taking values in $\Zc$ and where 
 $\gamma : \Rbb^\Xc \times \Zc \to \Rbb$ is some contrast function.  The minimizer of the risk over measurable functions defined on $\Xc$ is the target function $f^\star.$ For $f$ random (depending on the data),  $ \Ebb(\gamma(f,Z))$ shall be understood as an expectation $\Ebb_Z(\gamma(f,Z))$ w.r.t. $Z$ (conditional to the data). We introduce the excess risk
$$
\Ec(f) = \Rc(f) - \Rc(f^\star).
$$

Given the model class $M$, we denote by $f^M$ a minimizer over $M$ of the risk $\Rc$, and by  $\hat f^M_n$ a minimizer  over $M$ of the empirical risk 
$$
\widehat \Rc_n(f) = \frac{1}{n} \sum_{i=1}^n \gamma(f,Z_i),
$$  
which is seen as an empirical process over $M$.
To obtain bounds of the estimation error, it remains to   
quantify the fluctuations of the centered empirical process $\bar \Rc_n(f)$ defined by
\begin{equation*}
 \bar \Rc_n(f)  =  \widehat \Rc_n(f) - \Rc(f) = \frac{1}{n} \sum_{i=1}^n \gamma(f,Z_i) - \Ebb(\gamma(f,Z)).
 \label{centered-empirical-process}
\end{equation*}

\begin{assumption}[Bounded contrast]\label{ass:bounded}
Assume that $\gamma$ is uniformly bounded over $M \times \Zc$, i.e. 
\begin{align*}
\vert \gamma(f , Z) \vert \le B
\end{align*}
holds almost surely for all $f \in M$, with $B$ a constant independent of $f$.
\end{assumption}
 
\begin{assumption}\label{ass:lipschitz}
Assume that $\gamma(\cdot,Z)$ is Lipschitz continuous over $M \subset L^{ \infty}_{\mu}(\Xc)$, i.e.
\begin{align*}
\vert \gamma(f , Z) - \gamma(g,Z) \vert \le \mathcal L \Vert f - g \Vert_{ \infty,\mu} 
\end{align*}
holds almost surely for all $f,g \in M$, with $\mathcal L$ a constant independent of $f$ and $g$. 
\end{assumption}

\begin{example}[Least-squares bounded regression] \label{ex:bounded regression}
We consider a random variable $Z = (X,Y)$, with \modif{$Y$ a random variable with values in $\Rbb$}, $X$ a  $\Xc$-valued random variable with probability law $\mu$. 
We consider the least-squares contrast   $\gamma(f,Z) = \vert Y - f(X) \vert^2$.  The excess risk   $\mathcal E(f) = \Rc(f) - \Rc(f^\star) = \Vert f - f^\star \Vert^2_{ 2,\mu}$   admits   $f^\star(x) = \Ebb(Y\vert X = x)$ as a minimizer. In the bounded regression setting, it is assumed that 
 $\vert Y\vert  \le R$ almost surely. For all $f\in M$, we have  
 $
 \gamma(f,Z)  \le  2  ( \vert Y\vert^2 +  \Vert f \Vert_\infty^2) $, so that 
$0 \le \gamma(f,Z) \le B$ almost surely, with 
$B = 4 R^2$. Also,  it holds almost surely 
 \begin{align*}
 \vert  \gamma(f,Z) - \gamma(g,Z) \vert &= \vert (   2Y - f(X) - g(X) ) ( f(X) - g(X))  \vert \\
 &\le (2\vert Y\vert + \Vert g \Vert_{\infty,\mu} +  \Vert f \Vert_{\infty,\mu})  \Vert f - g \Vert_{\infty,\mu}  .
 \end{align*}
 Then for all $f,g\in M$, $ \vert  \gamma(f,Z) - \gamma(g,Z) \vert \le  \mathcal L \Vert f-g\Vert_{\infty,\mu}$ with 
 $\mathcal L = 4 R$. 
 \end{example}

\begin{example}[$L^2$ density estimation] \label{dens_estim}
For the problem of estimating the probability distribution of a random variable $X$, we consider $Z = X$. 
We consider the estimation of the probability law $\eta$ of $X$. Assuming that $\eta$ admits a density $f^\star$ with respect to the measure $\mu$, and assuming $f^\star \in L^2_\mu(\Xc)$, we consider the contrast $\gamma(f,x) = \Vert f \Vert_{2,\mu}^2 - 2f(x) $, so that 
$\mathcal E(f) = \Rc(f) - \Rc(f^\star) = \Vert f - f^\star \Vert^2_{ 2,\mu}$ admits $f^\star$ as a minimizer. We assume that $\mu$ is a finite measure on $\mathcal{X}$  and that $f^\star$ is uniformly bounded by $R$. 
Then  $\vert \gamma(f,X) \vert \le B$ almost surely  with $B = R(\mu(\Xc) R + 2 )$. Also, for all $f,g \in M$, we have almost surely 
\begin{align*}
\vert \gamma(f,X) - \gamma(g,X) \vert &= \vert \Vert f \Vert_{ 2,\mu}^2 - \Vert g\Vert_{ 2,\mu}^2 - 2(f(X) - g(X)) \vert 
\\
&\le \vert \int ( f - g )(f+g)  d\mu \vert + 2  \Vert f - g \Vert_{\infty,\mu}  \\
&\le ( \Vert f+ g\Vert_{1,\mu}  +2 )  \Vert f - g \Vert_{\infty,\mu} 
\\
&\le  \mathcal L \Vert f-g \Vert_{ \infty,\mu}
\end{align*} 
with $\mathcal L =2(\mu(\Xc)R +1 ) $.
\end{example}

\begin{proposition}\label{prop:probabound}
Under Assumptions \ref{ass:bounded} and \ref{ass:lipschitz},  for  any $t >0$,  with probability larger than $ 1-  \exp (- t  )$,  
\begin{equation*} \label{riskboundexecess}
\mathcal E (\hat f_n^M)  \leq \mathcal E (f^M)   + 8B     \sqrt{C_M} \sqrt{  \frac{  2 \log (  6\mathcal LB^{-1}  R  |T|   \sqrt n) } n } + 4 B   \sqrt{\frac t {2n} }.
  \end{equation*}
By integrating according to $t$,  we obtain that 
\begin{equation*} \label{riskboundexecess_expe}
\mathbb E \mathcal E (\hat f_n^M)  \leq \mathcal E (f^M)   + 8B     \sqrt{C_M} \sqrt{  \frac{  2 \log (  6\mathcal LB^{-1}  R  |T|   \sqrt n) } n } + 2 B  \sqrt \frac{  \pi}{n} .
  \end{equation*}
\end{proposition}
This result is a standard application of the    bounded difference inequality (see for instance Theorem 5.1   in  \cite{Massart:07}) applied  to $\sup_{f \in M} |  \bar \Rc_n(f)  |  $, together with a control on $\mathbb E \sup_{f \in M} |  \bar \Rc_n(f)  | $ with the metric entropy result of  \Cref{prop:metric-entropy}. The proof is given in \Cref{sec:proof:prop:probabound}.


\subsection{Model selection for tree tensor networks}\label{sec:model-selection}

\modif{We now consider a family of tensor networks $(M_m)_{m \in \mathcal M} $ indexed by a countable set $\mathcal M$.
For full tensor networks, a model $M_m = M_{r_m}^{T_m} (\Hc_m)_{R}  $ is associated with a particular tree $T_m$, a rank $r_m$, an approximation space $\Hc_m,$ and a radius $R$. For sparse tensor networks, a model  $M_m = M_{r_m,\Lambda_m}^{T_m} (\Hc_m)_{R}  $ has for additional parameter a sparsity pattern $\Lambda_m$. We denote by $C_m$ the  number of parameters of the model $M_m$, that is $C_m = 
 C(T_m,r_m,\Hc_m)$ for full tensors, or 
$C_m = 
 C(T_m,r_m,\Hc_m,\Lambda_m)$ for sparse tensors}. 
 
 For some $m\in  \mathcal M$, we let $f_m$ be a minimizer of the risk over $M_m$, 
$$ f_m \in \arg\min_{f \in M_m}  \Rc(f),$$
 and $\hat f_m$ be a minimizer of the empirical risk over $M_m,$
$$ \hat f_m \in \arg\min_{f \in M_m} \widehat \Rc_n (f) .$$

At this stage of the procedure, we have at hand a family of predictors $ \hat f_m$  and our goal is to provide a strategy for selecting a good predictor in the collection. We follow a standard strategy  that corresponds to the so-called {\it Vapnik's structural minimization of the risk method}  (see for instance \cite[Section~8.2]{Massart:07}). Given some penalty function $\pen : \mathcal M \rightarrow  \R^+$, we define $\hat m$ as the minimizer over $\Mc$ of the criterion
\begin{equation}
\label{critpen}
 \crit(m) :=   \widehat \Rc_n(\hat f_m) + \pen(m),
 \end{equation}
and we finally select the predictor $\hat f_{\hat m}$ according to the criterion~\eqref{critpen}. This procedure is classical in non parametric statistics and similar model selection approaches can be found in \cite{vapnik2013nature,gyorfi2006distribution,bousquet2004introduction}.
 %

For a suitable choice of penalty which takes into account both the complexity of the models and the richness of the model collection, we provide a risk bound for the selected predictor. Let $$\Nc_c := \Nc_c(\Mc) =   \left | \left\{ m \in \mathcal M \:  : \: C_m = c \right\} \right| $$  be the number of models with complexity $c$ in the collection.
\modif{The following result corresponds to the general Theorem 8.1 in \cite{Massart:07} applied to our framework.
\begin{theorem}  \label{theo:selec_model_tensor}
Let $\bar w >0$. Under Assumptions~\ref{ass:bounded} and \ref{ass:lipschitz},
 if the penalty  is such that 
 \begin{equation}
\pen(m) \geq \lambda_m  \sqrt{ \frac{C_m}{n}} + 2 B   \sqrt{\frac {\bar w C_m +  \log  (\Nc_{C_m}) } {2n} }, \label{pen_choice}
 \end{equation}
 with $$\lambda_m = 4B \sqrt{   2  \log (  6\mathcal LB^{-1}  R   |T_m|   \sqrt n) },$$ 
then the estimator $ \hat f_{\hat m}$ selected according to the criterion~\eqref{critpen} satisfies the following risk bound 
 \begin{equation}
\label{risk_bound_gen}
\mathbb E   (\Ec( \hat f_{\hat m}) ) \leq \inf_{m \in \mathcal M} \left\{ \Ec(f_m)       +  \pen (m) \right\}   +  \frac{B}{\exp (\bar w) - 1}  \sqrt{ \frac \pi{2n}}.
 \end{equation}
\end{theorem}
\begin{proof}
The proof of Theorem~\ref{theo:selec_model_tensor} is given in \Cref{sec:proof:theo:selec_model_tensor}, it is a direct adaptation of the proof of Theorem 8.1 in \cite{Massart:07}. 
\end{proof}
}

\subsection{\modif{Collections of models and their richness}}\label{sec:collections}

We here present and analyze the richness of different collections of tensor networks $(M_m)_{m\in \Mc}$, where each model has a particular feature space $\Hc_m$, a 
tree $T_m$, a tuple of ranks $r_m$. 
These collections of models depend on whether the feature space and the tree are considered as fixed. More precisely, we consider the following collections of models $(M_m = M_{r_m}^{T_m}(\Hc_m))_{m\in \Mc}$ with $\Mc$ corresponding to one of the following collections:
\begin{itemize}
\item $\Mc_{\Hc,T}$: fixed feature space $\Hc_m = \Hc$, fixed tree $T_m = T$, variable ranks $r_m$, 
\item $\Mc_T$: variable feature space $\Hc_m$, fixed tree $T$, variable ranks $r_m$, 
\item $\Mc_\star$: variable feature space $\Hc_m$, variable tree $T_m$, variable ranks $r_m$.
\end{itemize} 
 For variable feature spaces, we classically consider that $\Hc_m := \Hc_{N_m}$ with $N_m \in \Nbb^{d}$ and for any $N\in \Nbb^d$, $\Hc_N = \Hc_{1,N_1} \otimes \hdots \otimes \Hc_{d,N_d},$ where $(\Hc_{\nu,N_\nu})_{\nu \in \Nbb}$ is a sequence of subspaces of univariate functions, with $N_\nu = \dim(\Hc_{\nu,N_\nu}).$  For variable trees, we consider trees in the family of trees with arity $a$ (or $a$-ary trees), the case $a=2$ corresponding to (full) binary trees.
The next result provides upper bounds of the complexity of the above defined families of tensor networks.
\begin{proposition}[Collections of full tensor networks] \label{prop:N_C}
Consider a family of full tensor networks $(M_m = M_{r_m}^{T_m}(\Hc_m))_{m\in \Mc}$
 with $\Mc$ equal to $\Mc_{\Hc,T}$, $\Mc_{T}$ or $\Mc_{\star}$. For any tree $T$ and any feature space $\Hc$, 
$
\Nc_c(\Mc_{\Hc,T}) \le \Nc_c(\Mc_{T}) \le \Nc_c(\Mc_\star),
$
and 
$$
 \log(\Nc_c(\Mc_\star)) \le 2a( c + d \log(c) ).
$$
with $a$ the arity of the considered trees.
\end{proposition} 
\begin{proof} See \Cref{sec:proof:prop:N_C}. \end{proof}
When exploiting sparsity, we consider models $M_m=M_{r_m,\Lambda_m}^{T_m}(\Hc_m)$ depending on an additional sparsity pattern $\Lambda_m$. For variable feature spaces $\Hc_m = \Hc_{N_m}$, we consider models $m$ such that
$N_m \in \Nbb^{d}$ satisfies  
\begin{equation}
N_m  \le g(C_m),
\end{equation}
with $g$ some increasing function of the complexity $C_m$ of the model $m$. This is a reasonable assumption from a practical point of view, where for a given complexity, we avoid the exploration of infinitely many features. 
We use the same notations $\Mc_{\Hc,T}$, $\Mc_T$ and $\Mc_\star$ for the corresponding families of models, with $\Lambda_m$ considered as an additional free variable. 
The complexities of these collections of sparse tensor networks are higher than the corresponding complexities for full tensor networks, but only up to logarithmic terms, as shown in the next result.
\begin{proposition}[Collections of sparse tensor networks] \label{prop:N_C_sparse}
Consider a family of sparse tensor networks $(M_m = M_{r_m,\Lambda_m}^{T_m}(\Hc_{N_m}))_{m\in \Mc}$  with $\Mc$ equal to $\Mc_{\Hc,T}$, $\Mc_{T}$ or $\Mc_{\star}$, with variable sparsity patterns $\Lambda_m$ and $N_m \le g(C_m)$. For any tree $T$ and any feature space $\Hc$, 
$
\Nc_c(\Mc_{\Hc,T}) \le \Nc_c(\Mc_{T}) \le \Nc_c(\Mc_\star),
$
and 
$$
 \log(\Nc_c(\Mc_\star)) \le 5ac\log(c) + 2c \log(g(c)).
$$
If we further assume that $\log(g(c)) \le \delta \log(c)$ for some $\delta>0$,
then 
$$
\log(\Nc_c(\Mc_\star)) \le  (5a + 2\delta) c \log(c).
$$
\end{proposition}
\begin{proof} See \Cref{sec:proof:prop:N_C_sparse}. \end{proof}

Together with \Cref{prop:N_C} (or \Cref{prop:N_C_sparse}), \Cref{theo:selec_model_tensor} provides  a strong justification for using a penalty proportional to  $\sqrt{C_m/n}$. However, it is known that the Vapnik's structural minimization of the risk may lead to suboptimal rates of convergence.  For instance, in the bounded regression setting, it is known that a penalty 
proportional to the  Vapnik–Chervonenkis  dimension (typically in $O(C_m /n)$) leads to minimax rates of convergence in various setting  (see for instance Chapter 12 in \cite{gyorfi2006distribution}) whereas Vapnik's structural minimization of the risk (typically with penalty in $O(\sqrt{C_m /n})$) is too  pessimistic to provide fast rates of convergence.
\\
In the case of bounded least squares contrasts, we give in \Cref{sec:leastsquares} improved risk bounds. That allows us to prove that our model selection strategy is (near to) adaptive minimax in several frameworks, as shown in \Cref{sub:rates}.

\section{Oracle inequality for least squares inference with tree tensor networks}
\label{sec:leastsquares}
 
In this section, we provide an improved excess risk bound  in the specific case of least squares contrasts. Our results come from Talagrand inequalities and  generic chaining bounds;  we follow the presentation given in  the monograph  \cite{koltchinskii2011oracle}. The excess risk bound given below strongly relies on the link between the excess risk and the variance of the excess loss, as explained in Chapter~5 of \cite{koltchinskii2011oracle} and Chapter 8 in \cite{Massart:07}. We then derive an improved model selection result for least squares inference by following the approach presented in Sections 8.3 and 8.4 of \cite{Massart:07} or in Section 6.3 of \cite{koltchinskii2011oracle}. 

Let  $\gamma$ be either the least squares contrast in the bounded regression setting  (as described in  Example~\ref{ex:bounded regression}), or the least squares contrast for density estimation (as described in  Example~\ref{dens_estim}). 

\subsection{Improved risk bounds for least squares contrasts}
\label{sub:improved}


We first consider as model class a tree tensor network  $M = M^T_r(\Hc)_R  $ \added{or $M = M^T_{r,\Lambda}(\Hc)_R  $ (respectively full or $\Lambda$-sparse)} with bounded parameters and it assumed that the feature tensor space $\Hc \subset L_\mu^{\infty}(\Xc)$  where $\mu$ is the distribution of the random variable $X$ in the regression setting (see  Example~\ref{ex:bounded regression}) or the reference measure for density estimation (see Example~\ref{dens_estim}).  
 
\begin{proposition} \label{prop:riskboundTalagrand}
\modif{Under Assumptions~\ref{ass:bounded} and \ref{ass:lipschitz},} there exists an  absolute constant $\mathcal A$ and a constant  $\kappa$  such that for any $\varepsilon \in (0,1]$ and any $t >0$,  with probability at least $1- \mathcal A \exp(-t)$, it holds 
\begin{equation}  \label{riskbound-leastsquares}
\mathcal E (\hat f_n^M )   \leq 
(1+ \varepsilon) \mathcal E  (f^M)  + 
  \frac{\kappa  R^2}{n}  \left[  \frac{a_T  C_M } {\varepsilon^2 } \log^+  \left( \frac{ n  \varepsilon ^2  }{ a_T C_M}\right) 
 +  \frac{t }{ \varepsilon} \right]    
\end{equation}
where $ a_T =   1 + \log^+ \left( \frac{3 |T| } {4e}  \right)$, and $\kappa$ depends   linearly on $\mu(\Xc)$\footnote{With $\mu(\Xc)=1$ for regression.}. Then by integrating according to $t$,  we obtain that for any $\varepsilon \in (0,1]$,
$$ \mathbb E \mathcal E (\hat f_n^M )   \leq (1+ \varepsilon) \mathcal E (f^M)  +   \frac{\kappa  R^2  }{n }    \left[  \frac{a_T C_M  }{ \varepsilon^2}   \log^+ \left( \frac{ n  \varepsilon ^2  }{ a_T C_M}\right) +   \frac{\mathcal A } {\varepsilon}  \right].
$$ 
\end{proposition}
\begin{proof}
The proof of the proposition is given in \Cref{sec:proofriskboundTalagrand}.\end{proof} 
Note that the term $a_T$ is upper bounded by a term of the order of $\log(d)$ because $|T| \leq 2d$.  Thus the constants in the risk bound~\eqref{riskbound-leastsquares} does not explode with the dimension $d$ in regression. Note however that in density estimation, the constant $\kappa$ depends linearly on the mass $\mu(\Xc)$ of the reference measure, which may grow exponentially with $d$.

 \subsection{Oracle inequality}

\modif{As in  \Cref{sec:model-selection}, we now consider a family of tensor networks $(M_m)_{m \in \mathcal M} $ indexed by a countable set $\mathcal M$, with  either $M_m = M_{r_m}^{T_m} (\Hc_m)_{R}  $  for full tensor networks, or 
$M_m = M_{r_m,\Lambda_m}^{T_m} (\Hc_m)_{R}  $ for sparse tensor networks. We consider  
features spaces   
$\Hc_m    \subset L^{\infty}_\mu(\Xc)$ with $\Xc$ equipped with a finite measure $\mu$.
As before,  $\Nc_c(\Mc)$ denotes the number of models  with complexity $c$ in the collection $\Mc$ (see \Cref{sec:collections}).}

\begin{theorem}  \label{theo:fast_rates}
Let $\bar w >0$.  \modif{Under Assumptions~\ref{ass:bounded} and \ref{ass:lipschitz},},  there exists numerical constants $K_1$ and  $K_2$ and $K_3$ such that if the penalty  satisfies
$$
\pen(m) = K_1  R ^2  \left[   \frac{b_m C_m}{n \varepsilon^2}   \log^+  \frac{n \varepsilon^2}{b_m C_m}  + \frac{\bar w C_m +  \log  (\Nc_{C_m}) }{n \varepsilon}   \right] 
$$
with  $b_m =   1 + \log^+ \left( \frac{3 |T_m| } {4e}  \right)$, then  the estimator $ \hat f_{\hat m}$ selected according to the penalized criterion~\eqref{critpen} satisfies the following oracle inequality
\begin{equation} \label{eq:riskbound-fast}
\begin{split} \\
\mathbb E 
 \mathcal E (\hat f _{\hat m})   & \leq   \frac{1+\varepsilon}{1-\varepsilon} 
  \inf_{m \in \mathcal M} \left\{  \mathcal E (f _m)   
  +K_2 \pen(m)  \right\}  
    + \frac{K_3  R ^2}{ \exp (\bar w) - 1  } \frac{1+\varepsilon}{\varepsilon(1-\varepsilon)} \frac 1n .
  \end{split}
 \end{equation}
\end{theorem}
\begin{proof}
The proof, adapted  from Theorem 6.5 in \cite{koltchinskii2011oracle}, is given in \Cref{sec:proof:theo:fast_rates}. 
\end{proof}
\modif{For collections of models $\Mc$ such that $$\log(\Nc_{C_m}(\Mc)) \sim C_m \log(C_m)^\delta$$
for some $\delta \ge 1$,  
this theorem provides an improved oracle inequality bound 
\begin{align}
\mathbb E 
 \mathcal E (\hat f _{\hat m})  \lesssim  \inf_{m \in \mathcal M} \mathcal E (f _m)    + \frac{C_m}n \log(n) \log(C_m)^\delta,
 \label{orable-Ncgrowth}
\end{align}
with a penalty in $\frac{C_m}n$, up to logarithmic terms. }

 In Section~\ref{sub:rates}, we will derive adaptive  (near to) optimal rates of convergence for smoothness classes (in the minimax sense) from this model selection result. In  Section~\ref{sec:slope-heuristics} we explain how to calibrate the penalty in practice using the slope heuristics method.

\section{\modif{Least-squares inference and minimax adaptivity for smoothness classes}} \label{sub:rates}
 
Here we consider bounded least-squares inference with target functions $f^\star$ in classical smoothness spaces including Sobolev or Besov spaces (with isotropic, anisotropic or mixed dominating smoothness), or spaces of analytic functions. We consider functions defined on the hypercube $[0,1)^d$ equipped with the uniform measure $\mu$. For clarity, we let $L^p := L^p_\mu([0,1)^d).$

A classical approach is to consider tensor networks with feature tensor spaces $\Hc_{m}$ that are adapted to the smoothness of the function (e.g. tensorized splines or wavelets for Besov smoothness, or tensorized polynomials for analytic functions). Here, we use an alternative and powerful approach based on tensorization of functions, which can be interpreted as a particular definition of feature space. It does not require to adapt the tool to the regularity of the function. This approach is described in \Cref{sec:tensorization} and \Cref{sec:approx-classes} (for  more details see \cite{Ali2020partI,Ali2021partIII}). 
Then in \Cref{sec:rates},  
we show that our model selection strategy with this tool is minimax adaptive to a wide range of smoothness classes.

\subsection{Feature space based on tensorization of functions at fixed resolution}\label{sec:tensorization}

For any integers $b,L\in \Nbb$ with $b \ge 2$, we introduce a uniform partition of the interval $[0,1)$ into $b^L$ 
intervals of equal length $b^{-L}$. Any $x\in [0,1)$ can be written 
$$
x = \sum_{k={1}}^{{L}} i_k b^{-k} +  b^{-L}\bar x := t_{b,L}(i_1,\hdots,i_L,\bar x),
$$
where $(i_1,\hdots,i_L) \in \{0,\hdots,b-1\}^L$ is the representation in base $b$ of the integer $i$ such that $x\in [b^{-L}i,b^{-L}(i+1)),$ and $\bar x \in [0,1)$. The integer $L$ is called the \emph{resolution}. The map 
$t_{b,L}$ is a bijection from $\{0,\hdots,b-1\}^d \times [0,1)$ to $[0,1)$  with inverse $t_{b,L}^{-1}(x)  = (i_1,\hdots,i_L,\bar x)$ such that 
$$
i_k =  \lfloor b^k x\rfloor \, \mathrm{mod} \, b, \quad \bar x = b^L x - \lfloor b^L x \rfloor.
$$
A function $f(x)$ defined on $[0,1)$ can then be linearly identified with a $(L+1)$-variate function $\boldsymbol{f}(i_1,\hdots,i_L,\bar x)$ defined on $ \{0,\hdots,b-1\}^L \times [0,1)$. The map $\Tc_{b,L}$ which associates to a function $f$ the multivariate function $\boldsymbol{f}$ is called the tensorization map. 

For multivariate functions $f(x_1,\hdots,x_d)$ defined on the hypercube $[0,1)^d$, we proceed in a similar way for each dimension. Each variable $x_\nu$ is identified with a tuple $(i_1^\nu,\hdots,i_L^\nu,\bar x_\nu) = t_{b,L}(x_\nu)$, and $f$ is linearly identified with a $d(L+1)$-variate function 
$\boldsymbol{f}(i_1^1,\hdots,i_L^1, \hdots, i_1^d,\hdots,i_L^d,\bar x_1,\hdots , \bar x_d)$ defined on $ \{0,\hdots,b-1\}^{Ld} \times [0,1)^d$. 

For any $1\le p \le \infty$, the tensorization map $\Tc_{b,L}$ which associates to a $d$-variate function $f$ the tensor   $\boldsymbol{f}$ of order $(L+1)d$  is a linear isometry from $L^p([0,1)^d)$ to the tensor Banach space $(\Rbb^b)^{\otimes Ld} \otimes L^p([0,1)^d) = L^p_{\boldsymbol \mu}(\{0,\hdots,b-1\}^{Ld} \times [0,1)^d)$ equipped with the uniform measure $\boldsymbol{\mu}$ over $\{0,\hdots,b-1\}^{Ld} \times [0,1)^d$ \cite[Theorem 2.2]{Ali2021partIII}.

To define an approximation tool, we then introduce a finite-dimensional tensor space 
$$\mathbf{V}_{L} = (\Rbb^b)^{\otimes d L} \otimes (\Pbb_k)^{\otimes d}$$
where $\Pbb_k$ is the space of univariate polynomials of degree less than $k$. To a tensor 
$\boldsymbol{f} \in \mathbf{V}_{L} $ correspond a function $f = \Tc_{b,L}^{-1}\boldsymbol{f} \in L^\infty([0,1)^d)$ which is a spline of degree $k$ on the uniform partition of $[0,1)^d$.  This defines a feature tensor space with dimensions $N = (N_1,\hdots,N_{(L+1)d})$, $N_\nu = b$ for $1 \le \nu \le Ld $ and $N_\nu = k+1$ for $\nu >Ld,$ with a feature map 
$$
\phi(x) = e(i_1^1) \otimes e(i_2^1) \otimes \hdots \otimes  e(i_L^d) \otimes  \varphi(\bar x_1) \otimes \hdots \otimes \varphi(\bar x_d)
$$
where $e(i) \in \Rbb^{b}$ is such that $e(i)_j = \delta_{i,j}$ and $\varphi(t) = (\varphi_j(t))_{0\le j\le k} $ is a basis of $\Pbb^k$.

\subsection{Tensor networks with variable resolution: complexity and approximation classes} \label{sec:approx-classes}

Here we consider 
tensor networks  over the tensor space $\mathbf{V}_{L} $, either $M^T_{r}(\mathbf{V}_{L} )$ for full tensor networks or 
$M^T_{r,\Lambda}(\mathbf{V}_{L} )$ for sparse tensor networks, where $T$ is a dimension tree over $\{1,\hdots,d(L+1)\}$, $r \in \Nbb^{\vert T \vert}$, and $\Lambda$ some sparsity pattern. This defines a subset of $d$-variate functions through the map $\Tc_{b,L}^{-1}$. For a \emph{linear tree} $$T = T_L :=  \{\{1\}, \hdots , \{d(L+1)\}  , \{1,2\}, \{1,2,3\},  \hdots,\{1,\hdots,d(L+1)\}\},$$ the tensor network corresponds to a \emph{tensor train} (TT) format. 
\begin{remark}
For the approximation of functions from classical smoothness classes, and when working with a fixed tree, this choice of tree is rather natural. Each interior node in $T_L$ is related to a splitting of variables into a group of low-resolution variables and high-resolution variables (see discussions in \cite{Ali2020partI,Ali2021partIII} on the impact of the tree).
\end{remark}

\subsubsection{Collections of tensor networks and their richness.}
 We consider as an approximation tool a collection of tensor networks with variable resolutions and variable ranks with a tensor train format. 
 More precisely, we define a collection of models $(M_m)_{m\in \Mc}$ in $L^\infty([0,1)^d)$ defined by 
 $$
 M_m =  \Tc_{b,L_m}^{-1} M^{T_{L_m}}_{r_m}(\mathbf{V}_{L_m} )_R \quad \text{or} \quad  M_m =  \Tc_{b,L}^{-1} M^{T_{L_m}}_{r_m,\Lambda_m}(\mathbf{V}_{L_m} )_R ,
 $$
 with variable resolutions $L_m \in \Nbb$, linear trees $T_{L_m}$ and variable ranks $r_m \in \Nbb^{\vert T_{L_m}\vert}$. 
 
   \begin{remark}
Note that for a particular resolution $L$, we here consider a single tree $T_L$. This is sufficient for obtaining our minimax  results for classical smoothness classes in \Cref{sec:rates}. Working with variable trees may be relevant for highly structured functions or functions beyond classical smoothness classes. 
Our tree selection procedure should be able to recover a near-optimal tree, that is relevant for applications where there is no a priori for the selection of a good tree.
 \end{remark}
 
  Note that since $\Tc_{b,d}$ is a linear isometry from $L^p_\mu$ to $L^p_{\boldsymbol{\mu}}$, the metric entropy $H(\epsilon , M_m , \Vert \cdot\Vert_{p,\mu})$ of $M_m$ is equal to the metric entropy of the corresponding tensor network in  
 $L^p_{\boldsymbol{\mu}}$.

 For a model $m$ with complexity $c$, we clearly have $L_m \le c$. Then 
the number $\Nc_c(\Mc)$ of models with complexity 
 $c$ is such that 
$$
\Nc_c(\Mc) = \sum_{L = 0}^\infty \Nc_c(\Mc_{T_L}) =\sum_{L=0}^c \Nc_c(\Mc_{T_L})  ,
$$
with $\Mc_{T_L}$ the collection of tensor networks with fixed tree $T_L$, fixed feature space $\mathbf{V}_L$, and variable ranks. 
 We deduce from \Cref{prop:N_C} (with $d$ replaced by $(L+1)d \le c$) and \Cref{prop:N_C_sparse} (with a constant function $g(c) = \max\{b,k+1\}$) that 
\begin{align}
\Nc_c(\Mc) \lesssim c \log(c),\label{Nctensorized}
\end{align}
for both full and sparse tensor networks. 

Given a collection  of tensor networks $(M_m)_{m\in \Mc}$ introduced above (either full or sparse), 
 we define an approximation tool $\Phi = (\Phi_c)_{c\in \Nbb}$, where the set $\Phi_c$ is the union of models with complexity less than $c$, i.e. 
$$ \Phi_c = \bigcup_{m \in \Mc , C_m\le c} M_m.$$
The approximation tool $\Phi$ is respectively denoted by $\Phi^\mathcal{F} = (\Phi^\mathcal{F} _c)_{c\in \Nbb}$ and $\Phi^\mathcal{S}=(\Phi^\mathcal{S}_c)_{c\in \Nbb}$  for full and sparse tensor networks.


%
 
 \subsubsection{Approximation classes}
 The best approximation error   of $f^\star$ in $L^2$ by a tensor network with complexity less than $c$ is 
  $$
  E(f^\star , \Phi_c)_{L^2} = \inf_{ f \in \Phi_c} \Ec(f)^{1/2} = \inf_{ f \in \Phi_c} \Vert f - f^\star \Vert_{2,\mu} .$$
 Then  given a growth function $\gamma : \Nbb\to  \Nbb$, an approximation class  for  tensor networks can be defined as the set of functions
$$
\Ac_\infty(\gamma,\Phi,L^2) = \{ f : \sup_{c \ge 1} \gamma(c)  E(f,\Phi_c)_{L^2} <\infty \},
$$ 
which corresponds to functions that can be approximated with tree tensor networks with an error $E(f^\star,\Phi_c)$ in $ O(\gamma(c)^{-1})$.

To polynomial growth functions $\gamma(c) = c^\alpha$ ($\alpha>0$) correspond approximation classes
$$
\Ac^\alpha_\infty := \Ac^\alpha_\infty(\Phi,L^2) = \{ f : \sup_{c \ge 1} c^\alpha  E(f,\Phi_c)_{L^2} <\infty \}
$$
containing functions that can be approximated by tensor networks with algebraic convergence rate in $E(f^\star,\Phi_c) \lesssim c^{-\alpha}.$ 
In \cite{Ali2020partI,Ali2021partIII}, it is proved that the sets $\Ac^\alpha_\infty$ are quasi-Banach spaces, equipped with the quasi-norm $\Vert f \Vert_{\Ac^\alpha_\infty} = \Vert f \Vert_{L^2} + \vert f\vert_{\Ac^\alpha_\infty} $ with $\vert f\vert_{\Ac^\alpha_\infty} = \sup_{c \ge 1} c^\alpha  E(f,\Phi_c)_{L^2}  $. A whole range of quasi-Banach spaces $\Ac_q^{\beta} $ can be defined by interpolation between $L^2$ and a space $\Ac^\alpha_\infty$, with $\Ac_q^{\beta} = (L^2 , \Ac_\infty^{\alpha})_{\beta/\alpha ,q} $, $0<\beta<\alpha$, $0<q\le \infty.$ The spaces   $\Ac_q^{\alpha}$ are included in $\Ac_\infty^{\alpha}$ and correspond to a slightly stronger convergence of approximation error.

The approximation classes associated with full and sparse tensor networks (associated with two different notions of complexity) are respectively denoted by 
$$
\Fc^\alpha_q = \Ac^\alpha_q(\Phi^\Fc,L^2) \quad \text{and} \quad \mathcal{S}^\alpha_q = \Ac^\alpha_q(\Phi^\mathcal{S},L^2).
$$ 
For any $0<q\le \infty$, we have the following continuous embeddings \cite[Theorem 4.12]{Ali2021partIII}
\begin{equation}
\Fc^{\alpha}_q \hookrightarrow \mathcal{S}^{\alpha}_q \hookrightarrow \Fc^{\alpha/2}_q. \label{relation-classes-F-S}
\end{equation}
That means that if full tensor networks achieve an approximation rate as $O(c^{-\alpha})$ then sparse tensor networks achieve at least the same approximation rate. However, if sparse tensor networks achieve an approximation rate as $O(c^{-\alpha})$, then full tensor networks achieve at least an approximation rate  as 
  $O(c^{-\alpha/2})$, i.e. with a possible deterioration of the rate by a factor $2$.

\begin{remark}
We recall that the results of this section  are valid for a collection of models where for a given resolution $L$, we consider a single tree $T_L$. When considering variable trees for a fixed resolution, we obtain much larger approximation classes. However, these are highly nonlinear classes and their properties have not been studied yet.
 \end{remark}


%
%

\subsection{Rates for smoothness classes}\label{sec:rates}

Here we show that (near to) minimax rates can be achieved by tensor networks with our model selection strategy for a wide range of smoothness classes encompassing isotropic Besov spaces, anisotropic Besov spaces, Besov spaces with mixed dominating smoothness and spaces of analytic functions..

For that, we rely on the oracle inequality from \Cref{sub:improved}, the estimates of the complexity of collections of tensor networks from \Cref{sec:approx-classes},  and 
 approximation results from \cite{Ali2020partII,Ali2021partIII}. 
 
 We start by providing a useful lemma which provides of convergence of the estimator $\hat f_{\hat m}$ for a target function in an approximation class of tensor networks. 
 In this section, we work under the assumptions of \Cref{theo:fast_rates}.
\begin{lemma} \label{lemma:orable-approx-class}
For any $\alpha>0$, if $f^\star \in \Ac^\alpha_\infty(\Phi,L^2)$, the estimator  $\hat f_{\hat m}$ obtained with the model selection strategy and the approximation tool $\Phi$ (either $\Phi^\Fc$ or $\Phi^\Sc$) satisfies 
$$
\mathbb E 
 \Vert \hat f _{\hat m} - f^\star \Vert_{2,\mu}^2
  \lesssim n^{-\frac{2\alpha}{2\alpha+1}} \log(n)^{\frac{4\alpha}{2\alpha+1}}.
$$
\end{lemma}
\begin{proof}
Using \eqref{theo:fast_rates} and the complexity estimate \eqref{Nctensorized}, we have
 \begin{align*}
\mathbb E 
 \mathcal E (\hat f _{\hat m}) & \lesssim  \inf_{m \in \mathcal M} \mathcal E (f _m)    + \frac{C_m}n \log(n) \log(C_m) \lesssim   \inf_{c\in \Nbb} c^{-2\alpha}   +\frac{c}{n} \log(n) \log(c).
\end{align*}
Let $c$ be such that $c^{-2\alpha} = \frac{c}{n} \log(n) \log(c)$. We have 
\begin{align}
c = n^{\frac{1}{2\alpha+1}}  \log(n)^{-\frac{1}{2\alpha+1}} \log(c)^{-\frac{1}{2\alpha+1}}. \label{equil-c}
\end{align}
For $n,c \ge 2$, we thus have $\log(c) \le \frac{1}{2\alpha+1} \log(n)$. Together with 
\eqref{equil-c}, it yields  $c\gtrsim n^{\frac{1}{2\alpha+1}} \log(n)^{-\frac{2}{2\alpha+1}} $ and therefore 
$
\mathbb E 
 \mathcal E (\hat f _{\hat m}) \lesssim c^{-2\alpha} \lesssim n^{-\frac{2\alpha}{2\alpha+1}} \log(n)^{\frac{4\alpha}{2\alpha+1}}.
$
\end{proof}

Next we denote by $\hat f^\Fc _{\hat m}$ and $\hat f^\Sc_{\hat m}$ the estimators obtained with our model selection strategy using full tensor networks or sparse tensor networks respectively.

\subsubsection{Besov spaces with isotropic smoothness}
 
 We let 
 $B_q^s(L^p)$ denote the Besov space of functions with regularity order $s>0$, primary parameter $p$ and secondary parameter $q$ (see \cite{Ali2021partIII,devore1988interpolation} for a definition and characterization). The parameter $p$ is related to the norm with which the regularity is measured. 
 
 For $s < 1$ and $q=\infty$, $B^{s}_{\infty}(L^p) $ corresponds to the space 
 $\mathrm{Lip}(s,L^p)$.
For $p=q$ and non-integer $s>0$, $B^{s}_{p}(L^p) $ corresponds to the (fractional) Sobolev space $W^{s,p}$.
For the special case $p=2$, $B^s_{2}(L^2) $ is equal to the Sobolev space $W^{s,2} = H^s$ for any $s>0$. 
For $s>d(1/\tau - 1/p )_+$, it holds that $B^{s}_q(L^\tau) \hookrightarrow L^p$. 

It is known that the minimax rate for functions $f^\star \in B^s_q(L^p)$ is lower bounded by $n^{-\frac{2s}{2s+d}}$ (see e.g. \cite{donoho1998minimax,gine2016mathematical}).  

\paragraph{Besov spaces $B^{s}_q(L^p)$ for $p\ge 2.$} 
We first consider Besov spaces $B^{s}_q(L^p)$ with smoothness measured in $L^2$ norm or stronger norm.
\begin{theorem}[Minimax rates for Besov spaces $B^{s}_q(L^p)$ for $p \ge 2$]
Assume the target function $f^\star \in B^s_q(L^p)$ with $s>0$, 
$2\le p \le \infty$ and $0<q \le \infty$.  Then for sufficiently large $n$, 
 $$
\mathbb E 
 \Vert \hat f _{\hat m}^\Fc - f^\star \Vert_{2,\mu}^2
 \lesssim n^{-\frac{2 \tilde s}{2\tilde s+d}} \log(n)^{\frac{4 \tilde s}{2\tilde s+d}}
 $$
 with $\tilde s = s$ if $k \ge s-1/2$ or an arbitrary $\tilde  s<s$ if $k < s-1/2$.
\end{theorem}
\begin{proof}
From \cite[Theorem 6.6]{Ali2021partIII}, we have the continuous embedding $
B^{s}_q(L^p) \hookrightarrow \Fc_q^{\tilde s/d}\hookrightarrow \Fc_\infty^{\tilde s/d},$
for any $s>\tilde s >0$, $2 \le p \le \infty$ and $0<q\le \infty$. The result follows from  \Cref{lemma:orable-approx-class}.
\end{proof}
The above theorem implies that our model selection procedure with full tensor networks achieves minimax rates (up to   logarithmic term) for the whole range of Besov spaces 
$B^{s}_q(L^p)$, $p\ge 2$. It is thus minimax adaptive to the regularity over these Besov spaces, i.e. it achieves minimax rates without the need to adapt the approximation tool to the regularity of the target function. Note that miximax rates for $B^{s}_q(L^p)$, $p\ge 2$, are also achieved with linear approximation tools such as splines, wavelets or kernel methods, but obtaining minimax adaptivity requires a suitable strategy for the selection of a particular family of splines, wavelets or kernels. Here, tensor networks are associated with spline functions of a fixed degree $k$, and minimax adaptivity is obtained for any fixed value of $k$, including $k=0$. This is made possible by allowing models with high resolution (corresponding to deep tensor networks). 

\paragraph{Besov spaces $B^{s}_q(L^\tau)$ for $\tau < 2.$} 

Now we consider the case of Besov spaces $B^{s}_q(L^\tau)$ with a regularity measured in a weaker $L^\tau$-norm, $\tau<2$.  These are spaces of functions with "inhomogeneous smoothness" that can be only well captured by nonlinear approximation tools. 
We consider spaces $B^{s}_q(L^\tau)$ with 
$ 1/2 < 1/\tau < s/d + 1/2$. In the usual  $(1/\tau,s)$ DeVore diagram of smoothness spaces, this corresponds to Besov spaces strictly above the critical line characterized by $s = d(1/\tau - 1/2).$ 
Besov spaces strictly above this line ($s > d(1/\tau - 1/2)$) are compactly embedded in $L^2$, while Besov spaces stricly below this line ($s<d(1/\tau - 1/2)$) are not embedded in $L^2$. 
\begin{theorem}[Minimax rates for Besov spaces $B^{s}_q(L^\tau)$]
Assume the target function $f^\star \in B^s_q(L^\tau)$ with $s>0$, 
$1/2 < 1/\tau < s/d + 1/2$ and $0<q \le \tau$. Then for sufficiently large $n$, the estimators using full or sparse tensor networks respectively satisfy 
 $$
\mathbb E 
  \Vert \hat f _{\hat m}^\Sc - f^\star \Vert_{2,\mu}^2
 \lesssim n^{-\frac{2\tilde s}{2\tilde s+d}} \log(n)^{\frac{4\tilde s}{2\tilde s+d}}
 $$
 and 
$$
\mathbb E 
  \Vert \hat f _{\hat m}^\Fc - f^\star \Vert_{2,\mu}^2
  \lesssim n^{-\frac{2\tilde s}{2\tilde s+2d}} \log(n)^{\frac{4\tilde s}{2\tilde s+2d}} 
 $$
with $\tilde s = s$ if $k \ge s-1/2$ or an arbitrary $\tilde  s<s$ if $k < s-1/2$.
   \end{theorem}
\begin{proof}
From \cite[Theorem 6.8]{Ali2021partIII}, we have the continuous embedding 
$B^{s}_q(L^\tau) \hookrightarrow \mathcal{S}_q^{\tilde s/d} \hookrightarrow \mathcal{S}_\infty^{\tilde s/d}$. The result then follows from \Cref{lemma:orable-approx-class} and \eqref{relation-classes-F-S}.
\end{proof}
For such spaces $B^{s}_q(L^\tau)$ above the critical line and $\tau<2$,  it is known that optimal linear estimators do not achieve the optimal rate. For $d=1$, optimal linear estimators achieve a  rate in 
$n^{-\frac{2s - 2 (1/\tau - 1/2)}{2s+1- 2 (1/\tau - 1/2)}}$, which is larger than  the minimax rate $n^{-\frac{2s }{2s+1}}$.
 Only nonlinear methods of estimation are able to achieve the minimax rate \cite{donoho1996density}. 
The above result shows that our model selection strategy with sparse tensor networks  achieves minimax rates or rates arbitrarily close to minimax (up to a logarithmic term) for the whole range of spaces $B^{s}_q(L^\tau)$, without requiring to adapt the tool to the regularity.   Note that   the estimation using full tensor networks presents a slightly deteriorated rate. In this nonlinear estimation setting, exploiting sparsity of the tensor network  is useful to obtain an optimal performance. Note that the chosen polynomial degree $k$ has only a little impact on the obtained results. If  this degree is adapted to the regularity ($k\ge s-1/2$), the minimax rate is achieved (up to logarithmic term) but any degree $k$ (including $k=0$) allows to achieve a rate arbitrarily close to optimal.

\subsubsection{Besov spaces with mixed dominating smoothness }

We here consider Besov spaces $MB_q^s(L^p)$ with mixed dominating smoothness (see \cite{Ali2021partIII,Hansen2012,NonCompact} for a definition and characterization). For $p=q=2$, $MB_2^s(L^2)$ corresponds to the mixed Sobolev spaces $H^{s,mix}$ of functions $f$ with partial derivatives $\partial_\alpha f$ in $L^2$ for any tuple $\alpha = (\alpha_1,\hdots,\alpha_d) $ with $\max_\nu \alpha_\nu \le s.$

We consider spaces $MB_q^s(L^\tau)$ such that $s > (1/\tau - 1/2)_+$, which are embedded in $L^2$ and strictly above the critical embedding line (with $\tau<2$ and $s < 1/\tau - 1/2$, spaces $MB_q^s(L^\tau)$ are not embedded in $L^2$).

\begin{theorem}[Minimax rates for Besov spaces $MB^{s}_q(L^\tau)$ with mixed dominating smoothness]\label{th:mixed}
Assume the target function $f^\star \in MB^s_q(L^\tau)$ with  
$s >  (1/\tau - 1/2)_+$ and $0<q \le \tau$. 
 Then for sufficiently large $n$, the estimators using full or sparse tensor networks respectively satisfy 
 $$
\mathbb E 
  \Vert \hat f _{\hat m}^\Sc - f^\star \Vert_{2,\mu}^2
 \lesssim n^{-\frac{2\tilde s}{2\tilde s+\red{1}}} \log(n)^{\frac{4\tilde s}{2\tilde s+\red{1}}}
 $$
 and 
$$
\mathbb E 
  \Vert \hat f _{\hat m}^\Fc - f^\star \Vert_{2,\mu}^2
  \lesssim n^{-\frac{2\tilde s}{2\tilde s+\red{2}}} \log(n)^{\frac{4\tilde s}{2\tilde s+\red{2}}} 
 $$
with $\tilde s = s$ if $k \ge s-1/2$ or an arbitrary $\tilde  s<s$ if $k < s-1/2$.
   \end{theorem}
   \begin{proof}
   From \cite[Theorem 6.8]{Ali2021partIII}, we have the continuous embedding 
$M B^{s}_q(L^\tau) \hookrightarrow \mathcal{S}_q^{\red{\tilde s}} \hookrightarrow \mathcal{S}_\infty^{\red{\tilde s}} $. The result then follows from \Cref{lemma:orable-approx-class} and \eqref{relation-classes-F-S}.
   \end{proof}		
   For $s > (1/p - 1/2)_+$, it is known that the minimax rate is lower bounded by 
$n^{-\frac{2s}{2s+1}}$ (up to a logarithmic term) \cite{suzuki2018adaptivity}. Therefore, \Cref{th:mixed} implies that our model selection strategy using sparse tensor networks achieve a rate arbitrarily close to minimax, up to a logarithmic term.
With full tensor networks, the rate is close to minimax but slightly worse. We emphasize that this result is valid for any value of $k$, including $k=0$. However, by adapting the degree $k$ to the regularity (i.e., $k \ge s-1/2$), sparse tensor networks even achieve exactly the minimax rate.  

Note that for $p \ge 2$, linear estimators based on hyperbolic cross approximation \cite{dung2018hyperbolic} achieve minimax rates, with a suitable choice of univariate approximation tools adapted to the regularity.   
Let us finally mention that for $p\ge 2$ and full tensor networks, by using  \cite[Theorem 6.6]{Ali2021partIII}, we can obtain a slightly better rate in $n^{-\frac{2s}{2s+C(d)}}$  with $1<C(d)<2$.

\subsubsection{Anisotropic Besov spaces}

We now consider anisotropic Besov spaces $AB^{\boldsymbol{\alpha}}_q(L^p)$, $\boldsymbol{\alpha} = (s_1,\hdots,s_d) \in \Rbb_{+}^d$, where $s_\nu >0$ is related to the regularity order with respect to the $\nu$-th coordinate (see   \cite{Ali2021partIII,Leisner} for a definition based on directional moduli of smoothness and the characterization of these spaces). For $\boldsymbol{\alpha} = (s,\hdots,s)$ with $s>0$, $AB^{\boldsymbol{\alpha}}_q(L^p)$ coincides with the isotropic Besov space $B^s_{q}(L^p)$. For a tuple $\boldsymbol{\alpha}$, we let $ s(\boldsymbol{\alpha}) := d(s_1^{-1} + \hdots + s_d^{-1})^{-1}$ be the aggregated smoothness parameter, such that $\min_\nu s_\nu := \underline{s} \le s(\boldsymbol{\alpha}) \le \bar s := \max_{\nu} s_\nu. $

We consider spaces $AB_q^{\boldsymbol{\alpha}}(L^\tau)$ with $\boldsymbol{\alpha}$ such that 
  $s(\boldsymbol{\alpha}) > d(1/\tau - 1/2)_+$, which are embedded in $L^2$. For these spaces, the minimax rate  is in $n^{-\frac{2s(\boldsymbol{\alpha})}{2s(\boldsymbol{\alpha}) + d}}$ \cite{neumann2000multivariate} and this rate can be achieved by linear estimators only for $\tau \ge 2.$
     
\begin{theorem}[Minimax rates for anisotropic Besov spaces $AB^{\boldsymbol{\alpha}}_q(L^\tau)$]\label{th:anisotropic}
Assume the target function $f^\star \in AB^{\boldsymbol{\alpha}}_q(L^\tau)$ with $\boldsymbol{\alpha} \in \Rbb_+^d$  such that $s(\boldsymbol{\alpha}) > d(1/\tau - 1/2)_+ $ and 
$0<q \le \tau$. 
 Then for sufficiently large $n$, the estimators using full or sparse tensor networks respectively satisfy 
 $$
\mathbb E 
  \Vert \hat f _{\hat m}^\Sc - f^\star \Vert_{2,\mu}^2
 \lesssim n^{-\frac{2\tilde s}{2\tilde s+d}} \log(n)^{\frac{4\tilde s}{2\tilde s+d}}
 $$
 and 
$$
\mathbb E 
  \Vert \hat f _{\hat m}^\Fc - f^\star \Vert_{2,\mu}^2
  \lesssim n^{-\frac{2\tilde s}{2\tilde s+2d}} \log(n)^{\frac{4\tilde s}{2\tilde s+2d}} 
 $$
with $\tilde s = s(\boldsymbol{\alpha})$ if $k \ge \bar s -1/2$ or an arbitrary $\tilde  s<s$ if $k < \bar s-1/2$.
   \end{theorem}
   \begin{proof}
   From \cite[Theorem 6.8]{Ali2021partIII}, we have the continuous embedding 
$AB^{\boldsymbol{\alpha}}_q(L^\tau)\hookrightarrow \mathcal{S}_q^{\tilde s/d} \hookrightarrow \mathcal{S}_\infty^{\tilde s/d} $. The result then follows from \Cref{lemma:orable-approx-class} and \eqref{relation-classes-F-S}.
   \end{proof}		
   
 \Cref{th:mixed} implies that our model selection strategy using sparse tensor networks achieves a rate arbitrarily close to minimax, up to a logarithmic term.
With full tensor networks, the rate is close to minimax but slightly worse. 
We again emphasize that this result is valid for any $k \in \Nbb$, including $k=0$. However, by adapting the degree $k$ to the highest regularity $\bar s$ (i.e., $k \ge \bar s-1/2$), sparse tensor networks even achieve exactly the minimax rate (up to the logarithmic term).

Let us mention that for $p\ge 2$ and full tensor networks, by using  \cite[Theorem 6.6]{Ali2021partIII}, we can obtain a slightly better rate in $n^{- {2\tilde s}/{(2 \tilde s+C(d) d)}}$  with $1<C(d)<2$.

Note that with a sufficient anisotropy such that $ \sum_{\nu=1}^d s_ \nu^{-1} \le \beta^{-1}$ with $\beta$ independent of $d$, we have $s(\boldsymbol{\alpha}) \ge  d \beta^{-1}$,  and for an arbitrary $\tilde \beta <\beta$, our strategy with sparse (resp. full) tensor networks achieves a rate in $n^{- {2\tilde \beta}/({2\tilde \beta + 1}})$ (resp. $n^{- {\tilde \beta}/({\tilde \beta + 1})}$), which is independent of the dimension $d$.

\subsubsection{Analytic functions.}
Here, we consider the case of analytic functions on a bounded interval. We restrict the analysis to functions defined on $[0,1]$ but the result could be easily extended to the multivariate case. 
\begin{theorem}[Analytic functions]
Assume $f^\star : [0,1] \to \Rbb$ admits an analytic extension on an open complex domain  including $[0,1]$.  Then for sufficiently large $n$,
$$
\mathbb E 
   \Vert \hat f _{\hat m}^\Fc - f^\star \Vert_{2,\mu}^2
   \lesssim n^{-1} \log(n)^{5/2}
 $$
 up to logarithmic terms.
\end{theorem}
\begin{proof}
It results from 	\cite[Main result 3.5]{Ali2020partII} that 	the approximation error with full tensor networks converges exponentially fast as $E(f^\star,\Phi^\Fc_c)_{L^2} = O(\rho^{-c^{1/3}})$ for some $\rho>1$ related the   size of the analyticity region.
That means $f^\star \in \Ac_\infty(\gamma,\Phi^\Fc,L^2)$ with a growth function $\gamma(c) = \rho^{c^{1/3}}$. 
  \Cref{theo:fast_rates} then implies $\mathbb E \mathcal E (\hat f _{\hat m}^\Fc) \lesssim \inf_{c\in \Nbb} 
  \gamma(c)^{-2} + c\log(c) \log(n) /n$, and the result is obtained by taking $c  \sim (\log(n)/\log(\rho))^{3/2}$.
 \end{proof}
 The rate in $n^{-1}$ (up to logarithmic terms) achieved by full tensor networks  is known to be the minimax rate for analytic functions for nonparametric estimation of analytic densities \cite{belitser1998efficient}.

\subsection{Beyond smoothness classes}

We have seen that the proposed strategy is (near to) minimax  adaptive 
to a large range of classical smoothness classes. 
In \cite[Theorem 6.9]{Ali2021partIII}, it is proved that for any $\alpha>0$ and any $s >0$, it holds 
$$
\Fc^\alpha_q(L^2) \not\hookrightarrow B^s_q(L^2),
$$
that means that functions in the approximation classes of tensor networks do not need to have any smoothness in a classical sense. Tensor networks may thus achieve a good performance for functions that can not be captured by standard approximation tools such as  splines or  wavelets. 
That reveals the potential of tensor networks to achieve approximation or learning tasks for functions beyond standard smoothness classes. In particular, they have the potential to achieve a good performance in high-dimensional approximation tasks for function classes not described in terms of standard weighted or anisotropic smoothness.  

Note that in \cite[Proposition 5.21]{Ali2020partII}, it is proved that when limiting the resolution $L$ to be logarithmic in the complexity $c$ (i.e. when considering for $\Phi_c$ models for which $L = O(\log(c))$), the resulting approximation classes of tensor networks are continuously embedded in some Besov spaces. This highlights the importance of the resolution (or depth of the tensor network). Addressing learning tasks for functions beyond regularity classes requires to explore model classes with higher resolutions (i.e. with resolutions $L$ higher than $O(\log(c))$ and up to $c$).

Let us finally recall that the results of \Cref{sec:rates} have been obtained with tensor networks with variable resolution but a fixed tree at each resolution (corresponding to the tensor train format). Adaptiveness to a wide range of smoothness classes is thus achieved without tree adaptation. 
Much larger approximation classes are obtained by considering tensor networks with variables trees. 
In  many high-dimensional applications, adapting the tree to the target function 
is necessary to achieve a good performance and circumvent the curse of dimensionality.
Working with variable trees may thus be relevant to approximate highly structured functions beyond classical anisotropic smoothness spaces. 
Of course, this comes with a much higher computational complexity and requires in practice some exploration strategies as discussed in \Cref{sec:practical}.

\section{Practical aspects}\label{sec:practical}
\subsection{Slope heuristics for penalty calibration}\label{sec:slope-heuristics}

The aim of the slope heuristics method proposed by Birg\'e and  Massart~ \cite{birge2007minimal} is precisely to calibrate penalty function for model selection purposes. See \cite{baudry2012slope} and ~\cite{arlot2019minimal} for a general presentation of the method. This method has shown very good performances and comes with mathematical guarantees in various settings. For non parametric Gaussian regression with i.i.d. error terms,  see \cite{birge2007minimal,arlot2019minimal} and references therein. The slope heuristics have several versions (see \cite{arlot2019minimal}). 

The aim is to tune the constant $\lambda$ in a penalty of the form $\pen(m) = \lambda \pen_{\tiny \mbox{shape}}(m)$ where $\pen_{\tiny \mbox{shape}}$ is a known penalty shape. Let  $\hat{m}(\lambda)$ be the model selected by penalized criterion with constant $\lambda$:
\[\hat{m}(\lambda) \in \mathrm{argmin}_{m \in \mathcal{M}} \left\{  \widehat \Rc_n(\hat f_m)  +  \lambda \pen_{\tiny \mbox{shape}} (m) \right\}.\]
Let $C_m$ denote the complexity of the model. The complexity jump algorithm consists of the following steps:
\begin{enumerate}
\item  Compute the function $ \lambda  \mapsto  \hat{m}( \lambda )$,
\item Find the constant $ \hat \lambda^{cj}> 0$ that corresponds to the highest jump of the function $\lambda \mapsto C_{\hat{m}(\lambda)}$,
\item Select the model $\hat m = \hat{m}(2 {\hat \lambda}^{cj})$  such that 
  $$\hat{m} \in \arg\min_{m \in \mathcal{M}} \left\{  \widehat \Rc_n(\hat f_m)  + 2 {\hat \lambda}^{cj} \pen_{\tiny \mbox{shape}} (m)\right\} .$$
\end{enumerate}

\subsection{Exploration strategy}
\label{sec:exploration}

The exploration of all possible model classes 
$M_r^T(\Hc)$ (or $M_{r,\Lambda}^T(\Hc)$) with a complexity bounded by some $c$ is intractable since the number of such models is exponential in the number of variables $d$. 
Therefore, strategies should be introduced to propose a set of candidate model classes $M_m$, $m\in \Mc$. \\
 
In practice, a possible approach is to rely on adaptive learning algorithms from 
 \cite{Grelier:2018} (see also \cite{Grelier:2019}) that generate predictors $\hat f_m$ (minimizing the empirical risk) in a sequence of model classes $M_r^T(\Hc)$ . \added{Note that developing an exploration strategy for sparse tensor networks is more challenging.}

\subsubsection{Fixed tree}\label{sec:exploration-fixed-tree}
For a fixed tree $T$ and fixed feature space $V$, the proposed algorithm  generates a sequence of model classes $M_m = M^{T}_{r_m}(\Hc_m)$ with increasing ranks $r_m$, $m\ge 1$, by successively increasing the $\alpha$-ranks for nodes $\alpha$ associated with the highest (estimated) truncation errors  
$$
\inf_{\rank_\alpha(f)\le r_{m,\alpha}} \Rc(f) - \Rc(f^\star).
$$
\modif{For the strategy described in 
\Cref{sec:tensorization} with features based on tensorization, different resolutions $L_m \in \Nbb$ are explored. To each 
resolution $L$ corresponds a fixed tree $T = T_{L}$ and a fixed feature space $\mathbf{V}_{L}$. For each fixed resolution, the above   strategy can then be used to explore the set of possible ranks. }

\modif{A more classical approach (not using the tensorization technique) is to consider 
variable feature spaces $\Hc_m$ of the form }
$\Hc_m := \Hc_{N_m} =  \Hc_{1,N_{m,1}} \otimes \hdots \otimes \Hc_{d,N_{m,d}}$, where for each  dimension $\nu\in \{1,\hdots,d\}$,  $(\Hc_{\nu,k})_{k \in \Nbb}$ is a given approximation tool  (e.g., polynomials, wavelets).  Exploring all possible 
tuples $N_m \in \Nbb^d$ is again a combinatorial problem.  The algorithm proposed in \cite{Grelier:2018,Grelier:2019}  relies on a validation approach for the selection of a particular tuple. Note that 
a complexity-based model selection method could also be considered for the selection of a tuple $N_m$. 
\bigskip

\subsubsection{Variable tree}\label{sec:exploration-variable-tree}
Although the set of possible dimension trees over  $\{1,\hdots,d\}$ is finite, exploring this whole set of dimension trees is intractable for high and even moderate $d$. 
In  \cite{Grelier:2018}, a stochastic algorithm has been proposed for optimizing the dimension tree for the compression of a tensor. 
This tree optimization algorithm has been combined with the rank-adaptive strategy discussed above. The resulting algorithm  generates a sequence of predictors in tree tensor networks associated with different trees. 
In the numerical experiments, we use this  learning algorithm with tree adaptation to generate a set of candidate trees. Then the learning algorithm with rank adaptation but fixed tree is used with each of these trees. \modif{Note that this strategy provides a data-dependent collection of candidate trees. For our model selection results to remain valid, we could use a standard splitting strategy (one part of the data to identify a collection of candidate trees and the other part for the model selection strategy within this collection). Without splitting, a more advanced analysis is necessary to provide risks bounds for a model collection generated with the sample used for the model selection.}
 
In the next section we present some numerical experiments that validate the proposed model selection method and the exploration strategy.

\section{Numerical experiments}\label{sec:results}

In this section, we illustrate the proposed  model selection approach for supervised learning problems in a least-squares regression setting.  $Y$ is a real-valued random variable defined by
$$
Y = f^\star(X) + \varepsilon
$$ 
where $\varepsilon$ is independent of $X$ and has zero mean and  standard deviation $\gamma \sigma(f^\star(X))$.  The parameter $\gamma$ therefore controls the noise level in relative precision. 
 \\\par
 For a given training sample, we use the learning strategies described in \Cref{sec:exploration} that generate a sequence of predictors $\hat f_m$, $m\in \Mc$, associated with a certain collection of models $\Mc$ (which depends on the training sample). 
Given a set of predictors $\hat f_m$, $m\in \Mc$, we denote by $\hat m^\star$ the index of the model that minimizes the risk over $\Mc$, i.e. 
$$
\hat m^\star \in \arg\min_{m\in \Mc} \Rc(\hat f_m).
$$
 The model $\hat m^\star $ is the oracle model in $\Mc$ for a given training sample.

We also denote by $\hat m(\lambda)$ the model such that 
$$
\hat m(\lambda) \in \mathrm{argmin}_{m \in \mathcal{M}} \left\{  \widehat \Rc_n(\hat f_m)  +  \lambda \pen_{\tiny \mbox{shape}} (m) \right\},
$$
where $\pen_{\tiny \mbox{shape}} (m) = C_m / n$,  and by $\hat m = \hat m(2\hat \lambda^{cj})$ the model selected by our model selection strategy, where $\hat \lambda^{cj}$ is calibrated with the complexity jump algorithm (see \Cref{sec:slope-heuristics}). 
\\\par
We consider two different types of problems:  
the approximation of univariate functions defined on $(0,1)$, identified with a multivariate function through tensorization (\Cref{sec:ex-tensorized}), and the  
 approximation of multivariate functions defined on a subset of $\Rbb^d$ using classical feature tensor spaces (\Cref{sec:ex-multivariate}).
  \\\par  For a given function $f$, the risk $\Rc(f)$ is evaluated using a sample of size $10^5$ independent of the training sample.  Statistics of complexities and risks (such as the expected complexity $\Ebb(C_{\hat m})$  or the expected risk $\Ebb(\Rc(\hat f_{\hat m}))$) are computed using 20 different training samples.

\subsection{Tensorized function}\label{sec:ex-tensorized}
Here we consider tensor networks for the approximation of a univariate function in $L^2(0,1)$ \modif{using the tensorization approach introduced in \Cref{sec:tensorization} with $b=2$, that we briefly recap. 
}
A function $f$ defined on $[0,1)$ is linearly identified with a function $\boldsymbol{f} = \mathcal{T}_{L}(f)$ of $L+1$ variables defined on $\{0,1\}^L \times (0,1) $ such that 
$$
f(x) = \mathcal{T}_L(f)(i_1,\hdots,i_{L} , \bar x) \quad  \text{for}  \quad x = 2^{-L}(\sum_{k=1}^{L} i_k 2^{L-k} + \bar x).
$$
The map $\mathcal{T}_{L}$ is the tensorization map at resolution $L$. 
This allows to isometrically identify the space $L^2(0,1)$ with  the tensor space $\Rbb^2 \otimes \hdots \otimes \Rbb^2\otimes L^2(0,1)$ of order $L+1$. Then we consider the approximation space 
$ \mathbf{V}_L = \Rbb^2 \otimes \hdots \otimes \Rbb^2\otimes \Pbb_0$ of $(L+1)$-variate functions $\boldsymbol{f}(i_1,\hdots,i_{L},\bar x)$ independent of the variable $\bar x$. The space $ \mathbf{V}_L$ is linearly identified with the space 
of  piecewise constant functions on the uniform partition of $[0,1)$ into $2^L$ intervals.  
Then we consider model classes $M = \mathcal{T}_L^{-1} M_{r}^{T}(\mathbf{V}_L)$, 
which are piecewise constant functions $f$ whose tensorized version $\mathcal{T}_L(f)$ is in a particular tree-based tensor format. \bigskip 

In the following experiments, for each $L \in \{1,\hdots,12\}$, we consider a fixed linear binary tree $T = T_L$ (with interior nodes $\{1,\hdots,k\}$, $1\le k\le L+1$) and use the rank adaptive learning algorithm described in \Cref{sec:exploration-fixed-tree} to produce a sequence of $25$ approximations with increasing ranks.
 \\\par
 
 Three functions $f^\star(x) $ are considered. The first function $f^\star(x) = \sqrt{x}$ is analytic on the open interval $(0,1)$ and its  derivative has a singularity at zero.  The second function $f^\star(x) = \frac{1}{1+x}$ is analytic on a larger interval including $[0,1]$. The third function is in the Sobolev space $H^2(0,1).$
 For all functions, the proposed model selection approach shows a very good performance.
 It selects with high probability a model with a risk very close to the risk of the oracle $\hat f_{\hat m^\star}.$  
 
\subsubsection{Tensorized function $f^\star(x) = \sqrt{x}$}
We consider the function $f^\star(x) = \sqrt{x}$ which is analytic on the open interval $(0,1)$, with a singular derivative at zero. We observe on \Cref{fig:f1-n200,fig:f1-n1000} that the  model selection approach selects a model close to optimal for different sample size $n$ and noise level. \Cref{tab:f1-n,tab:f1-gamma} 
show expectations of complexities and errors for the selected estimator and illustrate the very good performance of the approach when compared to the oracle. 
\begin{figure}[h]\centering
\begin{subfigure}{.45\textwidth}\centering
\includegraphics[width=.8\textwidth]{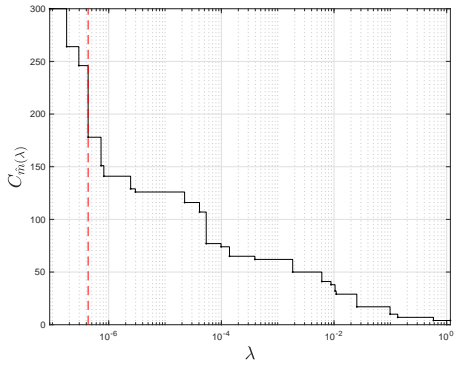}
\caption{Function $\lambda\mapsto C_{\hat m(\lambda)}$, $\lambda^{cj}$ (red).}
\end{subfigure}\hspace{.05\textwidth}
\begin{subfigure}{.45\textwidth}\centering
\includegraphics[width=.8\textwidth]{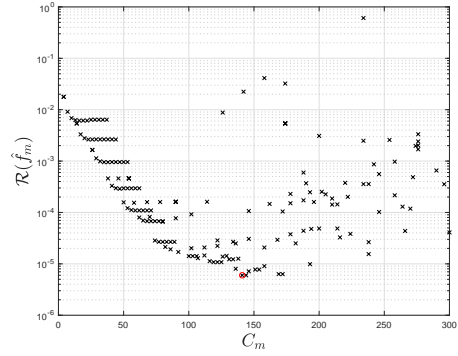}
\caption{Points $(C_m , \Rc(\hat f_m )$, ${m\in \Mc}$, and selected model (red).}
\end{subfigure}
\caption{Slope heuristics for the tensorized function $f^\star(x) = \sqrt{x}$ with $n=200$ and $\gamma= 0.001$.}
\label{fig:f1-n200}
\end{figure}
\begin{figure}[h]\centering
\begin{subfigure}{.45\textwidth}\centering
\includegraphics[width=.8\textwidth]{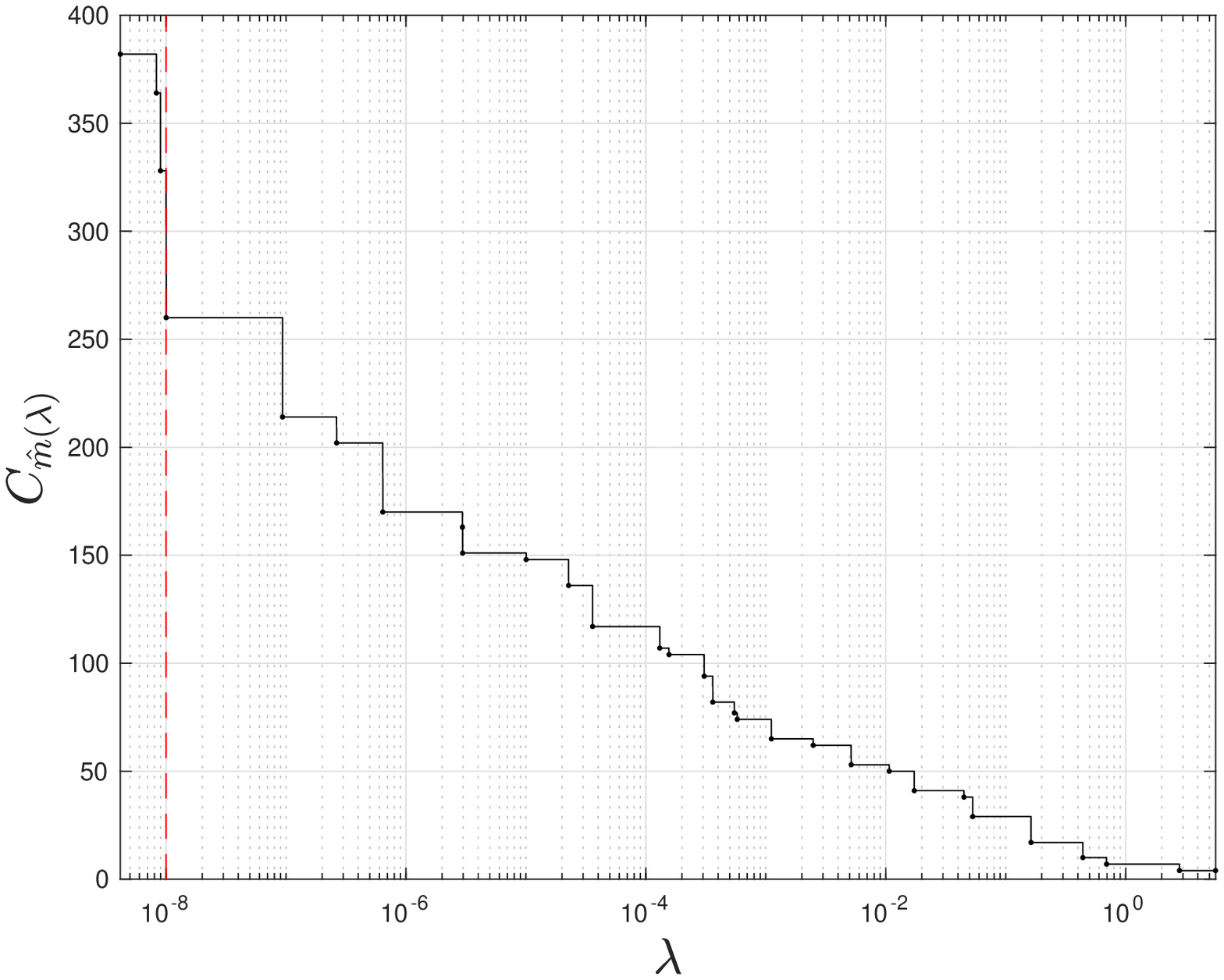}
\caption{Function $\lambda\mapsto C_{\hat m(\lambda)}$, $\lambda^{cj}$ (red).}
\end{subfigure}\hspace{.05\textwidth}
\begin{subfigure}{.45\textwidth}\centering
\includegraphics[width=.8\textwidth]{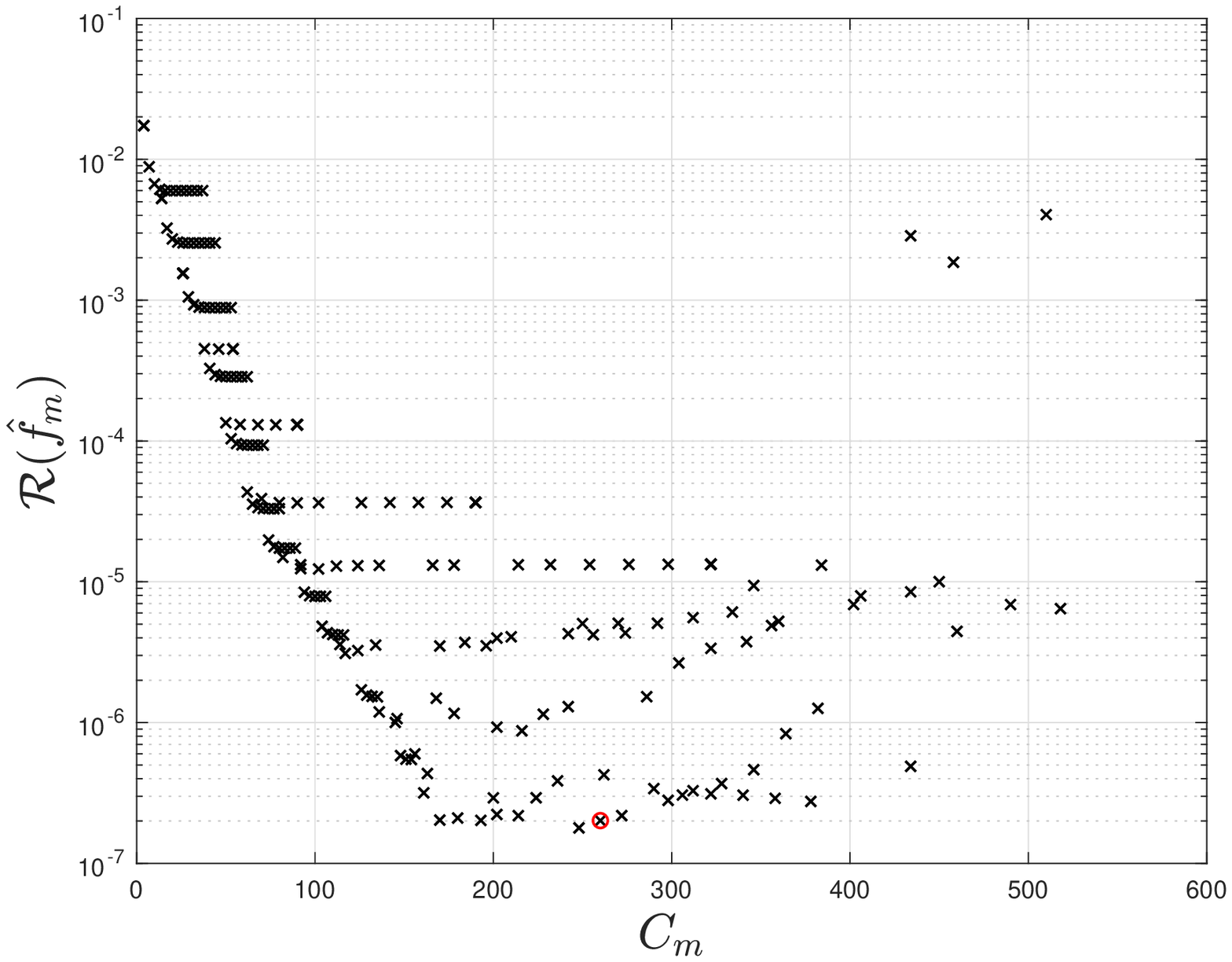}
\caption{Points $(C_m , \Rc(\hat f_m )$, ${m\in \Mc}$, and selected model (red).}
\end{subfigure}
\caption{Slope heuristics for the tensorized function $f^\star(x) = \sqrt{x}$ with $n=1000$ and $\gamma= 0.0001$.}
\label{fig:f1-n1000}
\end{figure}

\begin{table}[h]
\centering
\begin{tabular}{|c|c||c|c||c|c|}
 \hline 
 $n$ & $\mathbb{E}(C_{\hat m^\star})$  & $\mathbb{E}(C_{\hat m})$  & $\mathbb{E}(\mathcal{R}(\hat f_{m^\star}))$&  $\mathbb{E}(\mathcal{R}(\hat f_{\hat m}))$ \\ \hline 
 100 & 123.2 &  91.6 & 1.6e-05  & 5.0e-05 \\ 
 200 & 163.8 & 165.0 & 3.0e-06  & 5.1e-06 \\
500 & 182.2 & 182.6 & 9.2e-07  & 1.2e-06\\
1000 & 190.2 & 228.5 & 7.1e-07  & 1.4e-06 \\ 
\hline\end{tabular}
\caption{Expectation of complexities and risks of the model selected by the slope heuristics, with the function $f^\star(x) = \sqrt{x}$ and different values of $n$ and $\gamma= 0.001$.}\label{tab:f1-n}
\end{table}
\begin{table}[h]
\centering
\begin{tabular}{|c|c||c|c||c|c|}
 \hline 
 $\gamma$ & $\mathbb{E}(C_{\hat m^\star})$  & $\mathbb{E}(C_{\hat m})$  & $\mathbb{E}(\mathcal{R}(\hat f_{m^\star}))$&  $\mathbb{E}(\mathcal{R}(\hat f_{\hat m}))$  \\ \hline 
 $10^{-3}$ & 190.2 & 228.5 & 7.1e-07  & 1.4e-06 \\ 
 $10^{-4}$ & 242.8 & 251.4 & 1.5e-07  & 2.1e-07 \\ 
  $10^{-5}$ & 219.8 & 267.4 & 1.3e-07  & 2.4e-07 \\
 $0$ & 218.6 & 258.6 & 1.1e-07  & 2.1e-07 \\ 
\hline\end{tabular}
\caption{Expectation of complexities and risks of the model selected by the slope heuristics, with the function $f^\star(x) = \sqrt{x}$ and different values of $\gamma$ and $n=1000$.}\label{tab:f1-gamma}
\end{table}


\clearpage

\subsubsection{Tensorized function $f^\star(x) = \frac{1}{1+x}$.}
We consider the  function $f^\star(x) = \frac{1}{1+x}$ which is analytic on the interval $(-1,\infty)$ including $[0,1]$. 
\Cref{fig:f2-n200,fig:f2-n1000} illustrate the good behaviour of the model selection approach for different sample size and noise level. 
 \Cref{tab:f2-n,tab:f2-gamma} 
show expectations of complexities and errors for the selected estimator and illustrate again the very good performance of the approach when compared to the oracle. 

\begin{figure}[h]\centering
\begin{subfigure}{.45\textwidth}\centering
\includegraphics[width=.8\textwidth]{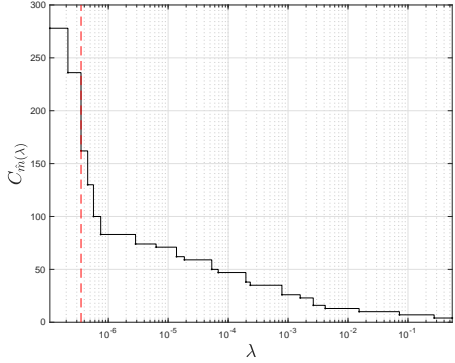}
\caption{Function $\lambda\mapsto C_{\hat m(\lambda)}$, $\lambda^{cj}$ (red).}
\end{subfigure}\hspace{.05\textwidth}
\begin{subfigure}{.45\textwidth}\centering
\includegraphics[width=.8\textwidth]{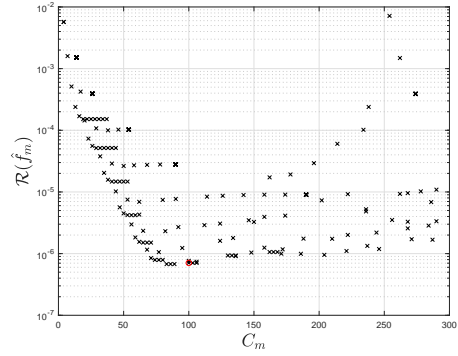}
\caption{Points $(C_m , \Rc(\hat f_m )$, ${m\in \Mc}$, and selected model (red).}
\end{subfigure}
\caption{Slope heuristics for the tensorized function $f^\star(x) = \frac{1}{1+x}$ with $n=200$ and $\gamma= 0.001$.}
\label{fig:f2-n200}
\end{figure}
\begin{figure}[h]\centering
\begin{subfigure}{.45\textwidth}\centering
\includegraphics[width=.8\textwidth]{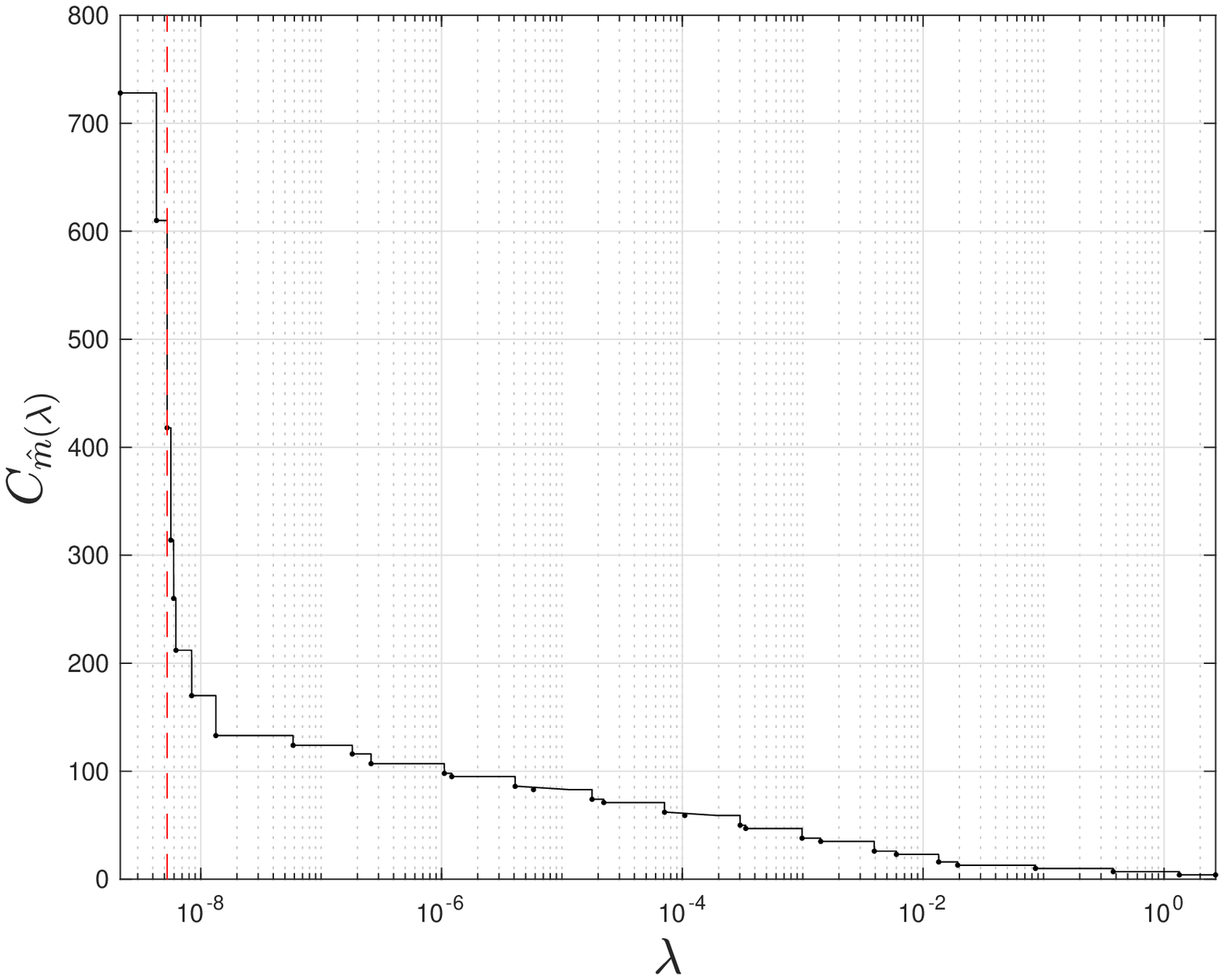}
\caption{Function $\lambda\mapsto C_{\hat m(\lambda)}$, $\lambda^{cj}$ (red).}
\end{subfigure}\hspace{.05\textwidth}
\begin{subfigure}{.45\textwidth}\centering
\includegraphics[width=.8\textwidth]{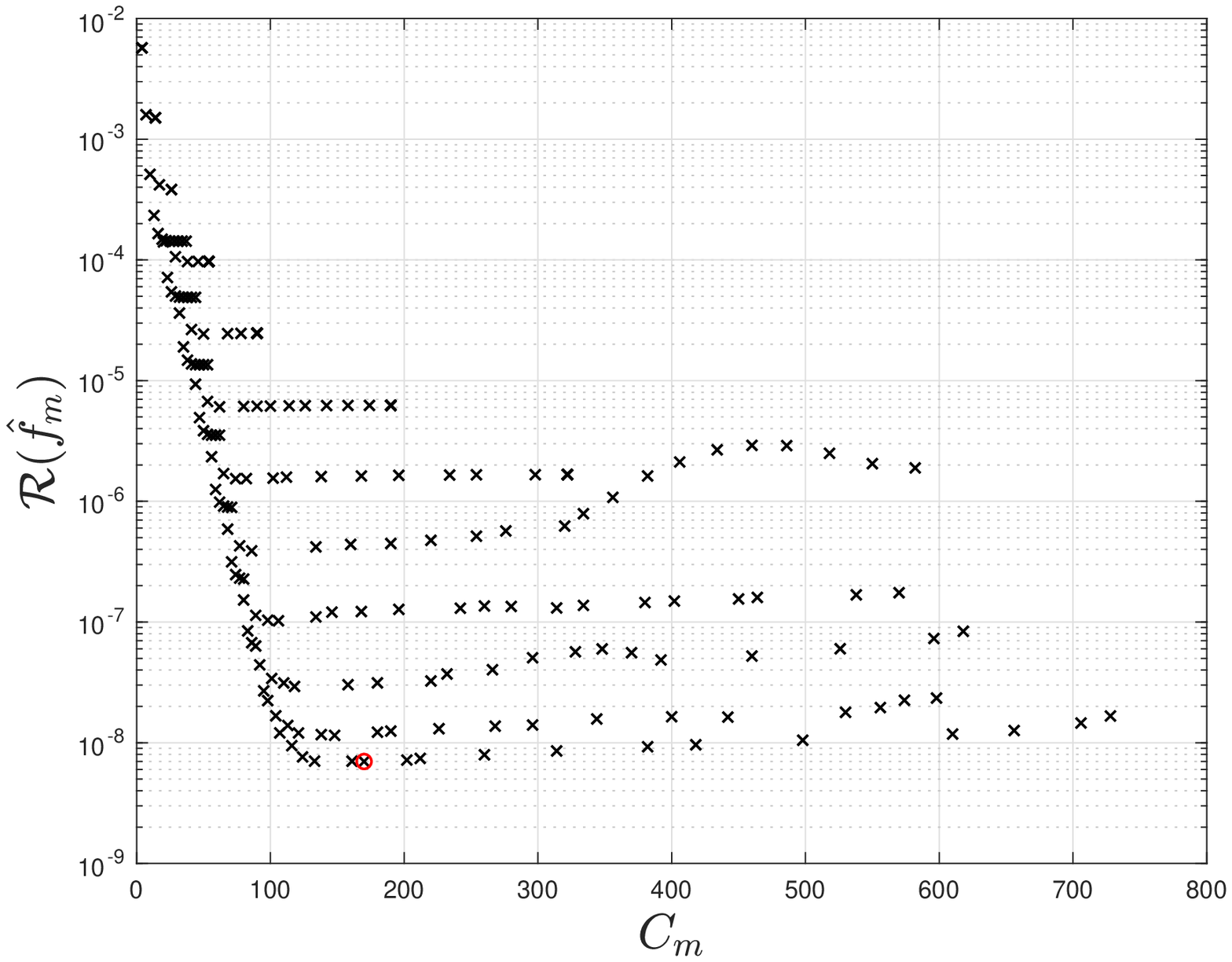}
\caption{Points $(C_m , \Rc(\hat f_m )$, ${m\in \Mc}$, and selected model (red).}
\end{subfigure}
\caption{Slope heuristics for the tensorized function $f^\star(x) = \frac{1}{1+x}$ with $n=1000$ and $\gamma= 0.0001$.}
\label{fig:f2-n1000}
\end{figure}

\begin{table}[h]
\centering
\begin{tabular}{|c|c||c|c||c|c|}
 \hline 
 $n$ & $\mathbb{E}(C_{\hat m^\star})$  & $\mathbb{E}(C_{\hat m})$  & $\mathbb{E}(\mathcal{R}(\hat f_{m^\star}))$&  $\mathbb{E}(\mathcal{R}(\hat f_{\hat m}))$ \\ \hline 
 100 &  88.0 &  83.0 & 9.3e-07  & 1.0e-06 \\
 200 &  97.3 &  92.8 & 6.4e-07  & 6.6e-07 \\ 
500 &  92.9 & 124.4 & 5.8e-07  & 6.9e-07 \\ 
1000 &  108.4 & 107.5 & 5.3e-07  & 5.3e-07  \\
\hline\end{tabular}
\caption{Expectation of complexities and risks of the model selected by the slope heuristics, with the function $f^\star(x) = \frac{1}{1+x}$, different values of $n$ and $\gamma= 0.001$.}
\label{tab:f2-n} 
\end{table}
\begin{table}[h]
\centering
\begin{tabular}{|c|c||c|c||c|c|}
 \hline 
 $\gamma$ & $\mathbb{E}(C_{\hat m^\star})$  & $\mathbb{E}(C_{\hat m})$  & $\mathbb{E}(\mathcal{R}(\hat f_{m^\star}))$&  $\mathbb{E}(\mathcal{R}(\hat f_{\hat m}))$ \\ \hline 
 $10^{-3}$ & 108.4 & 107.5 & 5.3e-07  & 5.3e-07  \\
 $10^{-4}$ & 159.3 & 151.1 & 6.9e-09  & 6.9e-09 \\ 
$10^{-5}$ & 152.0 & 182.2 & 1.6e-09  & 1.9e-09 \\ 
$0$  & 156.8 & 155.8 & 1.6e-09  & 1.6e-09 \\
\hline\end{tabular}
\caption{Expectation of complexities and risks of the model selected by the slope heuristics, with the function $f^\star(x) = \frac{1}{1+x}$, different values of $\gamma$ and $n=1000$.}
\label{tab:f2-gamma} 
\end{table}


\clearpage

\subsubsection{Tensorized function $f^\star(x) = g(g(x))^2$ with $g(x) = 1-2 \vert x-\frac{1}{2} \vert$.}
We consider the function $f^\star(x) = g(g(x))^2$ with $g(x) = 1-2 \vert x-\frac{1}{2} \vert$, which is in the 
Sobolev space $H^2(0,1).$
 \Cref{fig:f3-n200,fig:f3-n1000} illustrate again the good behaviour of the model selection approach for different sample size and noise level. 
And  \Cref{tab:f2-n,tab:f2-gamma} again illustrate again the very good performance (in expectation) for the selected estimator of the approach when compared to the oracle. 
 
\begin{figure}[h]\centering
\begin{subfigure}{.45\textwidth}\centering
\includegraphics[width=.8\textwidth]{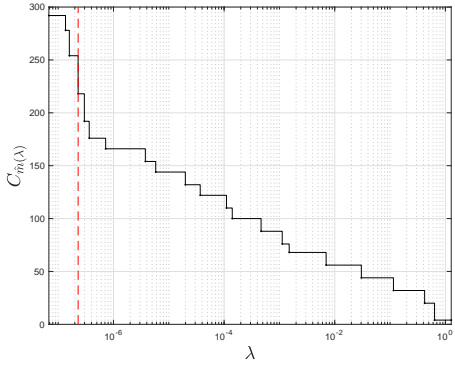}
\caption{Function $\lambda\mapsto C_{\hat m(\lambda)}$, $\lambda^{cj}$ (red).}
\end{subfigure}\hspace{.05\textwidth}
\begin{subfigure}{.45\textwidth}\centering
\includegraphics[width=.8\textwidth]{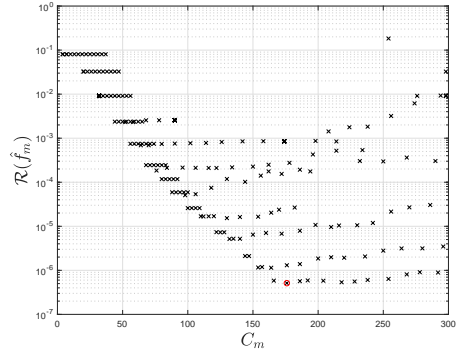}
\caption{Points $(C_m , \Rc(\hat f_m )$, ${m\in \Mc}$, and selected model (red).}
\label{fig:f3-n200}
\end{subfigure}
\caption{Slope heuristics for the tensorized function $f^\star(x) = (g(g(x)))^2$ with $n=200$ and $\gamma= 0.001$.}
\end{figure}
\begin{figure}[h]\centering
\begin{subfigure}{.45\textwidth}\centering
\includegraphics[width=.8\textwidth]{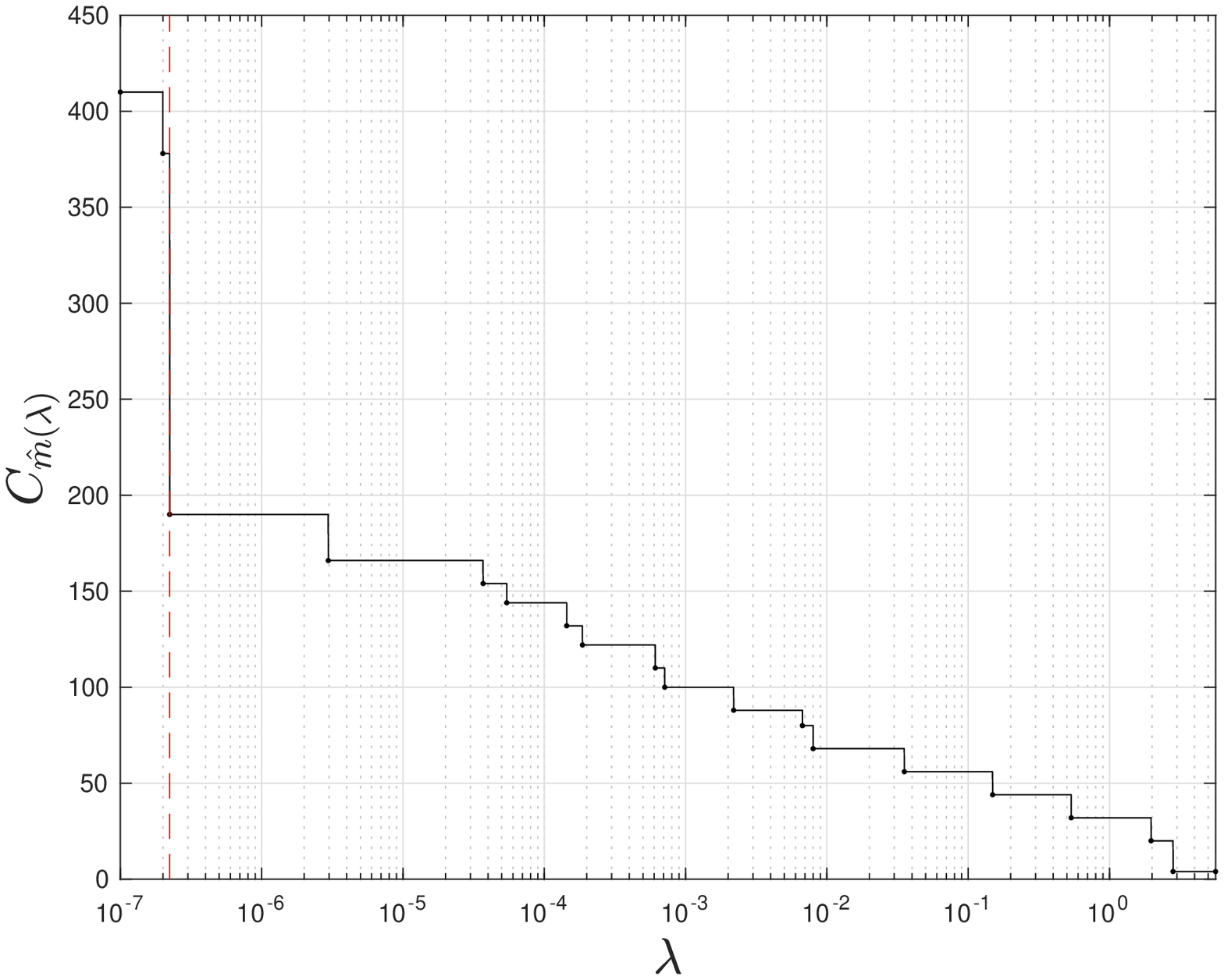}
\caption{Function $\lambda\mapsto C_{\hat m(\lambda)}$, $\lambda^{cj}$ (red).}
\end{subfigure}\hspace{.05\textwidth}
\begin{subfigure}{.45\textwidth}\centering
\includegraphics[width=.8\textwidth]{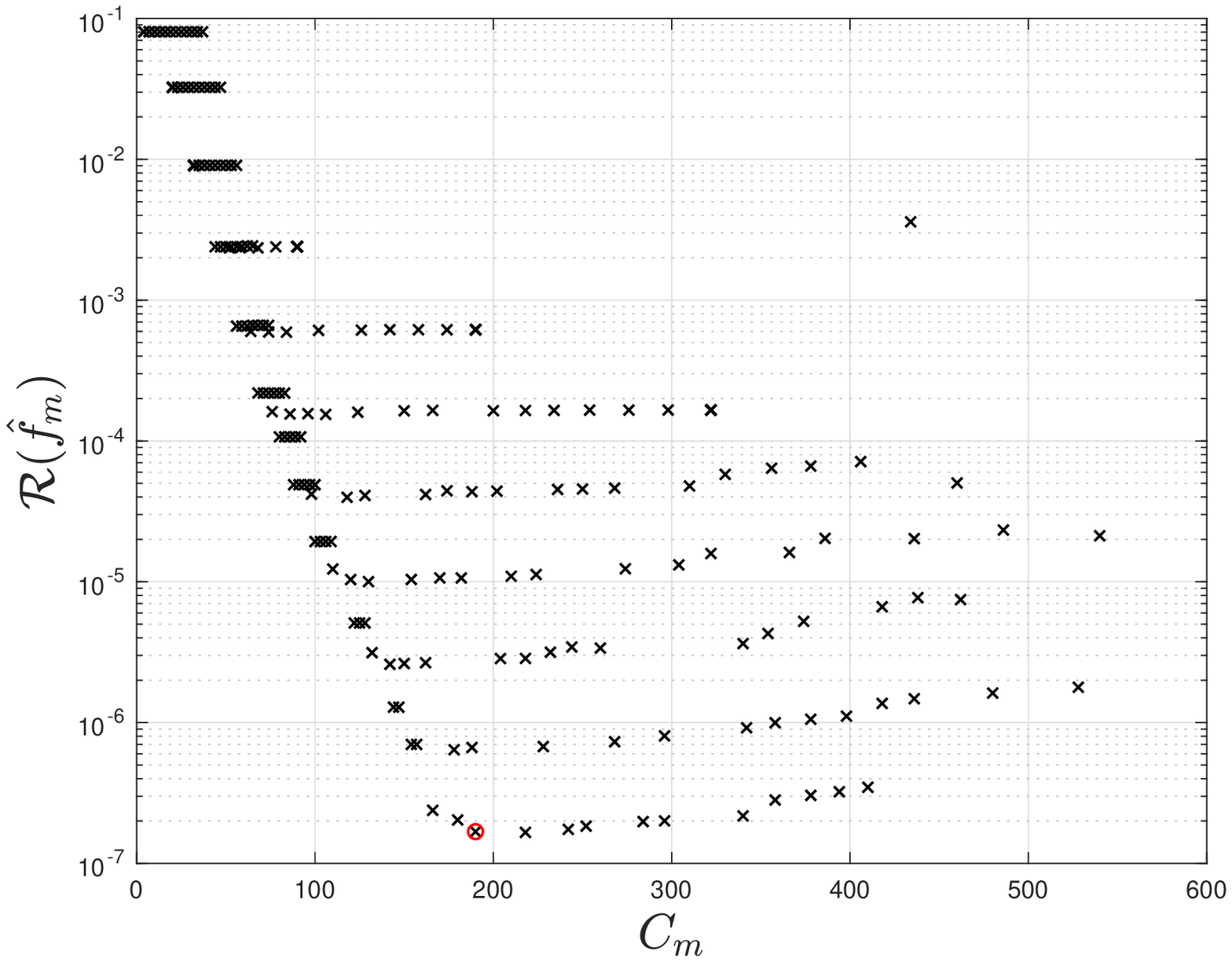}
\caption{Points $(C_m , \Rc(\hat f_m )$, ${m\in \Mc}$, and selected model (red).}
\end{subfigure}
\caption{Slope heuristics for the tensorized function $f^\star(x) = (g(g(x)))^2$ with $n=1000$ and $\gamma= 0.0001$.}
\label{fig:f3-n1000}
\end{figure}

\begin{table}[h]
\centering
\begin{tabular}{|c||c|c||c|c|}
 \hline 
n & $\mathbb{E}(C_{\hat m^\star})$  & $\mathbb{E}(C_{\hat m})$  & $\mathbb{E}(\mathcal{R}(\hat f_{m^\star}))$&  $\mathbb{E}(\mathcal{R}(\hat f_{\hat m}))$ \\ \hline 
200 & 176.4 & 181.6 & 6.3e-07  & 1.6e-06 \\ 
500 & 188.2 & 198.8 & 3.9e-07  & 4.1e-07 \\
1000 & 196.6 & 233.8 & 3.2e-07  & 3.5e-07 \\
\hline\end{tabular}
\caption{Expectation of complexities and risks for the function $f^\star(x) = (g(g(x)))^2$, different values of $n$ and $\gamma= 0.001$.}\label{tab:f3-n} 
\end{table}

\begin{table}[h]
\centering
\begin{tabular}{|c|c||c|c||c|c|}
 \hline 
 $\gamma$ & $\mathbb{E}(C_{\hat m^\star})$  & $\mathbb{E}(C_{\hat m})$  & $\mathbb{E}(\mathcal{R}(\hat f_{m^\star}))$&  $\mathbb{E}(\mathcal{R}(\hat f_{\hat m}))$ \\ \hline 
 $10^{-3}$ & 196.6 & 233.8 & 3.2e-07  & 3.5e-07 \\  
 $10^{-4}$ &  195.8 & 205.8 & 1.7e-07  & 1.7e-07 \\
 $10^{-5}$  & 191.0 & 226.6 & 1.7e-07  & 1.8e-07 \\ 
$0$ & 194.0 & 232.6 & 1.7e-07  & 1.9e-07 \\
\hline
\hline\end{tabular}
\caption{Expectation of complexities and risks of the model selected by the slope heuristics, with the function $f^\star(x) = (g(g(x)))^2$, different values of $\gamma$ and $n=1000$.}\label{tab:f3-gamma} 
\end{table}

%
%

 \subsection{Multivariate functions}\label{sec:ex-multivariate}

\subsubsection{Corner peak function}
We consider the function
$$
f^\star(X) = \frac{1}{1 + \sum_{\nu=1}^d \nu^{-2} X_\nu}
$$
with $d=10$, where the $X_\nu \sim U(0,1)$ are i.i.d. uniform random variables. The function $f^\star$ is analytic on $[0,1]^d$. We use the fixed balanced binary tree $T$ of \Cref{fig:corner-peak-dimension-tree}. \modif{As univariate approximation tools, we use polynomial spaces $\Hc_{\nu,N_\nu} = \Pbb_{N_\nu-1}(\Xc_\nu)$, $\nu \in D$. }
\Cref{fig:corner-peak-n1000,fig:corner-peak-n1000-10samples} illustrate the very good behaviour of the model selection approach for a sample size $n=1000$ and noise level $\gamma=0.001$, where the best model appears to be always selected. 
In \Cref{tab:corner-peak-n,tab:corner-peak-gamma}, we observe  that 
the expectation of complexities and errors for the selected estimator (for different values of $n$ and $\gamma$), which are of the are of the same order as for the oracle. 

\begin{figure}[h]
\centering \footnotesize
 \begin{tikzpicture}[scale=.5, level distance = 20mm]  \tikzstyle{level 1}=[sibling distance=75mm]  \tikzstyle{level 2}=[sibling distance=38mm]  \tikzstyle{level 3}=[sibling distance=19mm]  \tikzstyle{level 4}=[sibling distance=10mm]  \tikzstyle{root}=[circle,draw,thick,fill=black,scale=.8]  \tikzstyle{interior}=[circle,draw,solid,thick,fill=black,scale=.8]  \tikzstyle{leaves}=[circle,draw,solid,thick,fill=black,scale=.8]  \node [root,label=above:{$\{1,2,3,4,5,6,7,8,9,10\}$}]  {} child{node [interior,label=above:{$\{1,2,7,8,9,10\}$}] {}child{node [interior,label=above:{$\{7,8,9,10\}$}] {}child{node [interior,label=above:{$\{7,8\}$}] {}child{node [leaves,label=below:{$\{7\}$}] {}}child{node [leaves,label=below:{$\{8\}$}] {}}}child{node [interior,label=above:{$\{9,10\}$}] {}child{node [leaves,label=below:{$\{9\}$}] {}}child{node [leaves,label=below:{$\{10\}$}] {}}}}child{node [interior,label=above:{$\{1,2\}$}] {}child{node [leaves,label=below:{$\{1\}$}] {}}child{node [leaves,label=below:{$\{2\}$}] {}}}}child{node [interior,label=above:{$\{3,4,5,6\}$}] {}child{node [interior,label=above:{$\{3,4\}$}] {}child{node [leaves,label=below:{$\{3\}$}] {}}child{node [leaves,label=below:{$\{4\}$}] {}}}child{node [interior,label=above:{$\{5,6\}$}] {}child{node [leaves,label=below:{$\{5\}$}] {}}child{node [leaves,label=below:{$\{6\}$}] {}}}};\end{tikzpicture}
 \caption{Corner peak function. Dimension tree $T$.}
\label{fig:corner-peak-dimension-tree}
\end{figure}
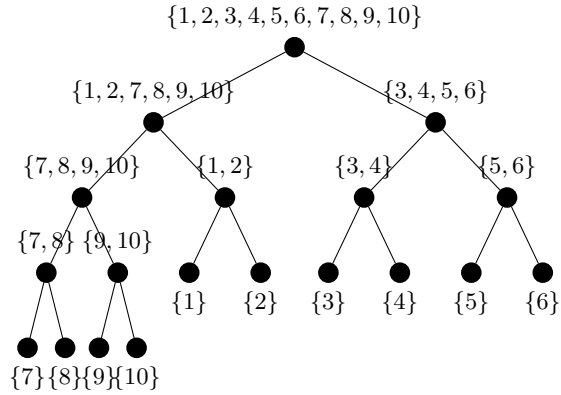

\begin{figure}[h]\centering
\begin{subfigure}{.45\textwidth}\centering
\includegraphics[width=.8\textwidth]{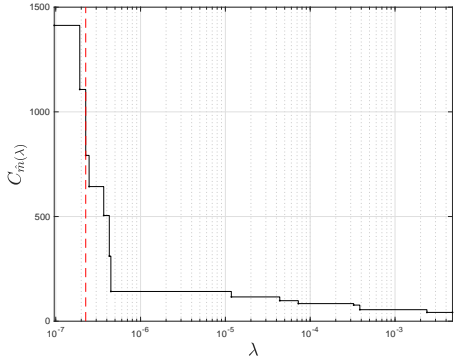}
\caption{Function $\lambda\mapsto C_{\hat m(\lambda)}$, $\lambda^{cj}$ (red).}
\end{subfigure}  \hspace{.05\textwidth}
\begin{subfigure}{.45\textwidth}\centering
\includegraphics[width=.8\textwidth]{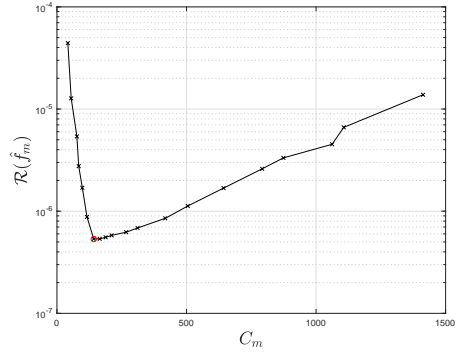}
\caption{Function $C_m \mapsto \Rc(\hat f_m )$ and selected model (red).}
\end{subfigure}
\caption{Slope heuristics for the Corner peak function with $n=1000$ and $\gamma= 0.001$.}
\label{fig:corner-peak-n1000}
\end{figure}

\begin{figure}[h]\centering
\begin{subfigure}{.45\textwidth}\centering
\includegraphics[width=.8\textwidth]{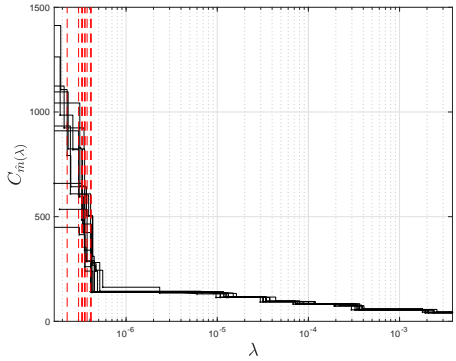}
\caption{Functions $\lambda\mapsto C_{\hat m(\lambda)}$, $\lambda^{cj}$ (red).}
\end{subfigure}  \hspace{.05\textwidth}
\begin{subfigure}{.45\textwidth}\centering
\includegraphics[width=.8\textwidth]{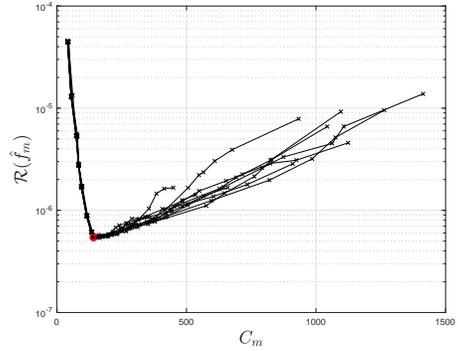}
\caption{Functions $C_m \mapsto \Rc(\hat f_m )$ and selected model (red).}
\end{subfigure}
\caption{Slope heuristics for the Corner peak function with $n=1000$ and $\gamma= 0.001$, superposition of $10$ different samples.}\label{fig:corner-peak-n1000-10samples}
\end{figure}

\begin{table}[h]
\centering
\begin{tabular}{|c|c||c|c||c|c|}
 \hline 
 $n$ & $\mathbb{E}(C_{\hat m^\star})$  & $\mathbb{E}(C_{\hat m})$  & $\mathbb{E}(\mathcal{R}(\hat f_{m^\star}))$&  $\mathbb{E}(\mathcal{R}(\hat f_{\hat m}))$ \\ \hline 
100 &  124.1 &  73.7 & 2.1e-06  & 1.1e-05 \\ 
500 & 286.7 & 291.3 & 9.8e-11  & 1.0e-10 \\ 
1000 &  286.2 & 293.8 & 6.6e-11  & 6.7e-11 \\
\hline\end{tabular}
\caption{Expectation of complexities and risks selected by the slope heuristics, with the Corner peak function, different values of $n$ and $\gamma= 10^{-5}$.}
\label{tab:corner-peak-n}
\end{table}

\begin{table}[h]
\centering
\begin{tabular}{|c|c||c|c||c|c|}
 \hline 
 $\gamma$ & $\mathbb{E}(C_{\hat m^\star})$  & $\mathbb{E}(C_{\hat m})$  & $\mathbb{E}(\mathcal{R}(\hat f_{m^\star}))$&  $\mathbb{E}(\mathcal{R}(\hat f_{\hat m}))$   \\ \hline 
 $10^{-2}$ &  95.5 &  79.8 & 5.4e-05  & 5.5e-05 \\
 $10^{-3}$ & 143.1 & 143.1 & 5.4e-07  & 5.4e-07 \\ 
  $10^{-4}$ & 223.2 & 193.7 & 5.9e-09  & 6.0e-09 \\ 
 $10^{-5}$ & 286.2 & 293.8 & 6.6e-11  & 6.7e-11 \\
 0 & 598.7 & 538.4 & 2.5e-15  & 1.8e-14 \\ 
\hline\end{tabular}
\caption{Expectation of complexities and risks selected by the slope heuristics, with the Corner peak function, different values of $\gamma$ and $n=1000$.}
\label{tab:corner-peak-gamma}
\end{table}

%
%
%

\clearpage

\subsubsection{Borehole function}
We consider the function 
\begin{align*}
g(U_1,\hdots,U_8)= \frac{2\pi U_3(U_4-U_6)}{(U_2-\log(U_1)) (1+\frac{2 U_7 U_3}{(U_2-\log(U_1))U_1^2  U_8}+
    \frac{U_3}{U_5})}
\end{align*}
which models the water flow through a borehole as a function of $8$ independent random variables $U_1 \sim \Nc(0.1,0.0161812) $, 
$U_2 \sim \Nc(7.71, 1.0056)$, 
$U_3 \sim U(63070,115600)$,
$U_4\sim U(990,1110)$,
$U_5\sim U(63.1,116)$,
$U_6\sim  U(700,820)$,
$U_7\sim U(1120,1680)$,
$U_8\sim U(9855,12045)$.
Then we consider the function 
$$
f^\star(X_1,\hdots,X_d) = g(g_1(X_1),\hdots,g_8(X_8)),$$
where $g_\nu$ are functions such that $U_\nu = g_\nu(X_\nu)$, with $X_\nu  \sim \Nc(0,1)$ for $\nu\in \{1,2\}$, and $X_\nu\sim U(-1,1)$ for $\nu\in \{3,\hdots,8\}.$ Function $f^\star$ is thus defined on  $\Xc= \Rbb^2 \times [-1,1]^{6}.$ 
As univariate approximation tools, we use polynomial spaces $\Hc_{\nu,N_\nu} = \Pbb_{N_\nu-1}(\Xc_\nu)$, $\nu \in D$.  

We use the exploration strategy described in \Cref{sec:exploration-fixed-tree}. More precisely, 
we first run a learning algorithm with tree adaptation from an initial binary tree drawn randomly, with $n=100$ samples. The learning algorithm visited the $9$ trees plotted in \Cref{fig:borehole-trees}. 
Then for each of these trees, we start a learning algorithm with fixed tree and rank adaptation. 
\Cref{fig:borehole-n100,fig:borehole-n200,fig:borehole-n1000} illustrate the behaviour of the model selection strategy for different sample size $n$. \Cref{tab:borehole} shows the expectation of complexities and risks. 
 The model selection approach shows very good performances, except for very small training size $n=100$, where the approach selects a model rather far from the optimal one (in terms of expected risk and complexity).
 
 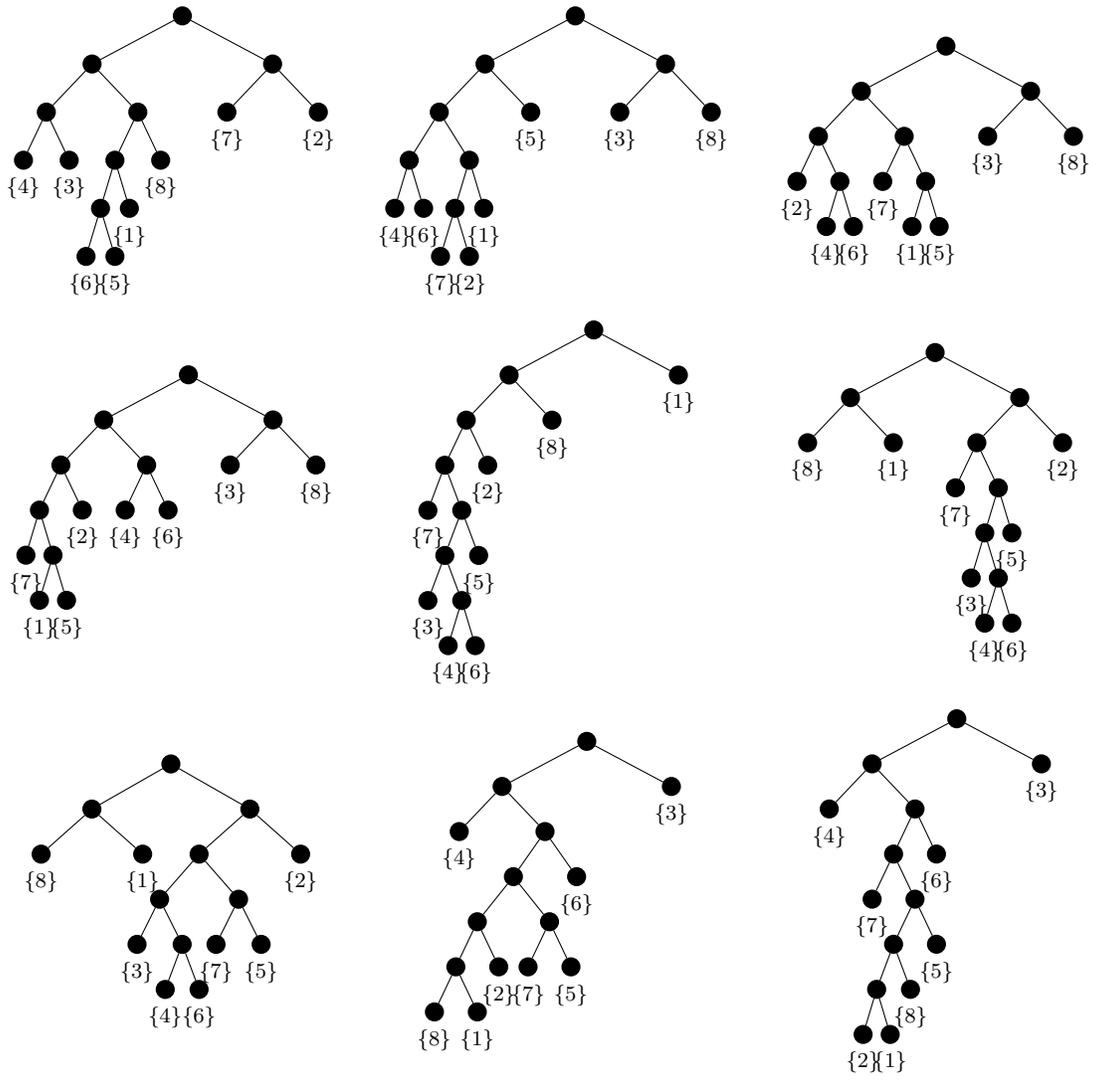
\begin{figure}
\centering
\begin{subfigure}{.30\textwidth}\scriptsize\centering\centering
    \begin{tikzpicture}[scale=0.32, level distance = 20mm]  \tikzstyle{level 1}=[sibling distance=75mm]  \tikzstyle{level 2}=[sibling distance=38mm]  \tikzstyle{level 3}=[sibling distance=19mm]  \tikzstyle{level 4}=[sibling distance=12mm]  \tikzstyle{level 5}=[sibling distance=12mm]  \tikzstyle{root}=[circle,draw,thick,fill=black,scale=.8]  \tikzstyle{interior}=[circle,draw,solid,thick,fill=black,scale=.8]  \tikzstyle{leaves}=[circle,draw,solid,thick,fill=black,scale=.8]  \node [root,label=above:{}]  {} child{node [interior,label=above:{}] {}child{node [interior,label=above:{}] {}child{node [leaves,label=below:{$\{4\}$}] {}}child{node [leaves,label=below:{$\{3\}$}] {}}}child{node [interior,label=above:{}] {}child{node [interior,label=above:{}] {}child{node [interior,label=above:{}] {}child{node [leaves,label=below:{$\{6\}$}] {}}child{node [leaves,label=below:{$\{5\}$}] {}}}child{node [leaves,label=below:{$\{1\}$}] {}}}child{node [leaves,label=below:{$\{8\}$}] {}}}}child{node [interior,label=above:{}] {}child{node [leaves,label=below:{$\{7\}$}] {}}child{node [leaves,label=below:{$\{2\}$}] {}}};\end{tikzpicture}
\end{subfigure}
\begin{subfigure}{.30\textwidth}\scriptsize\centering\centering
    \begin{tikzpicture}[scale=0.32, level distance = 20mm]  \tikzstyle{level 1}=[sibling distance=75mm]  \tikzstyle{level 2}=[sibling distance=38mm]  \tikzstyle{level 3}=[sibling distance=25mm]  \tikzstyle{level 4}=[sibling distance=12mm]  \tikzstyle{level 5}=[sibling distance=12mm]  \tikzstyle{root}=[circle,draw,thick,fill=black,scale=.8]  \tikzstyle{interior}=[circle,draw,solid,thick,fill=black,scale=.8]  \tikzstyle{leaves}=[circle,draw,solid,thick,fill=black,scale=.8]  \node [root,label=above:{}]  {} child{node [interior,label=above:{}] {}child{node [interior,label=above:{}] {}child{node [interior,label=above:{}] {}child{node [leaves,label=below:{$\{4\}$}] {}}child{node [leaves,label=below:{$\{6\}$}] {}}}child{node [interior,label=above:{}] {}child{node [interior,label=above:{}] {}child{node [leaves,label=below:{$\{7\}$}] {}}child{node [leaves,label=below:{$\{2\}$}] {}}}child{node [leaves,label=below:{$\{1\}$}] {}}}}child{node [leaves,label=below:{$\{5\}$}] {}}}child{node [interior,label=above:{}] {}child{node [leaves,label=below:{$\{3\}$}] {}}child{node [leaves,label=below:{$\{8\}$}] {}}};\end{tikzpicture}
\end{subfigure}
\begin{subfigure}{.30\textwidth}\scriptsize\centering\centering
    \begin{tikzpicture}[scale=0.30, level distance = 20mm]  \tikzstyle{level 1}=[sibling distance=75mm]  \tikzstyle{level 2}=[sibling distance=38mm]  \tikzstyle{level 3}=[sibling distance=19mm]  \tikzstyle{level 4}=[sibling distance=12mm]  \tikzstyle{root}=[circle,draw,thick,fill=black,scale=.8]  \tikzstyle{interior}=[circle,draw,solid,thick,fill=black,scale=.8]  \tikzstyle{leaves}=[circle,draw,solid,thick,fill=black,scale=.8]  \node [root,label=above:{}]  {} child{node [interior,label=above:{}] {}child{node [interior,label=above:{}] {}child{node [leaves,label=below:{$\{2\}$}] {}}child{node [interior,label=above:{}] {}child{node [leaves,label=below:{$\{4\}$}] {}}child{node [leaves,label=below:{$\{6\}$}] {}}}}child{node [interior,label=above:{}] {}child{node [leaves,label=below:{$\{7\}$}] {}}child{node [interior,label=above:{}] {}child{node [leaves,label=below:{$\{1\}$}] {}}child{node [leaves,label=below:{$\{5\}$}] {}}}}}child{node [interior,label=above:{}] {}child{node [leaves,label=below:{$\{3\}$}] {}}child{node [leaves,label=below:{$\{8\}$}] {}}};\end{tikzpicture}
\end{subfigure}

 \begin{subfigure}{.30\textwidth}\scriptsize\centering\centering
   \begin{tikzpicture}[scale=0.30, level distance = 20mm]  \tikzstyle{level 1}=[sibling distance=75mm]  \tikzstyle{level 2}=[sibling distance=38mm]  \tikzstyle{level 3}=[sibling distance=19mm]  \tikzstyle{level 4}=[sibling distance=12mm]  \tikzstyle{level 5}=[sibling distance=12mm]  \tikzstyle{root}=[circle,draw,thick,fill=black,scale=.8]  \tikzstyle{interior}=[circle,draw,solid,thick,fill=black,scale=.8]  \tikzstyle{leaves}=[circle,draw,solid,thick,fill=black,scale=.8]  \node [root,label=above:{}]  {} child{node [interior,label=above:{}] {}child{node [interior,label=above:{}] {}child{node [interior,label=above:{}] {}child{node [leaves,label=below:{$\{7\}$}] {}}child{node [interior,label=above:{}] {}child{node [leaves,label=below:{$\{1\}$}] {}}child{node [leaves,label=below:{$\{5\}$}] {}}}}child{node [leaves,label=below:{$\{2\}$}] {}}}child{node [interior,label=above:{}] {}child{node [leaves,label=below:{$\{4\}$}] {}}child{node [leaves,label=below:{$\{6\}$}] {}}}}child{node [interior,label=above:{}] {}child{node [leaves,label=below:{$\{3\}$}] {}}child{node [leaves,label=below:{$\{8\}$}] {}}};\end{tikzpicture}
\end{subfigure}
\begin{subfigure}{.30\textwidth}\scriptsize\centering\centering
    \begin{tikzpicture}[scale=0.30, level distance = 20mm]  \tikzstyle{level 1}=[sibling distance=75mm]  \tikzstyle{level 2}=[sibling distance=38mm]  \tikzstyle{level 3}=[sibling distance=19mm]  \tikzstyle{level 4}=[sibling distance=15mm]  \tikzstyle{level 5}=[sibling distance=15mm]  \tikzstyle{level 6}=[sibling distance=15mm]  \tikzstyle{level 7}=[sibling distance=12mm]  \tikzstyle{root}=[circle,draw,thick,fill=black,scale=.8]  \tikzstyle{interior}=[circle,draw,solid,thick,fill=black,scale=.8]  \tikzstyle{leaves}=[circle,draw,solid,thick,fill=black,scale=.8]  \node [root,label=above:{}]  {} child{node [interior,label=above:{}] {}child{node [interior,label=above:{}] {}child{node [interior,label=above:{}] {}child{node [leaves,label=below:{$\{7\}$}] {}}child{node [interior,label=above:{}] {}child{node [interior,label=above:{}] {}child{node [leaves,label=below:{$\{3\}$}] {}}child{node [interior,label=above:{}] {}child{node [leaves,label=below:{$\{4\}$}] {}}child{node [leaves,label=below:{$\{6\}$}] {}}}}child{node [leaves,label=below:{$\{5\}$}] {}}}}child{node [leaves,label=below:{$\{2\}$}] {}}}child{node [leaves,label=below:{$\{8\}$}] {}}}child{node [leaves,label=below:{$\{1\}$}] {}};\end{tikzpicture}
\end{subfigure}
  \begin{subfigure}{.30\textwidth}\scriptsize\centering\centering
  \begin{tikzpicture}[scale=0.30, level distance = 20mm]  \tikzstyle{level 1}=[sibling distance=75mm]  \tikzstyle{level 2}=[sibling distance=38mm]  \tikzstyle{level 3}=[sibling distance=19mm]  \tikzstyle{level 4}=[sibling distance=12mm]  \tikzstyle{level 5}=[sibling distance=12mm]  \tikzstyle{level 6}=[sibling distance=12mm]  \tikzstyle{root}=[circle,draw,thick,fill=black,scale=.8]  \tikzstyle{interior}=[circle,draw,solid,thick,fill=black,scale=.8]  \tikzstyle{leaves}=[circle,draw,solid,thick,fill=black,scale=.8]  \node [root,label=above:{}]  {} child{node [interior,label=above:{}] {}child{node [leaves,label=below:{$\{8\}$}] {}}child{node [leaves,label=below:{$\{1\}$}] {}}}child{node [interior,label=above:{}] {}child{node [interior,label=above:{}] {}child{node [leaves,label=below:{$\{7\}$}] {}}child{node [interior,label=above:{}] {}child{node [interior,label=above:{}] {}child{node [leaves,label=below:{$\{3\}$}] {}}child{node [interior,label=above:{}] {}child{node [leaves,label=below:{$\{4\}$}] {}}child{node [leaves,label=below:{$\{6\}$}] {}}}}child{node [leaves,label=below:{$\{5\}$}] {}}}}child{node [leaves,label=below:{$\{2\}$}] {}}};\end{tikzpicture}
\end{subfigure}

 \begin{subfigure}{.30\textwidth}\scriptsize\centering\centering
   \begin{tikzpicture}[scale=0.30, level distance = 20mm]  \tikzstyle{level 1}=[sibling distance=70mm]  \tikzstyle{level 2}=[sibling distance=45mm]  \tikzstyle{level 3}=[sibling distance=35mm]  \tikzstyle{level 4}=[sibling distance=20mm]  \tikzstyle{level 5}=[sibling distance=15mm]  \tikzstyle{root}=[circle,draw,thick,fill=black,scale=.8]  \tikzstyle{interior}=[circle,draw,solid,thick,fill=black,scale=.8]  \tikzstyle{leaves}=[circle,draw,solid,thick,fill=black,scale=.8]  \node [root,label=above:{}]  {} child{node [interior,label=above:{}] {}child{node [leaves,label=below:{$\{8\}$}] {}}child{node [leaves,label=below:{$\{1\}$}] {}}}child{node [interior,label=above:{}] {}child{node [interior,label=above:{}] {}child{node [interior,label=above:{}] {}child{node [leaves,label=below:{$\{3\}$}] {}}child{node [interior,label=above:{}] {}child{node [leaves,label=below:{$\{4\}$}] {}}child{node [leaves,label=below:{$\{6\}$}] {}}}}child{node [interior,label=above:{}] {}child{node [leaves,label=below:{$\{7\}$}] {}}child{node [leaves,label=below:{$\{5\}$}] {}}}}child{node [leaves,label=below:{$\{2\}$}] {}}};\end{tikzpicture}
\end{subfigure}
 \begin{subfigure}{.30\textwidth}\scriptsize\centering\centering
   \begin{tikzpicture}[scale=0.30, level distance = 20mm]  \tikzstyle{level 1}=[sibling distance=75mm]  \tikzstyle{level 2}=[sibling distance=38mm]  \tikzstyle{level 3}=[sibling distance=28mm]  \tikzstyle{level 4}=[sibling distance=32mm]  \tikzstyle{level 5}=[sibling distance=19mm]  \tikzstyle{level 6}=[sibling distance=19mm]  \tikzstyle{root}=[circle,draw,thick,fill=black,scale=.8]  \tikzstyle{interior}=[circle,draw,solid,thick,fill=black,scale=.8]  \tikzstyle{leaves}=[circle,draw,solid,thick,fill=black,scale=.8]  \node [root,label=above:{}]  {} child{node [interior,label=above:{}] {}child{node [leaves,label=below:{$\{4\}$}] {}}child{node [interior,label=above:{}] {}child{node [interior,label=above:{}] {}child{node [interior,label=above:{}] {}child{node [interior,label=above:{}] {}child{node [leaves,label=below:{$\{8\}$}] {}}child{node [leaves,label=below:{$\{1\}$}] {}}}child{node [leaves,label=below:{$\{2\}$}] {}}}child{node [interior,label=above:{}] {}child{node [leaves,label=below:{$\{7\}$}] {}}child{node [leaves,label=below:{$\{5\}$}] {}}}}child{node [leaves,label=below:{$\{6\}$}] {}}}}child{node [leaves,label=below:{$\{3\}$}] {}};\end{tikzpicture}
\end{subfigure}
 \begin{subfigure}{.30\textwidth}\scriptsize\centering\centering
   \begin{tikzpicture}[scale=0.30, level distance = 20mm]  \tikzstyle{level 1}=[sibling distance=75mm]  \tikzstyle{level 2}=[sibling distance=38mm]  \tikzstyle{level 3}=[sibling distance=19mm]  \tikzstyle{level 4}=[sibling distance=19mm]  \tikzstyle{level 5}=[sibling distance=19mm]  \tikzstyle{level 6}=[sibling distance=15mm]  \tikzstyle{level 7}=[sibling distance=12mm]  \tikzstyle{root}=[circle,draw,thick,fill=black,scale=.8]  \tikzstyle{interior}=[circle,draw,solid,thick,fill=black,scale=.8]  \tikzstyle{leaves}=[circle,draw,solid,thick,fill=black,scale=.8]  \node [root,label=above:{}]  {} child{node [interior,label=above:{}] {}child{node [leaves,label=below:{$\{4\}$}] {}}child{node [interior,label=above:{}] {}child{node [interior,label=above:{}] {}child{node [leaves,label=below:{$\{7\}$}] {}}child{node [interior,label=above:{}] {}child{node [interior,label=above:{}] {}child{node [interior,label=above:{}] {}child{node [leaves,label=below:{$\{2\}$}] {}}child{node [leaves,label=below:{$\{1\}$}] {}}}child{node [leaves,label=below:{$\{8\}$}] {}}}child{node [leaves,label=below:{$\{5\}$}] {}}}}child{node [leaves,label=below:{$\{6\}$}] {}}}}child{node [leaves,label=below:{$\{3\}$}] {}};\end{tikzpicture}
 \end{subfigure}
\caption{Borehole function. The path of $9$ trees generated by the tree adaptive learning algorithm.} \label{fig:borehole-trees}
\end{figure}

\begin{figure}[h]\centering
\begin{subfigure}{.45\textwidth}\centering
\includegraphics[width=.8\textwidth]{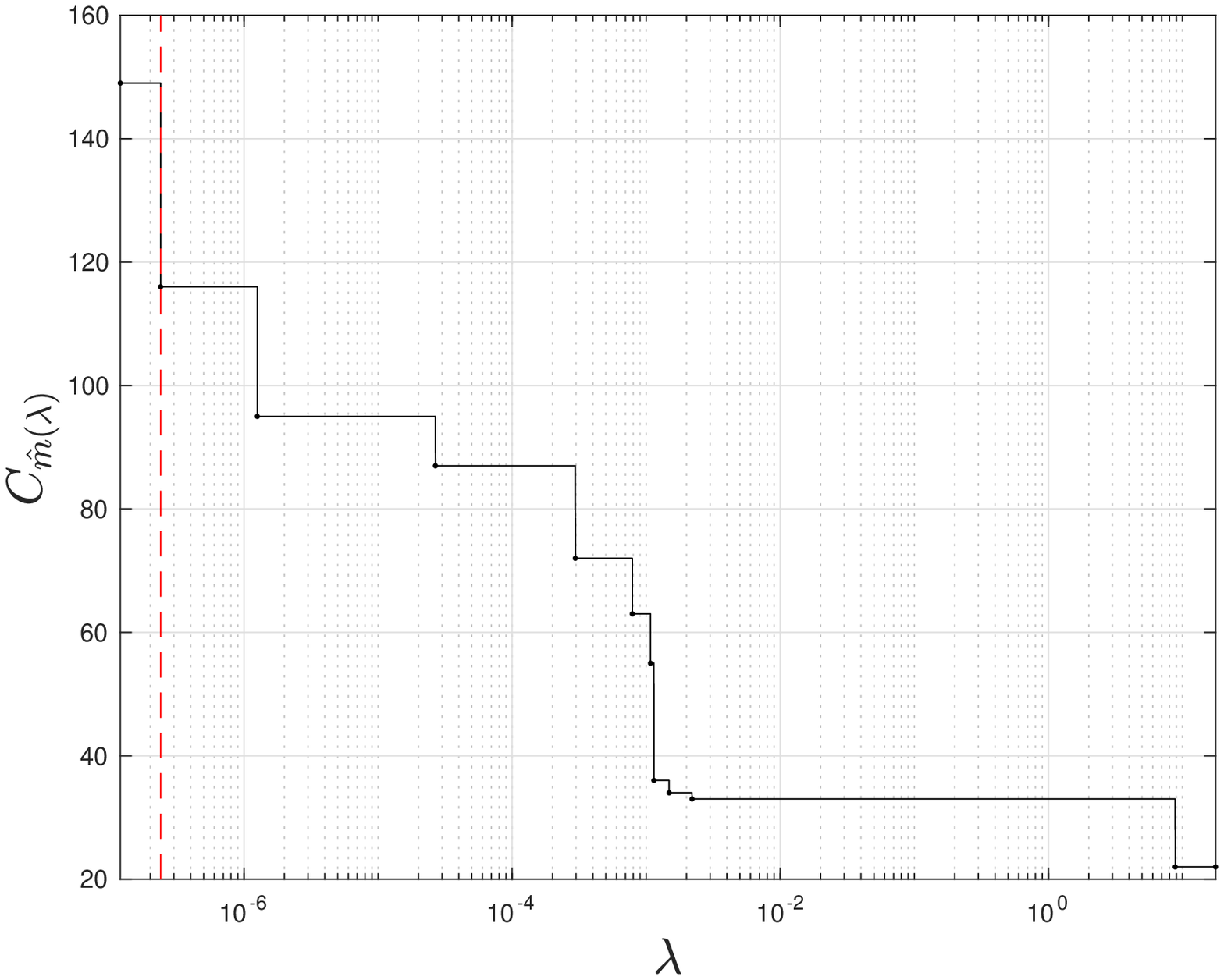}
\caption{Functions $\lambda\mapsto C_{\hat m(\lambda)}$, $\lambda^{cj}$ (red).}
\end{subfigure}  \hspace{.05\textwidth}
\begin{subfigure}{.45\textwidth}\centering
\includegraphics[width=.8\textwidth]{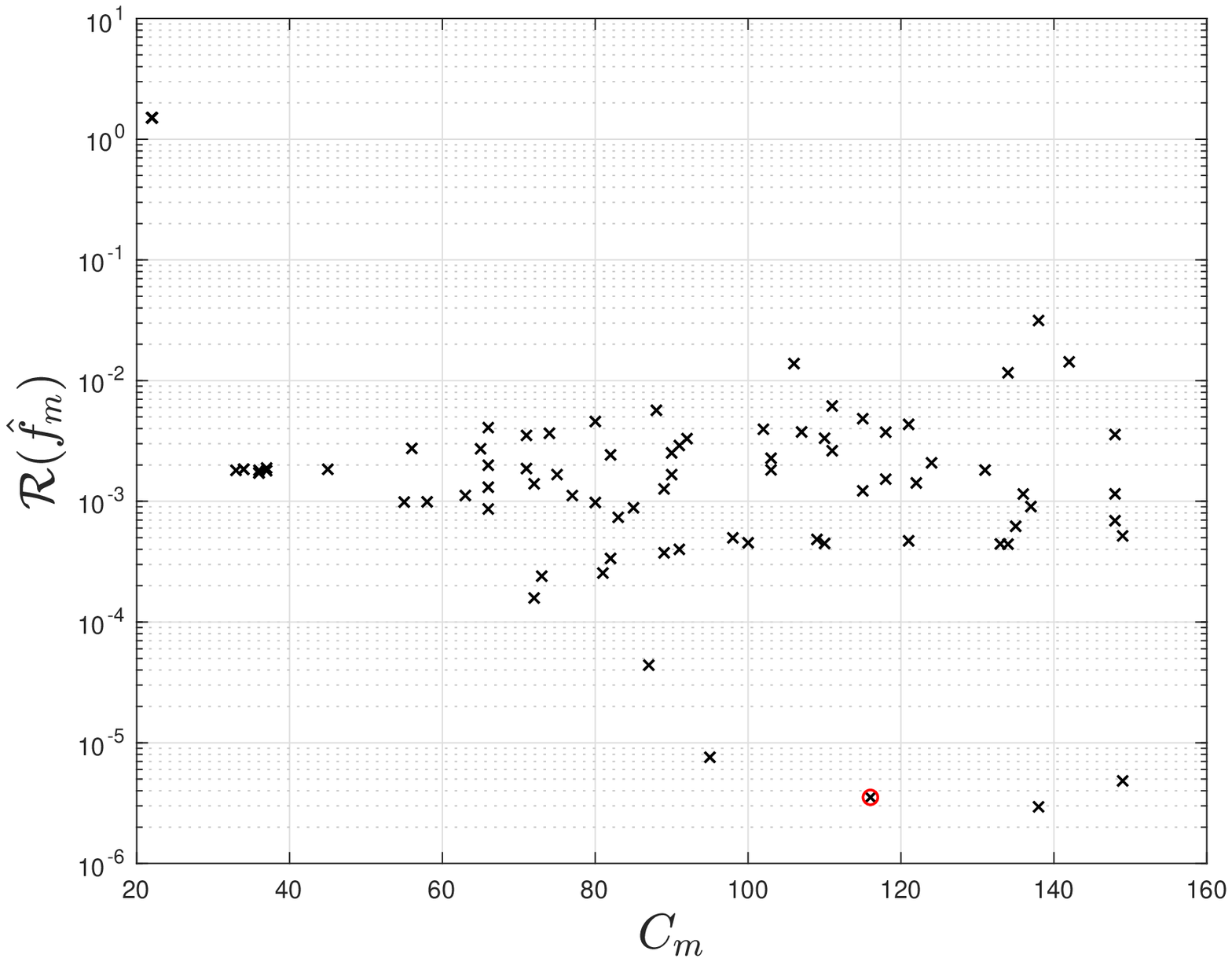}
\caption{Points $(C_m , \Rc(\hat f_m )$, ${m\in \Mc}$, and selected model (red).}
\end{subfigure}
\caption{Slope heuristics for  Borehole function with $n=100$ and $\gamma= 10^{-6}$.}
\label{fig:borehole-n100}
\end{figure}

\begin{figure}[h]\centering
\begin{subfigure}{.45\textwidth}\centering
\includegraphics[width=.8\textwidth]{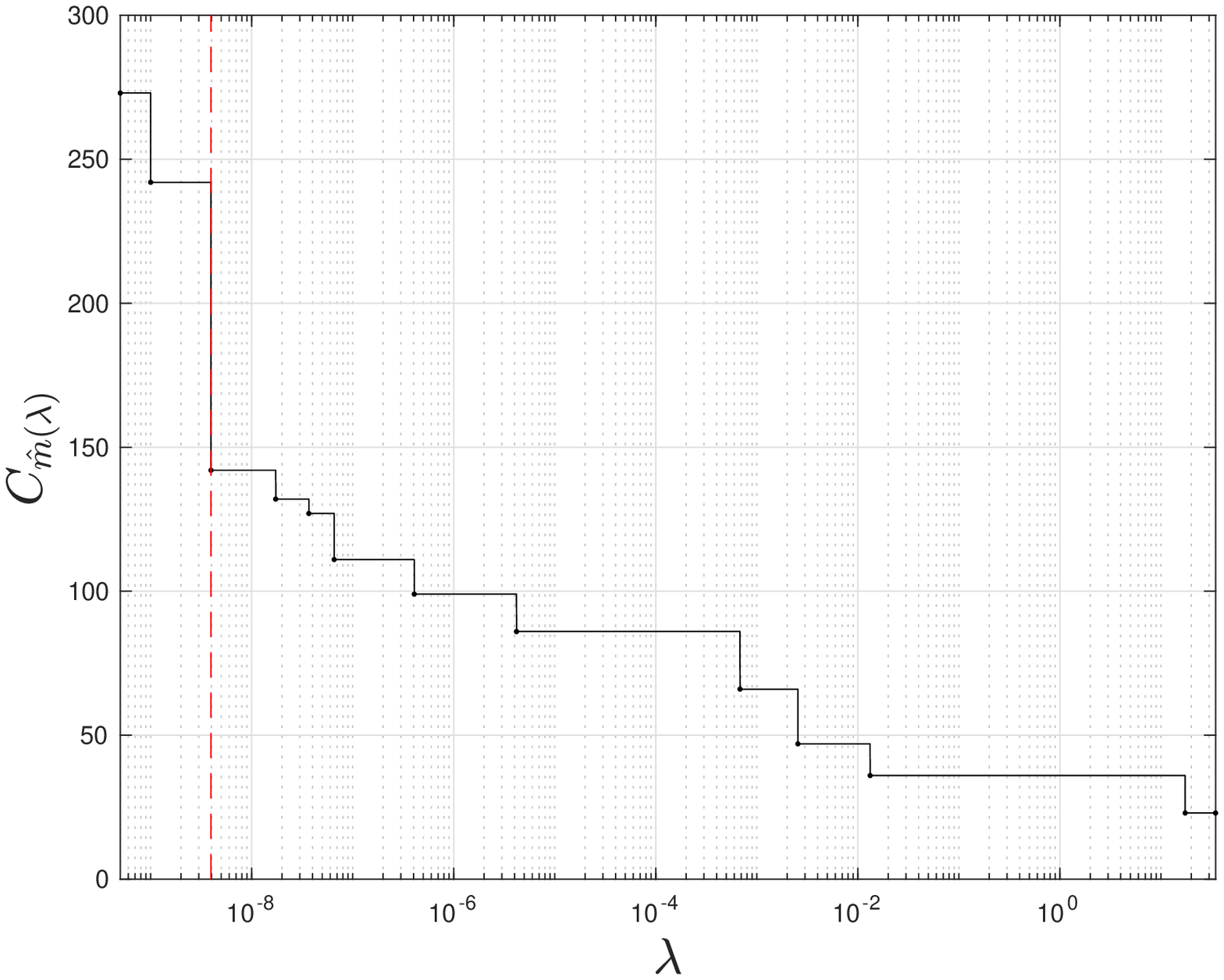}
\caption{Functions $\lambda\mapsto C_{\hat m(\lambda)}$, $\lambda^{cj}$ (red).}
\end{subfigure}  \hspace{.05\textwidth}
\begin{subfigure}{.45\textwidth}\centering
\includegraphics[width=.8\textwidth]{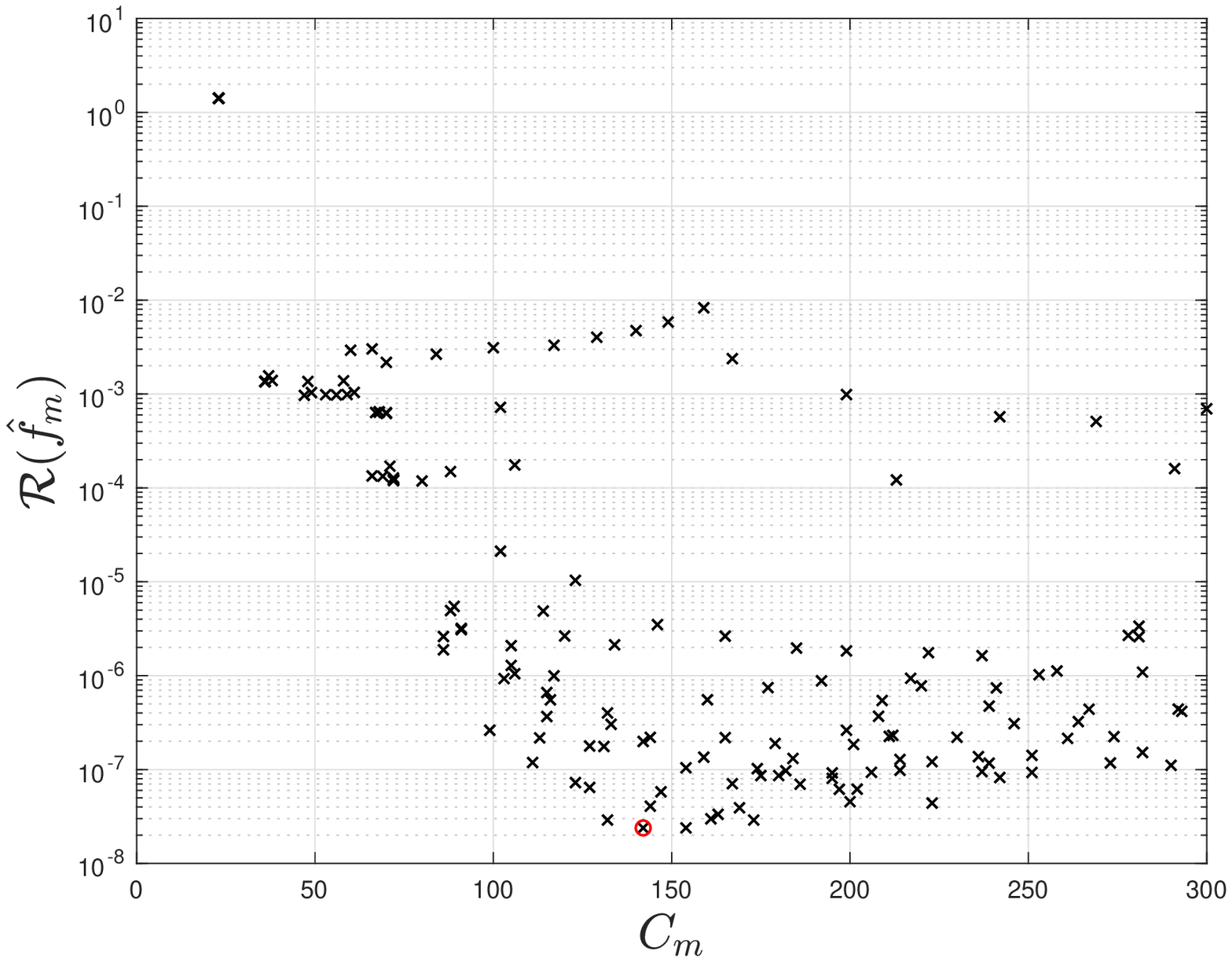}
\caption{Points $(C_m , \Rc(\hat f_m )$, ${m\in \Mc}$, and selected model (red).}
\end{subfigure}
\caption{Slope heuristics for  Borehole function with $n=200$ and $\gamma= 10^{-6}$.}
\label{fig:borehole-n200}
\end{figure}

\begin{figure}[h]\centering
\begin{subfigure}{.45\textwidth}\centering
\includegraphics[width=.8\textwidth]{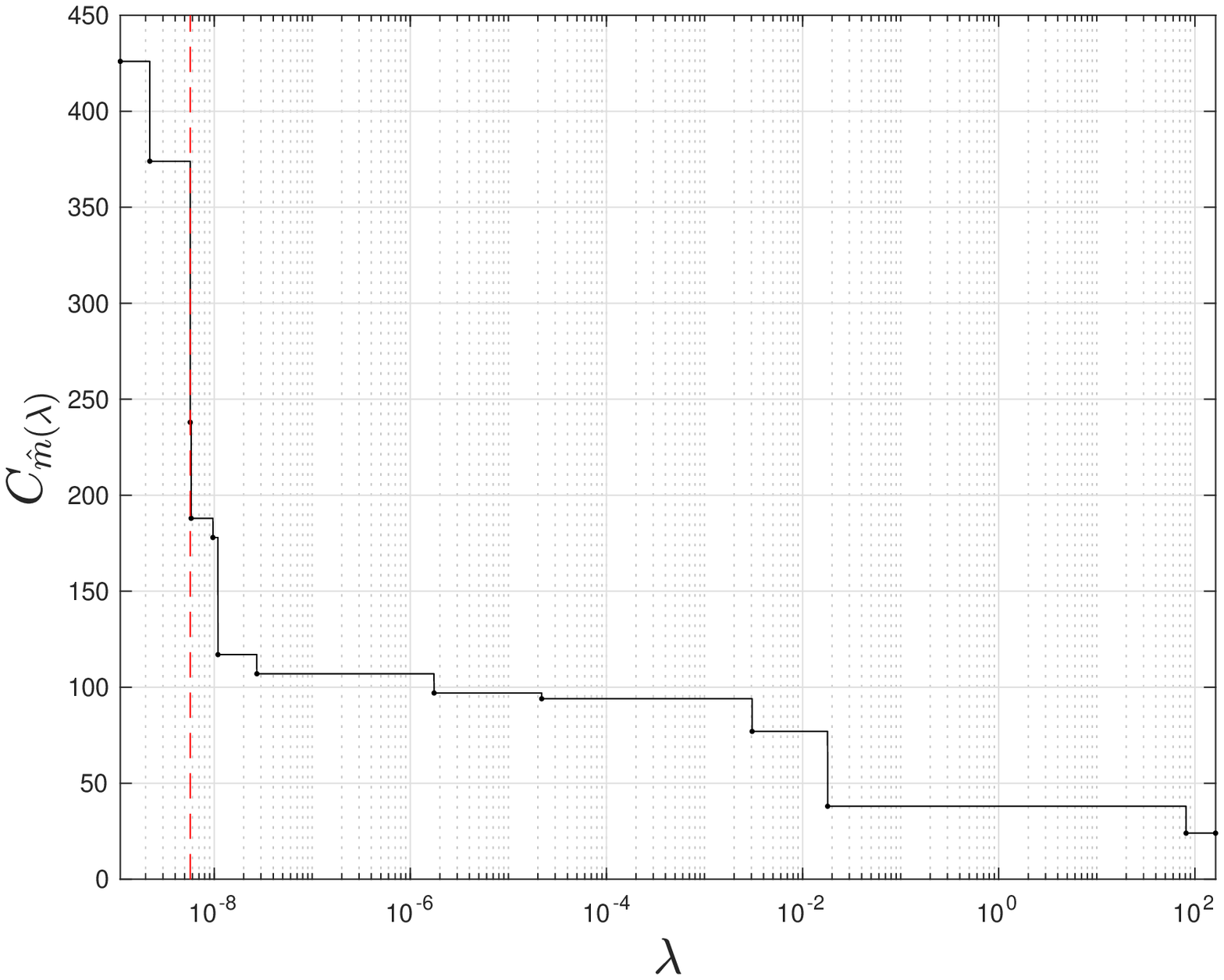}
\caption{Function $\lambda\mapsto C_{\hat m(\lambda)}$, $\lambda^{cj}$ (red).}
\end{subfigure}  \hspace{.05\textwidth}
\begin{subfigure}{.45\textwidth}\centering
\includegraphics[width=.8\textwidth]{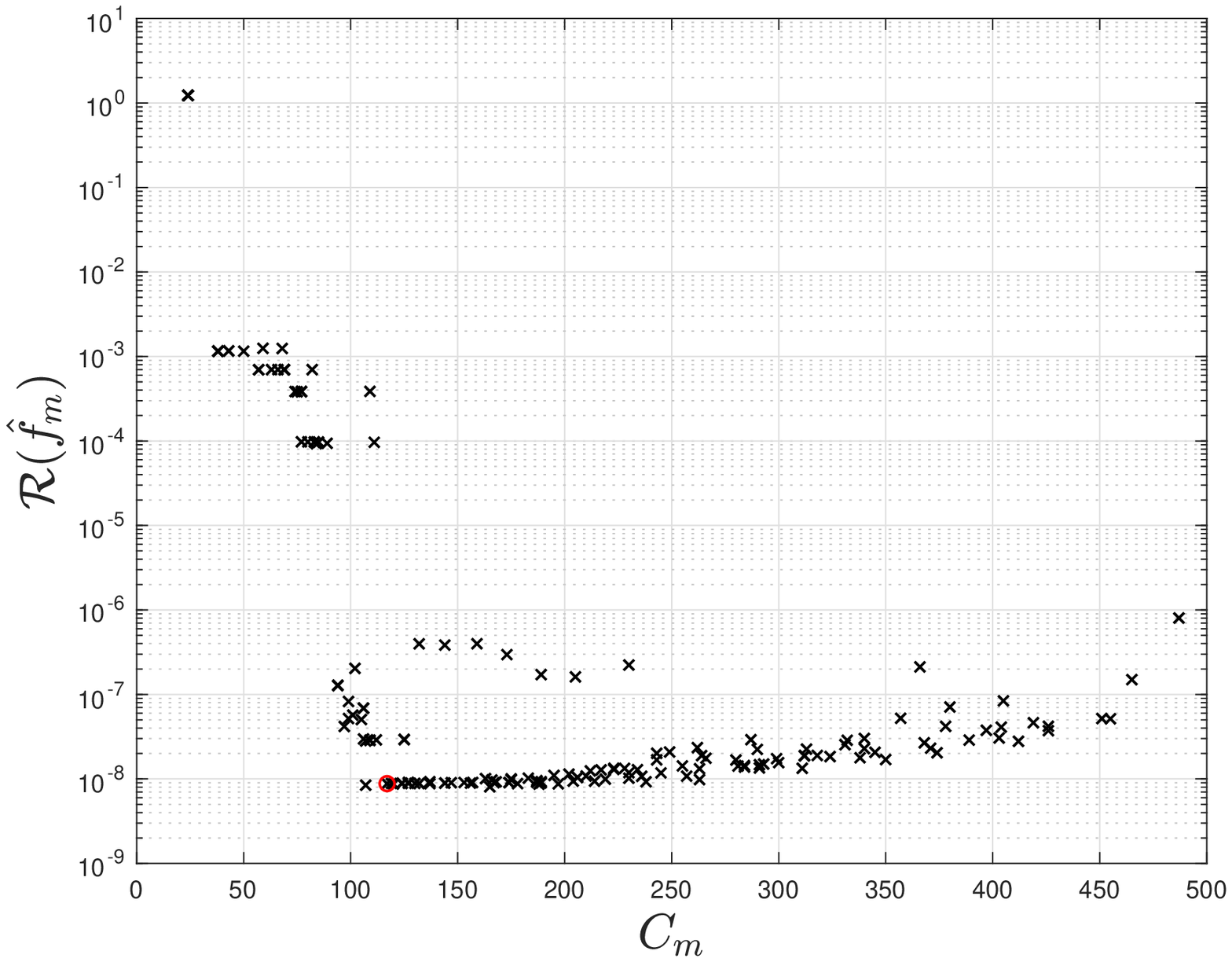}
\caption{Points $(C_m , \Rc(\hat f_m )$, ${m\in \Mc}$, and selected model (red).}
\end{subfigure}
\caption{Slope heuristics for  Borehole function with $n=1000$ and $\gamma= 10^{-6}$.}
\label{fig:borehole-n1000}
\end{figure}

\begin{table}[h]
\centering
\begin{tabular}{|c|c||c|c||c|c|}
 \hline 
 $n$ & $\mathbb{E}(C_{\hat m^\star})$  & $\mathbb{E}(C_{\hat m})$  & $\mathbb{E}(\mathcal{R}(\hat f_{m^\star}))$&  $\mathbb{E}(\mathcal{R}(\hat f_{\hat m}))$ \\ \hline 
 100 & 132.1 &  63.4 & 6.9e-06  & 9.3e-04 \\ 
 200 & 149.7 & 156.0 & 3.0e-08  & 1.1e-07 \\
 500 & 144.7 & 178.2 & 1.0e-08  & 1.8e-08 \\
 1000 &  154.1 & 194.2 & 8.3e-09  & 1.2e-08 \\ 
 \hline\end{tabular}
\caption{Borehole function. Expectation of complexities and risks. $\gamma= 10^{-6}$, different $n$.}
\label{tab:borehole}\end{table}


 \clearpage

\bibliographystyle{plain}

\appendix

  \section{Tree tensor networks as compositions of multilinear functions}\label{sec:multilinear-composition}

A function \modif{$f = \Rc_{\Hc,T,r}(\mathbf{v} )$} in $M^T_r(\Hc)$ admits a representation in terms of compositions of multilinear functions.  Indeed, for any interior node  $\alpha \in \Ic(T)$, the tensor $v^\alpha \in \Rbb^{r_\alpha \times_{\beta\in S(\alpha) }r_\beta}$ can be linearly identified with a multilinear map 
$$
v^\alpha : \bigtimes_{\beta \in S(\alpha)} \Rbb^{r_\beta} \to \Rbb^{r_\alpha}
$$
defined by
$$
v^\alpha((z^\beta)_{\beta \in S(\alpha)})_{k_\alpha} = \sum_{\substack{1\le k_\beta \le r_\beta \\ \beta \in S(\alpha)} } v_{k_\alpha,(k_\beta)_{\beta \in S(\alpha)}}^\alpha \prod_{\beta\in S(\alpha)} z_{k_\beta}^\beta
$$ 
for $z^\beta \in \Rbb^{r_\beta}.$  
For a given $\alpha \in T$, we let $g^\alpha(x_\alpha) = (g_{k_\alpha}^\alpha(x_\alpha))_{1 \le k_\alpha \le r_\alpha} \in \Rbb^{r_\alpha}$. Therefore, a function $f$  in $M^T_r(\Hc)$ admits the representation 
\modif{ $$
f(x) = v^D((g^\alpha(x_\alpha))_{\alpha\in S(D)}).
$$}
where for any $\alpha \in \Ic(T) \setminus  \{D\}$, $g^\alpha$ admits the representation 
\begin{equation}
\label{eq:decomp_g}
g^{\alpha}(x_\alpha) = v^\alpha((g^\beta(x_\beta)_{\beta\in S(\alpha)}).
 \end{equation}
For a leaf node $\alpha \in \Lc(T)$, the tensor $v^\alpha$ can be linearly identified with a linear map $v^\alpha : \Rbb^{n_\alpha} \to \Rbb^{r_\alpha},$ and  
\begin{equation}
\label{eq:decomp_g_leaf}
g^{\alpha}(x_\alpha) = v^\alpha(\phi^\alpha(x_\alpha)).
 \end{equation}

\section{Proofs of \Cref{sec:entropy}} \label{proofEntropy}
\begin{proof}[Proof of \Cref{prop:Lp-bound}]
Let $f = \Rc_{\Hc,T,r}((v^\alpha)_{\alpha \in T})$\modif{, where the tensor 
$v^\alpha$ is identified with a $\Rbb^{r_\alpha}$-valued multilinear (resp. linear) map  for $\alpha \in \Ic(T)$ (resp. $\alpha \in \Lc(T)$),  see \Cref{sec:multilinear-composition} for details.}
For $x\in \Xc$, we first note that  
$$
\vert f(x) \vert = \vert v^D((g^\alpha(x_\alpha))_{\alpha \in S(D)}) \vert \le \Vert v^D \Vert_{\mathcal{P}^\alpha} \prod_{\alpha \in S(D)} \Vert g(x_\alpha) \Vert_{p}, 
$$
with $\Vert \cdot \Vert_p$ the vector $\ell^p$-norm.
Then for any interior node $\alpha \in \Ic(T)$, we have
$$\Vert g^\alpha(x_\alpha)\Vert_p = \Vert v^\alpha((g^\beta(x_\beta))_{\beta\in S(\alpha)}) \Vert_p \le \Vert v^\alpha \Vert_{\Pc^\alpha} \prod_{\beta\in S(\alpha)} \Vert g^\beta(x_\beta) \Vert_p,$$
and for any leaf node $\alpha \in \Lc(T)$, 
$$
\Vert g^\alpha(x_\alpha) \Vert_p  = \Vert v^\alpha(\phi^\alpha(x_\alpha)) \Vert_p \le \Vert v^\alpha \Vert_{\Pc^\alpha} \Vert \phi^\alpha(x_\alpha) \Vert_p.
$$
We deduce that  
$$
\vert f(x) \vert_p \le \prod_{\alpha \in  T} \Vert v^\alpha \Vert_{\Pc^\alpha} \prod_{1\le \nu\le d} \Vert \phi^{\nu}(x_\nu) \Vert_p,
$$
and therefore, since $\mu$ is a product measure and from the particular normalization of functions $\phi^\nu$, we obtain
$$
\Vert f \Vert_{p,\mu} \le \prod_{\alpha \in  T} \Vert v^\alpha \Vert_{\Pc^\alpha} \prod_{1\le \nu\le d} \Vert \phi^{\nu}\Vert_{p,\mu} =  \prod_{\alpha \in  T} \Vert v^\alpha \Vert_{\Pc^\alpha},
$$
which proves that $L_{p,\mu} \le 1$. Finally for $1\le q \le p$, we note that $$\mu(\mathcal{X})^{1/p - 1/q} \Vert f \Vert_{q,\mu}  \le  \Vert f \Vert_{p,\mu} \le  \Vert f \Vert_{\infty,\mu} ,
$$
which yields $L_{q,\mu} \le  \mu(\mathcal{X})^{1/q - 1/p} L_{p,\mu}$.
\end{proof}

\section{Proofs of  \Cref{sec:general_risk_bounds}}

\subsection{Concentration inequalities for empirical processes}
\label{subs:concentration}

We here apply classical results to control the fluctuations of the supremum of the empirical process $\bar \Rc_n(f)$ over the model class $M$.  
 
\medskip 
 
Assumption \ref{ass:bounded} (Bounded contrast) yields a classical concentration inequality for the empirical process $ \bar \Rc_n(f)$.
\begin{lemma}\label{lem:concentration-hoeffding}
Under assumption \ref{ass:bounded}, we have that for all $\varepsilon >0$ and all $f\in M$
\begin{equation}
\Pbb( \bar \Rc_n(f)   > \varepsilon B) \vee \Pbb( \bar \Rc_n(f)   < -\varepsilon B) \le e^{-  n \frac{\varepsilon^2}{2}}.\label{concentration-hoeffding}
\end{equation}
\end{lemma}
\begin{proof}
We have $\widehat \Rc_n(f) - \Rc(f) = \frac{1}{n} \sum_{i=1}^n A^f_i - \Ebb(A^f)$, where  the $A^f_i = \gamma(f,Z_i)$ are i.i.d. copies of the random variable $A^f = \gamma(f,Z)$. From Assumption \ref{ass:bounded}, we have that $\vert A^f \vert \le B$ almost surely, so that 
$A^f $ is subgaussian with parameter $B^2$ and  the result simply follows from 
Hoeffding's inequality. 
\end{proof}
If $\gamma(\cdot,Z)$ is Lipschitz continuous over $M \subset L^{ \infty}_\mu(\Xc)$  we obtain a uniform concentration inequality for the empirical process $ \bar \Rc_n(f)$ over $M$:
\begin{lemma}\label{lem:uniform-concentration}
Under Assumptions \ref{ass:bounded} and \ref{ass:lipschitz}, we have that for all $\varepsilon >0$ and all $f\in M$
\begin{equation}
\Pbb(\sup_{f\in M}   \bar \Rc_n(f)   > 2 \varepsilon B) \vee \Pbb(\inf_{f\in M}   \bar \Rc_n(f)   < -2 \varepsilon B) \le  N_{\frac{\varepsilon B}{2\mathcal L}} e^{-\frac{n\varepsilon^2}{2}}, \label{uniform-concentration}
\end{equation}
where $N_{\frac{\varepsilon B}{2\mathcal L}} = N(\frac{\varepsilon B}{2\mathcal L} , M , \Vert \cdot \Vert_{ \infty,\mu})$ is the covering number of $M$ at scale $\frac{\varepsilon B}{2\mathcal L} $, and 
$$
\log N_{\frac{\varepsilon B}{2\mathcal L}} \le C_M\log \left( 6\mathcal L B^{-1} R |T| \varepsilon^{-1} \right) . $$
\end{lemma}

\begin{proof}[Proof of \Cref{lem:uniform-concentration}]
Let $\gamma = \frac{\varepsilon B}{2 \mathcal L}$ and let $\Nc$ be a $\gamma$-net of $M$ for the $\Vert\cdot \Vert_{\infty,\mu}$-norm, with cardinal $N_{\frac{\varepsilon B}{2 \mathcal L}}$. Using \Cref{lem:concentration-hoeffding} and a union bound argument, we obtain 
$$
\Pbb(\sup_{g\in \Nc}  \bar \Rc_n(g)   > \varepsilon B) \vee \Pbb(\inf_{g\in \Nc}  \bar \Rc_n(g)   < -\varepsilon B) \le N_{\frac{\varepsilon B}{2\mathcal L}} e^{-\frac{n\varepsilon^2}{2}}.
$$
For any $f \in M$, there exists a $g\in \Nc$ such that $\Vert f - g \Vert_{\infty,\mu} \le \gamma $. Noting that 
\begin{align*}
   \bar \Rc_n(f) &= \bar \Rc_n(g) + \widehat \Rc_n(f)-  \widehat \Rc_n(g) + \Rc(g) - \Rc(f), 
   \end{align*}
   we deduce from \Cref{ass:lipschitz} that 
   $$
   \bar \Rc_n(f) \le  \bar \Rc_n(g) + 2\mathcal L \Vert f - g \Vert_{\infty,\mu} \le \sup_{g\in \Nc} \bar \Rc_n(g) + \varepsilon B, 
   $$
   and 
    $$
\bar \Rc_n(f) \ge \bar \Rc_n(g) - 2\mathcal L \Vert f - g \Vert_{\infty,\mu} \ge  \inf_{g\in \Nc} \bar \Rc_n(g) - \varepsilon B.
   $$
   This implies that
   $$
   \Pbb(\sup_{f\in M} \bar \Rc_n(f) > 2 \varepsilon B) \le \Pbb(\sup_{g\in \Nc} \bar \Rc_n(f) > \varepsilon B),
   $$
   and 
   $$
   \Pbb(\inf_{f\in M} \bar \Rc_n(f) < -2 \varepsilon B) \le \Pbb(\inf_{g\in \Nc} \bar \Rc_n(f) < -\varepsilon B),
   $$
   which yields \eqref{uniform-concentration}. The bound on $N_{\frac{\varepsilon B}{2 \mathcal L}}$ directly follows from 
    \Cref{prop:metric-entropy} and \Cref{prop:Lp-bound}.
\end{proof}

\begin{lemma}\label{lem:bound-expectation-supremum}
Under Assumptions \ref{ass:bounded} and \ref{ass:lipschitz}, 
$$
\Ebb( \sup_{f \in M}  | \bar \Rc_n(f) | )  \le  4B     \sqrt{C_M} \sqrt{  \frac{ 2 \log ((\beta \vee e) \sqrt n) } n } .
$$
with $\beta = 6\mathcal LB^{-1}R |T|.$
\end{lemma}
 \begin{proof}[Proof of \Cref{lem:bound-expectation-supremum}]
We have
\begin{align*}
\Ebb(\sup_{f \in M} | \bar \Rc_n(f) | )&  =
  \int_{0}^\infty \Pbb(\sup_{f \in M} \vert \bar \Rc_n(f)  \vert >t) dt\\
&= 2B \int_{0}^\infty \Pbb(\sup_{f \in M} \vert \bar \Rc_n(f)  \vert > 2\varepsilon B) d\varepsilon .
\end{align*}
Let  $\beta = 6\mathcal LB^{-1}  R  |T|$. Then, according to \Cref{lem:uniform-concentration}, for any $\delta >0$,
\begin{align*}
\Ebb(\sup_{f \in M} | \bar \Rc_n(f) | )& \le 2B \left[ \delta +   \int_{\delta}^\infty 2  (\beta  \varepsilon^{-1})^{C_M} e^{- n \frac{\varepsilon^2}{2}} d\varepsilon \right], \\
& = 2B \left[ \delta +  2  \beta ^{C_M}  \int_{n \delta^2 /2 }^\infty   \left(\frac{2 u}n \right)^{-C_M/2 } e^{-u}  \frac{1 }{ \sqrt{2n u}}d u  \right] \\
& \le 2B \left[ \delta +   {2 n^{-1}  \beta     ^{C_M} \delta^{-C_M-1}}  e^{-n \delta^2 /2 }    \right],
\end{align*}
By taking 
$$
\delta =  \sqrt{  \frac{2 C_M}{n}   \log (  (\beta \vee e)   \sqrt n) }, 
$$
we have 
\begin{align*}
n^{-1}  \beta^{C_M} \delta^{-C_M-1}  e^{-n \delta^2 /2 } &= n^{-1}\beta^{C_M} \delta^{-C_M-1}  (\beta \vee e)^{-C_M} n^{-\frac{C_M}{2}} \\
&\le   \delta^{-C_M-1}  n^{-\frac{C_M}{2}-1}  \\
&=\delta (\delta^2 {n})^{-\frac{C_M}{2}-1}\\
&= \delta  (2C_M  \log ( (\beta \vee e)    \sqrt n) )^{-\frac{C_M}{2}-1} \\
&\le \delta  
\end{align*}
where we have used the fact that $2C_M  \log ( (\beta \vee e)    \sqrt n) \ge 1.$
Then 
\begin{align*}
\Ebb(\sup_{f \in M} | \bar \Rc_n(f) | )&\le  4   B  \delta  ,
\end{align*}
which concludes the proof. 
\end{proof}

\subsection{Proof of \Cref{prop:probabound}}\label{sec:proof:prop:probabound}

The excess risk for the estimator $\hat f^M_n $ satisfies 
\begin{equation*}
\Ec(\hat f^M_n ) = \Ec(f^M) + \Rc(\hat f^M_n) - \Rc(f^M),
 \label{excess-risk}
\end{equation*}
where $ \Ec(f^M)$ is the best approximation error in $M$ and $\Rc(\hat f^M_n) - \Rc(f^M)$ is the estimation error. 
Using the optimality of $\hat f^M_n$, we obtain that the estimation error satisfies
\begin{eqnarray*}
\Rc(\hat f^M_n) - \Rc(f^M) &\le & 
 \widehat \Rc_n(f^M) - \Rc(f^M) - \widehat \Rc_n(\hat f^M_n) + \Rc(\hat f^M_n)  \\
& \le  & \bar \Rc_n(f^M) - \bar \Rc_n(\hat f^M_n).
\end{eqnarray*}
Thus
$$ \mathcal E (\hat f_n^M)  \leq \mathcal E (f^M)   + 2 \sup_{f \in M}  | \bar \Rc_n(f)   | . $$ 
Under Assumption~\ref{ass:bounded}, the bounded difference Inequality (see for instance Theorem 5.1   in  \cite{Massart:07}) applied  to $\sup_{f \in M} |  \bar \Rc_n(f)  |  $ gives that with probability larger than $ 1-  \exp (- t  )$,
\begin{eqnarray*}
 \sup_{f \in M}  |  \bar \Rc_n(f) |
 & \le &     \Ebb( \sup_{f \in M}   |  \bar \Rc_n(f)  |) + 2B   \sqrt{\frac t {2n} }. 
  \end{eqnarray*}
 \Cref{lem:bound-expectation-supremum} together with Assumption~\ref{ass:lipschitz}  gives the risk bound.

\subsection{Proof of \Cref{theo:selec_model_tensor}}\label{sec:proof:theo:selec_model_tensor}

By definition of $\hat m$, for any $m \in \mathcal M$,
\begin{eqnarray*}
\mathcal R_n( \hat f_{\hat m})   + \pen(\hat m)  \leq   \mathcal R_n (\hat f_m)   + \pen(m)   
\leq   \mathcal R_n ( f_m)   + \pen(m) .
\end{eqnarray*}
Therefore,
 $$ \mathcal R_n( \hat f_{\hat m})    \leq  \mathcal R_n (f_m)   + \pen(m) -  \pen(\hat m) $$  
and thus
  $$   \mathcal R ( \hat f_{\hat m})  + \bar{ \mathcal R}_n( \hat f_{\hat m})    \leq  
  \mathcal R (f_m)  +\bar{\mathcal R}_n ( f_m )     + \pen(m) -  \pen(\hat m) , $$  
  where $ \bar  \Rc_n(f) $ is the centered empirical process defined in \eqref{centered-empirical-process}.  We finally derive the following upper bound on the excess risk 
\begin{equation} 
\Ec (\hat f_{\hat m}) \leq  \Ec(f_m) +  \bar \Rc_n (f_m) -  \bar \Rc_n(  \hat f_{\hat m}) - \pen (\hat m) + \pen (m).  \label{temp_risk_bound} 
\end{equation}
As in the proof of \Cref{prop:probabound}, by applying the bounded difference Inequality to $\sup_{f \in M}   - \bar \Rc_n(f)$  and by  \Cref{lem:bound-expectation-supremum}, it gives that for  any $t >0$,  with probability larger than $ 1-  \exp (- t  )$,  
\begin{equation} \label{riskbound}
 \sup_{f \in M}   - \bar \Rc_n(f)    
\le  4B     \sqrt{C_M} \sqrt{  \frac{  2 \log (  6\mathcal LB^{-1}  R  |T|   \sqrt n) } n } +2 B   \sqrt{\frac t {2n} } .
  \end{equation}
Thus, for any $t >0$ and any $m \in \mathcal M$, one has with probability larger than $1 - \exp(-t)$,
\begin{align*}
\sup_{ f \in M_m} - \bar \Rc_n (f)  &\leq  \lambda_m \sqrt{ \frac{C_m}{n}}  + 2B   \sqrt{\frac {t}  {2n} } .
\end{align*} 
Let $w_m = \bar w C_m +  \log  (\Nc_{C_m})$. Then,   with probability larger than $1 - \sum_{m\in \Mc} e^{-w_m-t} = 1- \frac 1{e^{\bar w} -1} e^{-t}$, it holds 
\begin{align*}
-\bar \Rc_n(\hat f_{\hat m})  \le  \sup_{f\in M_{\hat m}} -\bar \Rc_n(f)  \le \lambda_{\hat m} \sqrt{ \frac{C_{\hat m}}{n}} + 2 B   \sqrt{\frac {t + w_{\hat m}}  {2n} }, \end{align*} 
which together with \eqref{temp_risk_bound}  implies that 
\begin{align*} 
 \Ec (\hat f_{\hat m})   \leq   \Ec(f_m) +  \bar \Rc_n (f_m)  + \lambda_{\hat m} \sqrt{ \frac{C_{\hat m} }{n}} + 2B   \sqrt{\frac {w_{\hat m} }  {2n} } - \pen (\hat m) + \pen (m) + 2B   \sqrt{\frac t {2n} }
\end{align*} 
holds for all $m \in \mathcal M$. 
Then, with the condition \eqref{pen_choice} on the  penalty function, the upper bound
$$
 \Ec (\hat f_{\hat m})    \leq   \Ec(f_m) +  \bar \Rc_n (f_m)  + \pen (m) + 2 B   \sqrt{\frac t {2n} }
$$ 
 holds for all $m\in \Mc$ simultaneously, with probability larger than $1- \frac 1{e^{\bar w} -1} e^{-t}$. Next, integrating with respect to $t$ gives
$$
 \mathbb E \left[0\vee \left( \Ec( \hat f_{\hat m}) -  \Ec(f_m) -   \bar \Rc_n (f_m)  -  \pen (m) \right) \right] \leq  2 \frac B{\exp(\bar w) -1}  \sqrt{ \frac{2 \pi} n} \frac{1}{4} .
 $$
Finally, since $ \bar \Rc_n(f_m)$ has zero mean,  for any $m \in \mathcal M$,
$$ \mathbb E (  \Ec  ( \hat f_{\hat m}))  \leq \Ec(f_m)       +  \pen (m)  +  \frac B{\exp(\bar w) -1}  \sqrt{ \frac \pi{2n}},
$$ 
and we conclude by taking the infimum over $m\in \Mc$.

\subsection{Proof of \Cref{prop:N_C}}\label{sec:proof:prop:N_C}
The collections of models have complexities 
\begin{align*}
&\Nc_c(\Mc_{\Hc,T}) = \vert \{r\in \Nbb^{\vert T\vert} : C(T,r,\Hc_N) = c \} \vert \\
&\Nc_c(\Mc_{T}) = \vert \{r\in \Nbb^{\vert T\vert} ,  N \in \Nbb^{d} : C(T,r,\Hc_N)= c \} \vert \\
&\Nc_c(\Mc_\star) = \vert \{T \in \mathcal{T}_{a,d} , r\in \Nbb^{\vert T\vert} ,  N\in \Nbb^{d} :C(T,r,\Hc_N)= c \} \vert 
\end{align*}
where $\mathcal{T}_{a,d}$ denotes the collection of trees with arity $a$ (or $a$-ary trees).
We easily see that the above families of models have growing complexity, i.e.
$$
\Nc_c(\Mc_{\Hc,T}) \le \Nc_c(\Mc_T) \le \Nc_c(\Mc_\star),
$$ 
for any $T$ and $\Hc$. 
Let us first consider the collection $\Mc_T$ for a given tree $T$. 
Let us recall that for a tree $T$, a tuple $r\in \Nbb^{T}$ and a feature space $\Hc_N$ with $N\in \Nbb^d$, the full representation complexity is given by 
$$
C(T,r,\Hc_N) = \sum_{\alpha} \vert K_\alpha \vert, 
$$
with $ \vert K_\alpha \vert = r_\alpha \prod_{\beta\in S(\alpha)} r_\beta$ for $\alpha\not\in \Lc(T)$ and $ \vert K_\alpha \vert = r_\alpha N_\alpha$ for $\alpha\in \Lc(T)$. 
Then 
$$
\Nc_c(\Mc_{T}) \le \sum_{(q_\alpha)_{\alpha\in T}}\vert \{ (r,N) \in \Nbb^{\vert T\vert} \times \Nbb^d : \vert K_\alpha \vert = q_\alpha , \alpha \in T\} \vert
$$
where the sum is taken over all tuples $(q_\alpha)_{\alpha \in T} \in \Nbb^{\vert T \vert} $ such that 
$ \sum_{\alpha\in T} q_\alpha = c$.
For any $q_\alpha \in \Nbb$, the number of tuples $(r_1,\hdots,r_{a+1})$ such that $\prod_{k=1}^{a+1} r_k \le q_\alpha$  is less than $q_\alpha^{a-1}$. For $\alpha \in \Lc(T)$, the number of pairs $(r_\alpha,N_\alpha) \in \Nbb^2$ such that $N_\alpha r_\alpha= q_\alpha$ is less than $q_\alpha $. Thus for any tuple of integers $(q_\alpha)_{\alpha \in  T}$ such $\sum_{\alpha\in T} q_\alpha=c$, the number of tuples $(r,N) \in \Nbb^{\vert T\vert} \times \Nbb^d$ such that  $\vert K_\alpha \vert= q_\alpha$ for all $\alpha \in T$ is less than
\begin{align*} 
\prod_{\alpha \in T}  q_\alpha^{a-1}  & \leq   
\left[  \Big(  \prod_{\alpha \in T} q_\alpha  \Big)^{ \frac{1}{| T| }}   \right]^{(a-1) |T|}  \leq    \left[  \frac 1{ |T|}  \sum_{\alpha \in T}  q_\alpha     \right]^{(a-1) |T|} =   \left[  \frac c { |T|}        \right]^{(a-1)|T|}  .
\end{align*}
Moreover, the number of tuple of integers $(q_\alpha)_{\alpha \in  T}$ 
satisfying $  \sum_{\alpha \in T} q_\alpha  = c$ is bounded by ${c +|T| }\choose{c}$. 
Thus, we deduce that 
$$ \Nc_c(\Mc_T)  \le    { {c + |T| }\choose{c}}  \left[  \frac c { | T|}    \right]^{(a-1)|T|} .$$
Using the inequality
\begin{equation} \label{eq:logbino}
 \log { {k} \choose{ \ell }} \leq  \ell( 1+   \log \frac k \ell ) ,
  \end{equation}
and the fact that $\vert T\vert \le 2d$ and $|T| \le c$ for any model $m$ with complexity $c$, we obtain  
 \begin{align*}
\log(\Nc_c(\Mc_T))& \le    c( 1+\log (\frac{c+ |T|}{c})) + (a-1) \vert T \vert \log(\frac{c}{|T|})\\
& \le   c(  1+\log(2)) + 2d(a-1) \log(c),
\end{align*}  
which yields 
$$
 \log(\Nc_c(\Mc_T)) \le 2(a-1) (c + d\log(c)) .
$$
Now consider the family $\Mc_{\star}$. We first note that 
$$
\Nc_c(\Mc_\star) = \sum_{T \in \mathcal{T}_{a,d}}    \Nc_c(\Mc_{T}) ,
$$
and using the bound on  $\Nc_c(\Mc_{T})$, we obtain 
$$
\log(\Nc_c(\Mc_\star))  \le \log( \vert \mathcal{T}_{a,d} \vert ) +  2(a-1) (c + d\log(c)),
$$
where $ \vert \mathcal{T}_{a,d} \vert$ is the number of all possible trees with arity $a$ with $d$ leaves, which is the Fuss-Catalan number $C_a(d-1) = \frac{1}{(a-1)(d-1)+1}  { {a(d-1) }\choose{d-1}}$. Using again Inequality~\eqref{eq:logbino} and $d\le c$ and $a\le c$, we obtain 
 $$
 \log( \vert \mathcal{T}_{a,d} \vert) \le (d-1) (1 + \log(a)) \le c + d\log(c),
 $$
 and finally 
 \begin{align*}
\log(\Nc_c(\Mc_\star))& \le  2a (c + d  \log(c)  ).
\end{align*}

\subsection{Proof of \Cref{prop:N_C_sparse}}\label{sec:proof:prop:N_C_sparse}

For given tree $T$, ranks $r \in \Nbb^{\vert T \vert}$ and feature space $\Hc_N$, $N\in \Nbb^d$, we consider sparse tensor networks with arbitrary sparsity pattern $\Lambda \subset \mathcal{L}_{T,N,r}  :=  \times_{\alpha \in T} K_\alpha$,
where $K_\alpha = \{1,\hdots,r_\alpha\} \times  ( \times_{\beta \in S(\alpha)} \{1,\hdots,r_\beta\})$ for $\alpha \in \Ic(T)$, and $K_\alpha = \{1,\hdots,r_\alpha\} \times \{1,\hdots,N_\alpha\}$ for $\alpha\in \Lc(T)$. The sparse representation complexity of a sparse tensor network is given by $C(T,r,\Hc_N,\Lambda) = \sum_{\alpha \in T} \vert \Lambda_\alpha\vert$.   
We recall that  for models of complexity $c$, we restrict the dimensions $N$ of features spaces to be less than a certain increasing function $g(c)$ of the complexity. 
The collections of models of sparse tensor networks have complexities 
\begin{align*}
&\Nc_c(\Mc_{\Hc_N,T}) = \vert \{r\in \Nbb^{\vert T\vert} , \Lambda \subset \mathcal{L}_{T,N,r} : C(T,r,\Hc_N,\Lambda) = c \} \vert ,\\
&\Nc_c(\Mc_{T}) = \vert \{r\in \Nbb^{\vert T\vert} ,  N\in \Nbb^{d} , \Lambda \subset \mathcal{L}_{T,N,r}: C(T,r,\Hc_N,\Lambda)= c \text{ and } N \le g(c)  \} \vert, \\
&\Nc_c(\Mc_\star) = \vert \{T \in \mathcal{T}_{a,d} , r\in \Nbb^{\vert T\vert} ,  N\in \Nbb^{d} , \Lambda \subset \mathcal{L}_{T,N,r}:C(T,r,\Hc_N,\Lambda)= c \text{ and } N \le g(c)  \} \vert, 
\end{align*}
where $\mathcal{T}_{a,d}$ denotes the collection of trees with arity $a$ (or $a$-ary trees). It is clear that 
$\Nc_c(\Mc_{\Hc,T})\le \Nc_c(\Mc_{T}) \le \Nc_c(\Mc_{\star})$.
First, we note that  
\begin{align*}
\Nc_c(\Mc_{T}) \le \sum_{q = (q_\alpha)_{\alpha\in T}}\vert \{ r \in \Nbb^{\vert T\vert} , N \in \Nbb^d , \Lambda \in \mathcal{L}_{T,N,r} : &\\
\vert \Lambda_\alpha \vert = q_\alpha , \alpha \in T, \text{and } N\le g(c)\} \vert&,
\end{align*}
where the sum is taken over all tuples $(q_\alpha)_{\alpha \in T} \in \Nbb^{\vert T \vert} $ such that 
$ \sum_{\alpha\in T} q_\alpha = c$. Then 
$$
\Nc_c(\Mc_{T}) \le \sum_{q = (q_\alpha)_{\alpha\in T}} \sum_{ \substack{r \in \Nbb^{\vert T\vert} \\ r \le q}}
 \sum_{ \substack{N \in \Nbb^d \\ N \le g(c)}} N_{q,T,N,r}
 $$
with 
\begin{align*}
N_{q,T,N,r}& =  \vert \{ \Lambda \in \mathcal{L}_{T,N,r} : \vert \Lambda_\alpha \vert = q_\alpha , \alpha \in T\} \vert,
\\
&= \prod_{\alpha \in \Ic(T)} \binom{r_\alpha \prod_{\beta \in S(\alpha) } r_\beta}{q_\alpha}
\prod_{\alpha \in \Lc(T)} \binom{r_\alpha N_\alpha}{q_\alpha}.
\end{align*}
For $r\le q$, noting that $r_\alpha \le \vert \Lambda_\alpha \vert = q_\alpha$ for all $\alpha\in T$, we obtain 
$$
N_{q,T,N,r} \le \prod_{\alpha \in \Ic(T)} \binom{q_\alpha \prod_{\beta \in S(\alpha)} q_\beta}{q_\alpha}
\prod_{\alpha \in \Lc(T)} \binom{q_\alpha N_\alpha}{q_\alpha}, 
$$
and then using Inequality~\eqref{eq:logbino},  
$$\log(N_{q,T,N,r}) \le \sum_{\alpha\in \Ic(T)} q_\alpha (1+\log(\prod_{\beta \in S(\alpha)} q_\beta )) + 
\sum_{\alpha\in \Lc(T)} q_\alpha (1+\log(N_\alpha)).
$$  
Then with $N_\alpha \le g(c)$ and $ \sum_{\alpha \in T} q_\alpha = c$, we obtain 
$$\log(N_{q,T,N,r}) \le c (1+ a \log(c) )) + c (1 + g(c)).
$$ 
Noting that the number of tuples $q$ such that $\sum_{\alpha} q_\alpha = c$ is bounded  by 
$\binom{c+\vert T\vert }{c} \le e^{c (1+\log(2))}$, 
that the number of tuples $r$ such that $r\le q$ is equal to $\prod_{\alpha \in T} q_\alpha \le (\frac{c}{\vert T\vert} )^{\vert T \vert}$ and that the number of tuples $n $ less than $g(c)$ is equal to $g(c)^d$, we obtain 
\begin{align*}
\log(\Nc_c(\Mc_T) ) \le c(1+\log(2)) + \vert T \vert \log(c) + d \log(g(c))&\\ + c (1+ a \log(c) )) + c (1 + \log(g(c)))&.
\end{align*}
Then noting that $\vert T \vert \le 2d$, we obtain 
$$
\log(\Nc_c(\Mc_T) ) \le 4 c  + 2d \log(c) +  a c \log(c) + 2 c \log(g(c)). 
$$ 
For the collection $\Mc_\star$, we follow the proof of 
\Cref{prop:N_C} and deduce that 
\begin{align*}\log(\Nc_c(\Mc_\star) )& \le \log(\vert \mathcal{T}_{a,d} \vert) +  \max_{T \in \mathcal{T}_{a,d}} 
\log(\Nc_c(\Mc_T) )
\\&\le (d-1) (1+\log(a)) + 4 c  + 2d \log(c) +  a c \log(c) + 2 c  \log(g(c))
\\&\le 4c + 4d \log(c) + ac \log(c) + 2c \log(g(c))
\\&\le 5ac\log(c) + 2c \log(g(c)).
\end{align*}
\section{Proofs of \Cref{sec:leastsquares}}

\subsection{Proof of Proposition~\ref{prop:riskboundTalagrand}}\label{sec:proofriskboundTalagrand}

The proof follows the presentation of \cite{koltchinskii2011oracle}.  The least-squares contrast $\gamma$ corresponds either to the  regression contrast or the density estimation contrast.  Under the assumptions of the proposition, in  both frameworks the oracle   function satisfies  $\Vert f^\star  \Vert_{\infty,\mu} \leq R$.

\medskip

$\bullet$  We first prove the proposition in the case where $R=1$,  by assuming for the moment that 
 $ M  = M_1 = M^T_r(\Hc)_1  $. For the regression framework, it is also assumed for  the moment that  $| Y|  \le 1 $ almost surely. Note that we also have $\Vert f^\star  \Vert_{\infty,\mu} \leq 1$.
 
\medskip

$\bullet$  For the least-squares regression contrast (see  Example \ref{ex:bounded regression}), we have $\gamma(f,Z) = | Y - f(X) | ^2$. For all $f\in {M_1}$, it gives    $ \gamma(f,Z)  \le     2  ( |  Y| ^2 +  \Vert f \Vert_\infty^2) $ almost surely, so that 
$0 \le \gamma(f,Z) \le B$ almost surely, with $B = 4$. The distribution of the random  variable $X$ is denoted $\mu$. Then, almost surely,
  \begin{align}
 \Ebb( (\gamma(f,Z) - &\gamma(f^\star,Z))^2) =
  \Ebb \left[ \left( f^\star(X)-f(X) \right)  \left( 2Y - f(X) - f^\star(X) \right) \right] ^2  \notag \\
 &=\Ebb  \left[  \left(  f^\star(X)-f(X)\right)  \left( 2(Y-f^\star(X)) + f^\star(X) - f(X) \right)  \right]  ^2  \notag  \\
 &= \Ebb  \left[ \left(  f^\star(X)-f(X)\right)  \left( 2(Y-f^\star(X) \right)\right] ^2  + \Ebb \left[  f^\star(X)-f(X) \right] ^4   \notag
 \\ 
 &\le (4   |   Y-f^\star(X) |  ^2 +   \Vert f^\star -f \Vert_{\infty,\mu}^2) \Vert f - f^\star \Vert_{ 2,\mu}^2 \notag  \\
 &\le 8 \Vert f - f^\star \Vert_{ 2,\mu}^2 \notag = 2 B \Vert f - f^\star \Vert_{ 2,\mu}^2, \label{eq:link} 
 \end{align}
 where the last inequality has been obtained using $  |  Y - f^\star(X)  |  =  | Y - \Ebb(Y \vert X)  |   \le 1  $  almost surely.
 Let $ \gamma_1 = \gamma /B$. We have $0 \le  \gamma_1 \le 1$ and the normalized excess risk satisfies
  \begin{align*}
 \mathcal E_1 (f) &:=  \Ebb \left[ \gamma_1(f,Z) -  \gamma_1(f^\star,Z) \right]   = \frac{1}{B} \Vert f - f^\star \Vert _{ 2,\mu}^2 = \frac{1}{B} \mathcal E  (f) 
  \end{align*}
and
 \begin{align*}
 \Ebb( \left[  \gamma_1(f,Z) -  \gamma_1(f^\star,Z) \right]^2)   \le  D  \Vert f - f^\star \Vert _{ 2,\mu}^2
 \end{align*}
with $D = \frac 2 B = \frac 12$.

\medskip

\noindent $\bullet$ We now consider the density estimation framework with $\gamma(f,x) = \Vert f \Vert_{2,\mu}^2 - 2f(x) $. According to Example \ref{dens_estim}, $\vert \gamma(f,X) \vert \le B  = \mu(\Xc)  + 2 $. The excess risk satisfies $ \mathcal E (f) =  \Vert f^\star - f \Vert_{ 2,\mu}^2$
and
 \begin{align*}
 \Ebb(& \left[ \gamma(f,Z) - \gamma(f^\star,Z) \right])^2 = \Ebb( \left[ \Vert f \Vert_{ 2,\mu}^2 - \Vert f^\star \Vert^2_{ 2,\mu}  + 2 (f^\star(X) - f(X)) \right]^2) \\
 & \le   (\Vert f \Vert_{ 2,\mu}^2 - \Vert f^\star \Vert^2_{ 2,\mu})^2 + 4 (\Vert f \Vert_{ 2,\mu}^2 - \Vert f^\star \Vert^2_{ 2,\mu}) \langle f^\star - f ,f^\star \rangle_{ 2,\mu} +  4  \Vert f - f^\star \Vert_{ 2,\mu}^2 \\
 &= (\Vert f \Vert_{ 2,\mu}^2 - \Vert f^\star \Vert^2_{ 2,\mu}) ( \Vert f \Vert_{ 2,\mu}^2 - \Vert f^\star \Vert^2_{ 2,\mu} + 4\langle f^\star - f  ,f^\star \rangle_{ 2,\mu} ) + 4  \Vert f - f^\star \Vert_{ 2,\mu}^2\\
 &=\langle f-f^\star, f + f^\star \rangle_{ 2,\mu} \langle f-f^\star , f - 3f^\star \rangle_{ 2,\mu} + 4  \Vert f - f^\star \Vert_2^2 \\
 &=\langle f-f^\star, f + f^\star \rangle_{ 2,\mu} \Vert f - f^\star \Vert_2^2   -  \langle f-f^\star, f + f^\star \rangle_{ 2,\mu} \langle f-f^\star, 2 f^\star \rangle_{ 2,\mu}  \\ 
  & \hskip 1cm    + 4 \Vert f - f^\star \Vert_{ 2,\mu}^2  .
  \end{align*}
  We have $\langle f-f^\star, f + f^\star \rangle_{ 2,\mu} \le \Vert f \Vert_{2,\mu}^2 \le \mu(\Xc)  $,  
  $\Vert f^\star \Vert_{1,\mu} =1 \le  \mu(\Xc)^{1/2} $  and $\Vert f^\star \Vert_{ 2,\mu}^2 \le \Vert f^\star \Vert_{ \infty,\mu} \Vert f^\star \Vert_{1,\mu} \le 1  $. Then
\begin{align*}
 \Ebb( (\gamma(f,Z) - \gamma(f^\star,Z))^2) &\le  (\mu(\Xc) + 2 \Vert f + f^\star \Vert_{ 2,\mu} \Vert f^\star \Vert_{ 2,\mu} + 4) \Vert f - f^\star \Vert_{ 2,\mu}^2  \\
 &\le (\mu(\Xc) + 2 \mu(\Xc)   + 2 + 4) \Vert f - f^\star \Vert_{ 2,\mu}^2  \\
 &= 3(\mu(\Xc)  + 2)\Vert f - f^\star \Vert_{ 2,\mu}^2 \\
 &= 3B \Vert f - f^\star \Vert_{ 2,\mu}^2 .
 \end{align*}
Let $ \gamma_1 =  \frac1 {2B}(\gamma + B) $. Then $0 \leq \gamma_1(f,X) \le 1$ almost surely for any $f \in M_1$. Moreover, 
\begin{align*}
\mathcal E_1 (f) &  :=  \Ebb \left[ \gamma_1(f,Z) -  \gamma_1(f^\star,Z) \right]  =  \frac{1}{2B} \mathcal E  (f) 
\end{align*}
and 
\begin{align*}
 \Ebb(( \gamma_1(f,Z) -  \gamma_1(f^\star,Z))^2)  \le  D  \Vert f - f^\star \Vert_{ 2,\mu}^2  
\end{align*}
with $D = \frac{3}{4B}   \le \frac{3}{12} $, where we have used $\mu(\Xc) \ge 1.$

\medskip

\noindent $\bullet$ For $\delta >0$, we introduce 
$$
 \omega_n( \delta) = \omega_n( {M_1}, f^\star, \delta) = 
  \mathbb E \sup_{f \in {M_1} \, | \, \| f - f^{M_1} \|_{ 2,\mu} ^2 \leq \delta     / D}  \left| 
\frac{1}{n} \sum_{i=1}^n \gamma_1(f,Z_i) - \Ebb(\gamma_1(f,Z))   
   \right|
 $$
 Following \cite{koltchinskii2011oracle} (Section 4.1 p.57), we introduce the sharp transformation $\sharp$ of the function $\omega$:
$$  \omega_n^\sharp ( \varepsilon) =  \inf \left\{\delta >0 \,  : \, \sup_{\sigma \geq \delta} \frac{ \omega_n( \sigma)}{\sigma} \leq \varepsilon \right\} . $$
According to Proposition 4.1 in \cite{koltchinskii2011oracle}, there exist absolute constants $\kappa_1$ and $\mathcal A$ such  that for any $\varepsilon \in (0,1]$ and any $t >0$, with probability at least $1 - \mathcal A \exp(-t)$, 
\begin{equation}  \label{Prop4dot1}
\mathcal E_1 (\hat f_n^{M_1} )   \leq (1+ \varepsilon) \mathcal E_1 (f^{M_1})  + \frac 1 D   \omega_n^\sharp \left( \frac{\varepsilon}{\kappa_1 D} \right) 
 +  \frac{\kappa_1 D}{\varepsilon}  \frac t n .
\end{equation}
The sharp transformation is monotonic: if $\Psi_1 \leq \Psi_2$ then $\Psi_1^\sharp \leq \Psi_2^\sharp$
(see Appendix A.3 in \cite{koltchinskii2011oracle}). Thus  it remains to find an upper bound on the sharp transformation of an upper bound on  $\omega_n$.

\medskip

\noindent $\bullet$ We use standard symmetrization and contraction arguments for Rademacher variables. The Rademacher process indexed by  the class  ${M_1}$ is defined by 
$$ \Rad_n (f) = \frac 1n \sum_{i= 1}^{ n}  \varepsilon_i  f(X_i)  $$ 
where the $\varepsilon_i$'s are i.i.d. Rademacher random variables (that is, $\varepsilon_i$ takes the values $+1$ and  $-1$ with probability 1/2 each) independent of the $X_i$'s.  By the symmetrization inequality (see for instance Theorem 2.1 in \cite{koltchinskii2011oracle}),
$$
 \omega_n( \delta) \leq 2   \mathbb E \sup_{f \in {M_1} \, | \, \| f - f^{M_1} \|_{ 2,\mu}^2 \leq \delta     / D}  \left|  \Rad_n \left( \gamma(f,\cdot) -   \gamma(f^{M_1},\cdot) \right) \right|.
 $$
We introduce the function 
$$
\Psi_n(\delta)  = \mathbb E \sup_{f \in {M_1} \, | \, \| f - f^{M_1} \|_{ 2,\mu}^2 \leq \delta}  \left|  \Rad_n ( f - f^{M_1}) \right| .
$$
For bounded regression, using the contraction Lemma  with Lipschitz constant equal to $2$ (see for instance Theorem 2.3 in  \cite{koltchinskii2011oracle}), 
\begin{eqnarray*}
 \omega_n( \delta ) 
 & \leq &   8 \Psi_n(\delta/ D).
 \end{eqnarray*}
In the density estimation setting, we have  $\gamma(f,X) = \|f \|_{ 2,\mu}^2 - 2 f(X) $ and since the fluctuations of a constant function are obviously zero, we obtain 
\begin{eqnarray}
 \omega_n( \delta) &\leq &   4  \mathbb E \sup_{f \in {M_1} \, | \, \| f - f^{M_1} \|_{ 2,\mu}^2 \leq \delta     / D}  \left| \Rad_n ( f - f^{M_1} )  \right| \notag \\
 & \leq &    4 \Psi_n(\delta/ D)  \label{omega-Psi}.
\end{eqnarray}

\medskip 
\modif{
\noindent $\bullet$ We now introduce the subset of the  $L^2$ ball centered at $f^{M_1}$ 
$$ {M_1}(\delta,f^\star) =  \{f - f^{M_1} \, :  \,  f \in {M_1}  \, ,   \| f - f^{M_1} \|^2_{ 2,\mu} \leq \delta \} .$$
 In the density estimation setting, the distribution of the $X_i$'s is $\eta$ and  the empirical measure is denoted by $\eta_n$. We also denote by $\eta_n$ the empirical measure in the regression setting (by taking $\eta = \mu$).  
According to Proposition~\ref{prop:metric-entropy}, 
\begin{eqnarray*}
H\left(\varepsilon , {M_1}(\delta,f^\star) , \Vert \cdot\Vert_{2,\eta_n}\right)  &\leq  & 
  H\left(\varepsilon , \{ f \in {M_1}  \, :   \| f  \|^2_{ 2,\mu} \leq \delta \}, \Vert \cdot\Vert_{2,\eta_n}\right) \\
 &\leq  &   H\left(\varepsilon ,  M_1  , \Vert \cdot\Vert_{2,\eta_n}\right) \\
&\leq  & C_M  \log \left(  \frac{3 |T| L_{2,\eta_n}}{\varepsilon} \right) { \mathbb{1}_{\varepsilon \leq   2}}  \quad \textrm{$\eta^{\otimes n}$-almost surely,}
\end{eqnarray*}
where $L_{2,\eta_n}$ is defined by \eqref{Lp} for the measure $\eta_n$ and for $p=2$. It can be easily checked from \eqref{Lp}   that $ L_{2,\eta_n} \leq L_{\infty,\eta_n} $.  Next, the inequalities  $L_{\infty,\eta_n} \leq L_{\infty,\eta} \leq L_{\infty,\mu} $   hold $\eta^{\otimes n}$ -almost surely  in both settings. Moreover,  $L_{\infty,\mu} \leq 1$ according to Proposition~\ref{prop:Lp-bound}.
}
Next, the metric entropy of ${M_1}(\delta,f^\star) $ can be upper bounded  $\eta^{\otimes n}$ -almost surely as follows:
\begin{eqnarray*}
 H\left(\varepsilon , {M_1}(\delta,f^\star) , \Vert \cdot\Vert_{2,\eta_n}\right)  & \leq  & C_M \left[  \log \left(  \frac{ 4e}{\varepsilon} \right)    + \log^+ \left( \frac{3 |T|  }{4e}  \right)  \right] \mathbb{1}_{\varepsilon \leq   4}  \\
 &\leq  & C_M \left[ 1 + \log^+ \left( \frac{3 |T|     }{4e}  \right)  \right]  \log \left(  \frac{4e}{\varepsilon} \right)   \mathbb{1}_{\varepsilon \leq   4}  \\
  &\leq  & C_M \, b_T  \,     h \left(  \frac 2 \varepsilon \right)  
\end{eqnarray*}
with 
$  b_T =   1 + \log^+ \left( \frac{3 |T| }{4e}  \right)    $  and
$ h(u) :=  \log   \left(   2 e  u   \right) \mathbb{1}_{u \geq \frac 12} $. We are now in position to apply  Theorem~\ref{Lemma-Theorem3-12-Kolch}, which is given   at the end of this section. Note that the constant function $F=2 $ is an envelope for $ {M_1}(\delta,f^\star)$ and $\| F \|_{2,\eta_n}   = 2 $. We can take  $\sigma ^2 = \delta$ in Theorem~\ref{Lemma-Theorem3-12-Kolch} because $ \mathbb E_{\eta} ( g (X)^2) \leq \delta$ for $g \in  {M_1}(\delta,f^\star)$. Thus, there exists an absolute constant $\kappa_2 >0 $ such that
\begin{eqnarray*} 
\Psi_n(\delta)  &\leq & \kappa_2   
\left[   
\sqrt{  \frac \delta n C_M  b_T  h \left(  \frac{2}{\sqrt \delta}\right) }  
 \vee
 \left(  \frac{2}{n}       C_M  b_T h \left(  \frac{2}{\sqrt \delta} \right)\right)   
  \right]  \label{eq:Theorem3-12}.
   \end{eqnarray*}    
For regression, it can be easily checked that  (see also  Example~3~p.80 in  \cite{koltchinskii2011oracle}) 
$$ \Psi_{n} ^\sharp  (\varepsilon) \leq       \kappa_2   \frac{  C_M  b_T  } {\varepsilon^2 n} \log \left( \frac{  16 e^2   \varepsilon ^2  }{\kappa_2 C_M  b_T}\right) . $$
Similar calculations hold for  density estimation.
Together with inequalities  \eqref{Prop4dot1} and  \eqref{omega-Psi}, and according to the properties of the sharp transformation (see  Appendix A.3 in \cite{koltchinskii2011oracle}),  it gives that with probability at least $1 - \mathcal A \exp(-t)$, 
\begin{eqnarray}
\mathcal E_1 (\hat f_n^{M_1} )   &\leq & (1+ \varepsilon) \mathcal E_1 (f^{M_1})  + \frac 1 D  
  \left(  
   8  
     \Psi_n \left( \frac{\cdot}{D}
      \right)
      \right)  ^\sharp
  \left( \frac{\varepsilon}{\kappa_1 D} \right) 
 +  \frac{\kappa_1 D}{\varepsilon}  \frac t n \notag  \\
  &\leq & (1+ \varepsilon) \mathcal E_1 (f^{M_1})  + 
  \left(  8 \Psi_n \right)  ^\sharp
  \left( \frac{\varepsilon}{\kappa_1  } \right) 
 +  \frac{\kappa_1 D}{\varepsilon}  \frac t n  \notag \\
  &\leq & (1+ \varepsilon) \mathcal E_1 (f^{M_1})  + 
  \kappa_3   \frac{b_T  C_M     } {\varepsilon^2 n} \log \left( \frac{  \kappa_4 \varepsilon ^2  }{ b_T C_M  }\right) 
 +  \frac{\kappa_1 D}{\varepsilon}  \frac t n,   \notag 
\end{eqnarray}   
 where $ \kappa_3$ and $ \kappa_4$ are absolute constants. This completes the proof for $R=1$, by rewriting the risk bound for the excess risk $\mathcal E = B \mathcal E_1$.

\medskip

\noindent $\bullet$   We now consider the more general situation  where $ M = M^T_r(\Hc)_R  $ with $R \geq 1$.
We first consider regression.  We   assume that $\vert Y\vert  \le R $ almost surely.  
 Let $f^\star$, $f^M$ and  $\hat f^M$ defined as in Section~\ref{sec:general_risk_bounds} for the observations $Z_1, \dots, Z_n$. We consider the least squares regression problem for the normalized data $ (X_1,Y_1/R), \dots,(X_n,Y_n/R)  $ with the functional set $M_1$. For this problem the oracle $f^\star_1$  satisfies  $f^\star_1 = f^\star /R$, the best approximation  $f^{M_1}$ on $M_1$ satisfies $f^{M_1} = f^M/R$ and the least squares estimator $\hat f^{M_1}$ also satisfies  $\hat f^{M_1} = \hat f^M/R$. The risk bound  \eqref{riskbound-leastsquares} is valid for the normalized data (with $R=1$) and it directly gives \eqref{riskbound-leastsquares} for $R \geq 1$. The same arguments apply for proving the risk bound  in the density estimation case.

\subsection{An adaptation of Theorem 3.12  in  \cite{koltchinskii2011oracle}} 

We consider the same framework as in \cite{koltchinskii2011oracle}.  We observe $X_1,\dots,X_n$ according to  the distribution $\eta$ and let $\eta_n$ be the empirical measure.  Let $\mathcal F$ be a function space. Assume that the functions in $\mathcal F$ are uniformly bounded by a constant $U$ and let $F \leq U$ denote a measurable envelope  
of $\mathcal F$. We assume that $\sigma^2$ is a number such that 
$$ \sup_{f \in \mathcal F} \mathbb E_\eta  f^2  \leq \sigma ^2 \leq \| F \|_{2, \eta} .$$
 Let $h : [0, \infty) \mapsto [0,\infty)$ be a regularly varying function of exponent $0 \leq \alpha < 2$, strictly increasing for $u \geq 1/2$ and such that $h(u) = 0$ for $0 \leq u < 1/2$.

The next result is an adaptation of Theorem 3.12  in  \cite{koltchinskii2011oracle} which provides a better control on the constant $\kappa_h >0$   when multiplying the metric entropy  function by a constant. In particular in this version the constant $\kappa_h >0$  depends only  on $h$ and not on $c$. 
\begin{theorem}[Theorem 3.12   in \cite{koltchinskii2011oracle}]
 \label{Lemma-Theorem3-12-Kolch}
Let $c >0$. If, for all $\varepsilon >0$ and $n\geq 1$,
$$  \log N \left(  \varepsilon,  \mathcal F, \Vert \cdot\Vert_{2,\eta_n} \right) \leq  c  h \left( \frac{\| F \|_{2,\eta_n}}{\varepsilon} \right) \quad \textrm{$\eta^{\otimes n}$-almost surely,}  $$ 
then there exists a constant $\kappa_h >0$ that depends only  on $h$  such that 
$$ \mathbb E \sup_{f \in  \mathcal F} |\Rad_n  (f) | \leq   \kappa_h  \left[ \frac{\sigma }{\sqrt n} \sqrt{ c h \left( \frac{ \| F \|_{2,\eta}}{\sigma} \right)} \vee  \frac U n c h \left( \frac{ \| F \|_{2,\eta}}{\varepsilon}\right)  \right] . $$
\end{theorem}

\begin{proof}
The proof of Theorem 3.12 of  \cite{koltchinskii2011oracle} starts  by applying Theorem 3.11 of  \cite{koltchinskii2011oracle}. As in \cite{koltchinskii2011oracle} we assume without loss of generality that $U = 1$.
In our context it gives
$$ E:=\mathbb E \sup_{f \in  \mathcal F} |\Rad_n  (f) | \leq  C  \sqrt c  n^{-1/2}  \mathbb E \int_{0}^{2 \sigma_n} \sqrt{ h \left( \frac{\| F \|_{ 2,\eta_n}}{\varepsilon} \right) }  d \varepsilon$$
where $\sigma_n =\sup_{f \in  \mathcal F} \sum_{i=1}^n f(X_i)^2  $  and where $C$ is an universal numerical constant. By following  the lines of the proof of \cite{koltchinskii2011oracle}, we find that $E$ satisfies the following inequation 
$$ E \leq \sqrt  c  \kappa_{h,1}   n ^{-1} + \sqrt  c \kappa_{h,2} n^{-1/2} \sigma \sqrt{ h \left( \frac{\| F \|_{2, \eta } }{\sigma} \right) } + \sqrt  c  \kappa_{h,3} n^{-1/2}  \sqrt E  \sqrt{ h \left( \frac{\| F \|_{2, \eta }}{\sigma} \right) }   $$
where $\kappa_{h,1}$,  $\kappa_{h,2}$ and  $\kappa_{h,3}$ are positive numerical constants which only depends on the function $h$ (see the proof of Koltchinskii for the expression of these three constants). Solving this inequation completes the proof.
\end{proof}

\subsection{Proof of   \Cref{theo:fast_rates}}\label{sec:proof:theo:fast_rates}

The proof is adapted  from Theorem 6.5 in \cite{koltchinskii2011oracle}, which corresponds to an alternative statement of Theorem 8.5 in \cite{Massart:07}. We follow the lines of Section 6.3 in \cite{koltchinskii2011oracle} (p.107-108).

We first consider the case $R=1$ and we consider  the normalized contrast $\gamma_1$ and  the normalized risk $\mathcal E_1$ as for the proof of Proposition~\ref{prop:riskboundTalagrand}. We have shown that 
 \begin{align*}
 \Ebb \left[  \gamma_1(f,Z) -  \gamma_1(f^\star,Z) \right]^2   \le  D  \Vert f - f^\star \Vert^2_2 
 \end{align*}
where  $D$ does not depend on the model $M_m$.
Next, it has also been shown in the proof of Proposition~\ref{prop:riskboundTalagrand}, that for $\varepsilon \in (0,1]$,
 $$ 
 \omega_n^\sharp (\varepsilon) \leq  
 \kappa      \frac{b_m  C_m } {n \varepsilon^2 } \log^+  \left( \frac{ n  \varepsilon ^2  }{ b_m C_m} \right) 
$$
with $ b_m  =  1 + \log^+ \left( \frac{3 |T_m|  }{4e}  \right)$ and where 
 $\kappa$ is an absolute constant. We consider the penalized criterion \eqref{critpen} with a penalty of the form 
$$
\pen(m) =   \kappa_1  \frac{b_m C_m}{n \varepsilon^2}  \log^+  \frac{n \varepsilon^2}{b_m C_m}   +  \kappa_2  \frac{w_m}{n \varepsilon}  ,
$$
where  $w_m = \bar w C_m +  \log  (\Nc_{C_m})$.
Theorem 6.5  of \cite{koltchinskii2011oracle} can be applied here  with  $\bar \delta_n^\varepsilon(m)  = \tilde \delta_n^\varepsilon(m)   = \hat \delta_n^\varepsilon(m)  = \kappa  \frac{b_m C_m}{n \varepsilon^2}   {\log^+}  \frac{n \varepsilon^2}{b_m C_m}   +  { K} \frac{w_m+t}{n \varepsilon} $ (and thus $p_m=0$ in the theorem) and we also note that for any $t >0$ the penalty can be rewritten 
 $$
\pen(m) =  K_1 \left[ \frac{b_m C_m}{n \varepsilon^2}   {\log^+}  \frac{n \varepsilon^2}{b_m C_m}   + \frac{w_m+t}{n \varepsilon}\right].
$$
Finally, according to Theorem 6.5 in \cite{koltchinskii2011oracle}, there exist numerical constants $K_1$, $K_2$ and $K_3$ such that for any $t >0$, 
\begin{align*}
 P
\left(
 \mathcal E_1 (\hat f _{\hat m})   \leq  \frac{1+\varepsilon}{1-\varepsilon} 
  \inf_{m \in \mathcal M} \left\{  \mathcal E_1 (\hat f _{\hat m})   + K_2 
{ \pen(m)}
  \right\} 
 \right) 
  \leq K_3  \sum_{ m \in \mathcal M}  \exp(-t -w_m).
\end{align*}
We easily derive the oracle bound \eqref{eq:riskbound-fast} by rewriting it for the contrast $\gamma$ and then by integrating this probability bound with respect to $t$. This bound generalizes to the case $  R \geq 1$ as in the proof of Proposition~\ref{prop:riskboundTalagrand}.

\end{document}